\documentclass[11pt]{amsart}
\usepackage{rlepsf,  
graphicx, amssymb, 
amsmath, amssymb,
psfrag, pinlabel, rotating,
}

\let\oldmarginpar\marginpar
\renewcommand\marginpar[1]{\-\oldmarginpar[\raggedleft\footnotesize #1]%
{\raggedright\footnotesize #1}}

\theoremstyle{plain}
\newtheorem{thm}{Theorem}[section]
\newtheorem{lem}[thm]{Lemma}
\newtheorem{lemma}[thm]{Lemma}
\newtheorem{cor}[thm]{Corollary}
\newtheorem{claim}[thm]{Claim}
\newtheorem{prop}[thm]{Proposition}

\theoremstyle{definition}

\newtheorem{rem}[thm]{Remark}

\newtheorem{Ex}[thm]{Example}

\newtheorem*{Nt}{Notation}

\theoremstyle{remark}

\newcommand{\be}{\begin{enumerate}}
\newcommand{\ee}{\end{enumerate}}

\def\IZ{\mathbf{Z}}

\def\C{\mathbb{C}}
\def\R{\mathbb{R}}
\def\Z{\mathbb{Z}}

\def\dfn#1{{\em #1}}
\def\bdry{\partial}
\def\gcd{{\tt gcd}}

\newcommand{\cP}{\mathcal{P}}
\newcommand{\lutz}{\mbox{\tiny{Lutz}}}

\newcommand{\Id}{\mathit{id}}

\newcommand{\Map}{\mathit{Map}}
\newcommand{\Dehn}{\mathit{Dehn}}
\newcommand{\Stein}{\mathit{Stein}}
\newcommand{\Strong}{\mathit{Strong}}
\newcommand{\Weak}{\mathit{Weak}}
\newcommand{\Veer}{\mathit{Veer}}
\newcommand{\HFH}{\mathit{HFH}}

\title[Cabling and contact structures]{Cabling, contact structures\\ and mapping class monoids}

\author{Kenneth L.\ Baker}
\address{
    Department of Mathematics,
    University of Miami,
   PO Box 249085,
Coral Gables, FL 33124-4250}
\email{k.baker@math.miami.edu}
\urladdr{http://math.miami.edu/\char126 kenken}

\author{John B.\ Etnyre}
\address{
    School of Mathematics,
    Georgia Institute of Technology,
    686 Cherry St.,
    Atlanta, GA  30332-0160}
\email{etnyre@math.gatech.edu}
\urladdr{http://math.gatech.edu/\char126 etnyre}

\author{Jeremy Van Horn-Morris}
\address{American Institute of Mathematics, 360 Portage Ave
Palo Alto, CA  94306-2244 } 
\email{jvanhorn@aimath.org}
\urladdr{http://aimath.org/~jvanhorn}

\begin{document}
\begin{abstract}
In this paper we discuss the change in contact structures as their supporting open book decompositions have their binding components cabled. To facilitate this and applications we define the notion of a rational open book decomposition that generalizes the standard notion of open book decomposition and allows one to more easily study surgeries on transverse knots. As a corollary to our investigation we are able to show there are Stein fillable contact structures supported by open books whose monodromies cannot be written as a product of positive Dehn twists. We also exhibit several monoids in the mapping class group of a surface that have contact geometric significance. 
\end{abstract}

\maketitle
\section{Introduction}

In \cite{Giroux02} Giroux introduced a powerful new tool into contact geometry. Specifically he demonstrated there is a one to one correspondence between contact structures up to isotopy and open book decompositions up to positive stabilization. This Giroux correspondence is the basis for many, if not most, of the advances in contact geometry recently. Moreover, the correspondence opens two central lines of enquiry. The first is to see how properties of a contact structure are reflected in an open book decomposition associated to it, and vice versa. The second is to see how natural constructions on one side of this correspondence affect the other. This paper addresses both these themes. It is primarily focused on studying how changes in an open book decomposition, namely cabling of binding components, affect the supported contact structure. In this study one is confronted with fibered links that are not the bindings of open books. In order to deal with such objects we introduce the notion of a rational open book decomposition and contact structures compatible with them. 

While the heart of the paper involves studying the cabling procedure and general fibered links, there are several  interesting and unexpected corollaries dealing with the first theme mentioned above. The first corollary involves showing that there are Stein fillable open books that are supported by open books whose monodromies cannot be written as products of positive Dehn twists. We construct such an open book using our analysis of the behavior of monodromies under cabling. As a second corollary we again use our study of monodromies under cables to construct Stein cobordisms that can be used to construct geometrically interesting monoids in the mapping class group of a surface. We begin by discussing these corollaries. 


\subsection{Stein fillings and monodromy}
One immediate and obvious effect of Giroux's correspondence is a relationship between contact structures and mapping classes of surface automorphisms. More specifically recall that if $(B,\pi)$ is an open book decomposition of a 3-manifold $M$ that supports a contact structure $\xi$ then one can describe the fibration $\pi:(M\setminus B)\to S^1$ as the mapping torus of a diffeomorphism $\phi:\Sigma\to \Sigma,$ where $\Sigma$ is a fiber of $\pi.$ The map $\phi$ is called the monodromy of $(B,\pi).$ We will frequently denote by $M_{(B,\pi)}$ or $M_{(\Sigma,\phi)}$ the manifold defined by the open book decomposition $(B,\pi)=(\Sigma,\phi)$ and by $\xi_{(B,\pi)}$ or $\xi_{(\Sigma,\phi)}$  the associated contact structure. 

It has long been known, \cite{AkbulutOzbagci01, Giroux02, LoiPiergallini01}, that a contact 3-manifold $(M,\xi)$ is Stein fillable if and only if there is an open book decomposition $(\Sigma,\phi)$ supporting it such that $\phi$ can be written as a composition of right handed Dehn twists along curves in $\Sigma.$ This gives a nice characterization of Stein fillability in terms of monodromies, but it can sometimes be hard to use in practice as one only knows there is some open book decomposition for $\xi$ that has monodromy with the given presentation. So this result does not allow one to check Stein fillability using any open book decomposition for $\xi.$ While it seemed likely for a time that all open books supporting a Stein fillable $\xi$ might have monodromies given as a product of right handed Dehn twists this turns out not to be the case.  

\begin{thm} \label{thm:counterexampletoStein} \label{thm:counterexample}  
The tight contact manifolds $(L(p,p-1),\xi_{std})$ for $p\geq 1$ are Stein fillable but admit compatible open book decompositions whose monodromy cannot be factored as a  product of positive Dehn twists.
\end{thm}

The examples above are all $(2,1)$--cables of genus one open book decompositions, each of which is a Hopf stabilization of an annular open book.  The example $L(1,0)$ refers to $S^3$.  In this case, the open book has as a binding the $(2,1)$--cable of the right handed trefoil.  This theorem will be proven in Section~\ref{applications}. One has the following amusing corollary, previously observed by Melvin-Morton, \cite{melvinmorton}.
\begin{cor} The $(2,1)$--cable of the right handed trefoil is not a Hopf stabilization of the unknot.\qed
\end{cor}
Of course there is necessarily some sequence of Hopf stabilizations and destabilizations (even positive ones) that go from the unknot to the $(2,1)$--cable of the right handed trefoil knot by Giroux correspondence.  Indeed the $(2,2)$--cable of the right handed trefoil knot is a single stabilization of the $(2,1)$--cable as well as a sequence of stabilizations of the unknot. 

By different techniques, Wand has shown that there is an open book supporting a tight contact structure on $S^1 \times S^2 \# S^1 \times S^2$ whose monodromy cannot be factored as a product of positive Dehn twists, \cite{wand}.

\subsection{Stein cobordisms and monoids in the mapping class group}
Shortly after Giroux established his correspondence, observations were made linking geometric properties of the contact structure to the monodromy and certain monoids in the mapping class group of a compact oriented surface with boundary.  (We note that one must study monoids in the mapping class group instead of subgroups, as there are usually serious differences between the geometric properties of a contact structures associated to an open book with a given monodromy diffeomorphism and the open book whose monodromy is the inverse of this diffeomorphism.) There are two striking examples, $\Dehn^+(\Sigma)$ and $\Veer^+(\Sigma)$, which are used to detect Stein fillability and tightness of the contact structure, respectively.  The monoid $\Dehn^+(\Sigma)$ is the sub-monoid of the oriented mapping class group $\Map^+(\Sigma)$ generated by positive Dehn twists about curves in $\Sigma,$ and $\Veer^+(\Sigma)$ is the sub-monoid of {\em right veering} diffeomorphisms defined in \cite{HondaKazezMatic07}.

While it is shown in \cite{HondaKazezMatic07} that every open book $(\Sigma,\phi)$ compatible with a tight contact structure has $\phi \in \Veer^+(\Sigma)$, \cite{HondaKazezMatic07} also shows there may be (and for certain surfaces there are) monodromies in $\Veer^+(\Sigma)$ which correspond to overtwisted contact structures.   
Thus tight contact structures are not characterized by having compatible open book decompositions with monodromy in $\Veer^+(\Sigma)$.
Similarly Theorem~\ref{thm:counterexampletoStein} shows that Stein fillable contact structures are not characterized by having compatible open book decompositions with monodromy in $\Dehn^+(\Sigma)$.  
However, we show that, at least in the Stein fillable case (and much more generally) the set of monodromies compatible with Stein fillable contact structures (for example) forms a closed monoid in the mapping class group $\Map^+(\Sigma)$. We begin by observing the following result whose proof can be found in Section~\ref{applications}.

\begin{thm} \label{constructbordism}
Let $\phi_1$ and $\phi_2$ be two elements of $\Map^+(\Sigma).$ There is a Stein cobordism $W$ from 
$(M_{(\Sigma,\phi_1)},\xi_{(\Sigma,\phi_1)}) \sqcup (M_{(\Sigma,\phi_2)},\xi_{(\Sigma,\phi_2)}) $ to $(M_{(\Sigma,\phi_1\circ \phi_2)},\xi_{(\Sigma,\phi_1\circ\phi_2)}) $.
\end{thm}
Recall $W$ will be a Stein cobordism from $(M,\xi)$ to $(M',\xi')$ if it is a compact complex manifold and there is a strictly pluri-subharmonic function $\psi:W\to [0,1]$  such that $M=\psi^{-1}(0)$ and $M'=\psi^{-1}(1).$ In particular $W$ can be endowed with a symplectic form such that $(M,\xi)$ is a concave boundary component of $W$ and $(M',\xi')$ is a convex boundary component. 
Upon announcing this theorem John Baldwin noticed that his joint paper with Plamenevskaya \cite{BaldwinPlamenevskayaPre} contains an implicit proof of this result. He made this explicit in \cite{BaldwinPre??} and in addition observes the following corollaries of this result. 

Eliashberg proved that any Stein manifold/cobordism can be built by attaching to a piece of the symplectization of a contact manifold a collection of 4-dimensional 1-handles and 2-handles along Legendrian knots with framings one less than the contact framing, see \cite{Eliashberg90a}. As the attachment of 1-handles corresponds to (possibly self) connected sums (even in the contact category) and attaching 2-handles as above corresponds to Legendrian surgery we have the following immediate corollary of Theorem~\ref{constructbordism}.
\begin{thm} \label{thm:mainmonoid}  
Let $\mathcal{P}$ be any property of contact structures which is preserved under Legendrian surgery and (possibly self) connected sum. 
Let $\Map^\cP(\Sigma) \subset \Map^+(\Sigma)$ be the set of monodromies $\phi$ which give open book decompositions compatible with contact structures satisfying $\mathcal{P}$.  Then $\Map^\cP(\Sigma)$ is closed under composition. Thus if the identity map on $\Sigma$ is in $\Map^\cP(\Sigma)$  then $\Map^\cP(\Sigma)$  is a monoid. \qed
\end{thm}

Using results from \cite{Eliashberg90a,EtnyreHonda02a, OzsvathSzabo05a} we have the following corollary. 
\begin{cor}  For each of the properties $\cP$ listed below, the set of monodromies $\phi$ of open books $(\Sigma, \phi)$ compatible with contact structures satisfying $\cP$ forms a monoid in the mapping class group $\Map^+(\Sigma)$:
\be 
\item non-vanishing Heegaard-Floer invariant,
\item Weakly fillable,
\item Strongly fillable,
\item Stein fillable.\qed
\ee
\end{cor}

Denote by $\HFH(\Sigma)$, $\Weak(\Sigma)$, $\Strong(\Sigma)$ and $\Stein(\Sigma)$, the corresponding monoids in $\Map^+(\Sigma)$. 

That the first category forms a monoid was first observed in \cite{Baldwin08} using a comultiplication map in Heegaard Floer homology. (We note that the comultiplication map can be defined using our Theorem~\ref{constructbordism}).  The other three monoids were previously unknown.  It has long been known, {\em cf.\ }\cite{AkbulutOzbagci01,LoiPiergallini01}, that monodromies in $\Dehn^+(\Sigma)$, the monoid generated by all right-handed Dehn twists on $\Sigma$, give rise to Stein fillable contact structures and so $\Dehn^+(\Sigma) \subset Stein(\Sigma).$   Work of Honda-Kazez-Mati\'{c} \cite{HondaKazezMatic08} gives strong results on the fillability of $\xi$ when the monodromy is pseudo-Anosov.  Wendl in \cite{WendlPre??}  has very interesting results showing that $\Strong(\Sigma) = \Stein(\Sigma) = \Dehn^+(\Sigma)$ when $\Sigma$ is a planar surface.   We have the following sequences of inclusions:
$$
\begin{array}{cccccccccc}
& & & &     && \Weak(\Sigma) &&\\
& & & &     &\text{\begin{turn}{45} $\subsetneq$\end{turn}}& & \text{\begin{turn}{-45} $\subsetneq$\end{turn}}&\\
\Dehn^+(\Sigma) &\subsetneq& \Stein(\Sigma) & \subsetneq& \Strong(\Sigma) && \text{\begin{turn}{90} $\not\subset$\end{turn}} \text{\begin{turn}{90} $\not\supset$\end{turn}}
&& \Veer^+(\Sigma).\\
& & & &     & \text{\begin{turn}{-45} $\subsetneq$\end{turn}}& & \text{\begin{turn}{45} $\subsetneq$\end{turn}}&\\
& & & &     && \HFH(\Sigma) &&\\
\end{array}
$$
The first inclusion is discussed above and the fact that it is strict follows from Theorem~\ref{thm:counterexampletoStein}.  The second inclusion is well known and the fact that the inclusion is strict follows from \cite{Ghiggini05}. The inclusion $\Strong(\Sigma)\subset \Weak(\Sigma)$ is obvious and the strictness of the inclusion was first observed in \cite{Eliashberg96}.
It is  known that neither $\HFH(\Sigma)$ nor $\Weak(\Sigma)$ is included in the other, see \cite{GhigginiHondaVanHornMorrisPre, GhigginiLiscaStipsicz07}.  It was shown in \cite{OzsvathSzabo04a} ({\em cf.\ }\cite{Ghiggini06}) that $\Strong(\Sigma) \subset  \HFH(\Sigma)$ and the other two inclusions follow from \cite{HondaKazezMatic07} as it is well known that a weakly fillable contact structure or one with non-vanishing Heegaard Floer invariant is tight. The strictness follows as there are right veering monodromies that support overtwisted contact structures as noted above. 

\begin{rem}
It is unknown whether the set of tight contact structures is closed under Legendrian surgery (although it is closed under connected sum) and hence whether there is a tight monoid, though the above theorem says that 

\smallskip
\begin{center}
\begin{minipage}[h]{3.5in}
\begin{center}
\em tightness is preserved under Legendrian surgery \\
if and only if there is a tight monoid. 
\end{center}
\end{minipage}
\end{center}
\end{rem}
\medskip

\subsection{Rational open books and cabling}

Given a fibered knot $L$ whose fiber is a Seifert surface in a manifold $M$ it is well known (and will be proven below) that for $pq \neq 0$ the link obtained from the knot $L$ by a $(p,q)$--cable, denoted $L_{(p,q)}$, is also fibered.  Thus if $L$ is the connected binding of an open book decomposition of $M$, its cable is too, and then one might ask how their compatible contact structures are related to each other.

 However if $L$ is a fibered link with more than one component, then the $(p,q)$--cable of one component produces a link with fibration whose fibers run along the other components $p$ times rather than once.  This cabled open book is then not an honest open book.   

In Section~\ref{sec:rob} we define the notion of a {\em rational open book decomposition} that generalizes the notion of an open book decomposition. Roughly speaking a rational open book decomposition of a manifold $M$ is a fibered link for which the fiber provides a rational null-homology of the link.   (The similar concept of a ``nicely fibered'' link has been previously defined by Gay \cite{Gay02a} when studying symplectic 2-handles.)  When we want to restrict to ordinary open books we use the adjective ``integral''.  Generically, cabling one binding component of a rational open book (or just an integral open book as above) produces another rational open book. These objects also naturally show up when studying surgery problems as we will see below.

Throughout this paper when $L$ is the binding of a (rational) open book decomposition and we discuss cabling a component of the binding $L$ we really mean cabling the open book decomposition. This is an important distinction as a given link can be the boundary of many open book decompositions. Despite this distinction, our abuse of terminology should not cause confusion as when we discuss a link $L$ as a binding of an open book decomposition we will always have a fixed open book decomposition in mind.

  One can define what it means for such an open book decomposition to support a contact structure in direct analogy to what happens in the usual case. The main observation now is the following.
\begin{thm}\label{thm:support}
Let $(L,\pi)$ be any rational open book decomposition of $M$. Then there exists a unique contact structure $\xi_{(L,\pi)}$ that is supported by $(L,\pi).$
\end{thm}

With the notion of rational open book in hand we return to cabling binding components of open book decompositions. Before we can state our main theorem we briefly make a couple of definitions. 

Given an oriented knot $K$ let $N$ be a tubular neighborhood of $N$. Let $\mu$ be a meridional curve on $T=\partial N$ oriented so that it positively links $K$ and $\lambda$ some longitudinal curve on $T$, oriented so that it is isotopic, in $N$, to $K$. An $(p,q)$-curve on $T$ is an embedded curve, or collection of curves, that represent the homology class $p[\lambda]+q[\mu]$ in $H_1(T;\Z)$. (We notice that this convention for naming curves is different from what is commonly used in contact geometry, but the same as is commonly used in 3-manifold topology. See Subsection~\ref{sec:slopeconventions} for a more complete discussion.) We will also use the terminology that $q/p$ is the \dfn{slope} of this curve. The \dfn{$(p,q)$-cable of $K$} is the $(p,q)$-curve on $T$ and it is denoted $K_{(p,q)}$. If $K$ is a component of a link $L$ then the $(p,q)$-cable of $L$ along $K$ is the link obtained by replacing $K$ in $L$ by $K_{(p,q)}$. If $\Sigma$ is a (rational) Seifert surface for $L$ then $\Sigma\cap T$ can be assumed to be a (collection of) embedded curve(s) on $T$ (oriented as $\bdry (\Sigma-N)$) and hence there is some integers $r,s$ with $r>0$ such that $\Sigma\cap T$ is isotopic to the $(r,s)$-curve on $T$. We call $(r,s)$, or $s/r$, the \dfn{Seifert slope} of $\Sigma$ along $K$. We call a $(p,q)$-cable, \dfn{positive}, respectively \dfn{negative}, if $q/p>s/r,$ respectively $q/p<s/r$.   Observe that when the cable $K_{(p,q)} \subset T$ runs in the direction of $K$ (i.e.\ when $p>0$), the cable is positive, respectively negative, if and only if $K_{(p,q)}$ intersects $\Sigma$ positively, respectively negatively.

Suppose  $T$ is a transverse curve in a contact manifold $(M,\xi)$. If $\xi'$ is obtained from $\xi$ by performing a (half) Lutz twist on $T$ then there will be a  knot in the core of the Lutz solid torus that is topologically isotopic to $T$.   Around this core there will be a concentric torus whose characteristic foliation is by meridional curves. On this torus, let $T' = T_{(p,q)}$ be the $(p,q)$--cable of $T$ relative to the framing used when defining the Lutz twist.  We call $T'$ a \dfn{$(p,q)$--Lutz cable} of $T$. See Section~\ref{sec:lutz} for a more complete discussion of Lutz twists and cables. 

To state our theorem we need the notion of exceptional cablings. We briefly describe them here. Given a component $K$ of a fibered link $L$ in $M$ we say there are no exceptional cablings if the the fiber $\Sigma$ of the fibration of $M-L$ defines a longitude for $K$. Otherwise, choose a longitude $\lambda$ for $K$ so that the Seifert slope $s/r$ is between 0 and $-1$. The end points of the shortest path in the Farey tessellation from $-1$ to $s/r$ give the slopes of the exceptional cables of $K$.  (Alternatively one may reinterpret the exceptional cabling slopes as follows:  In the plane $H_1(T,\R) = \langle[\mu],[\lambda]\rangle$ let $C$ be the cone in the second quadrant between the two lines through the origin and each of the points $(-1,1)$ and $(s,r)$ where $r[\lambda]+s[\mu]$ is the Seifert slope.  Then $q/p \neq s/r$ is the slope of an exceptional cable if $(q,p)$ is a lattice point on the boundary of the convex hull of the integral lattice in $C$ minus the origin.  We leave the equivalence of these definitions for the reader.)
For more details and a simple method to compute the exceptional cables see Subsection~\ref{exception-section}. We note a few facts about exceptional cables. 
Any component of a fibered link has a finite number of exceptional cablings, and these are all easily computable from the link.  The only exceptional cabling slope of an integral open book is $-1$.   Also, we define a rational unknot to be a knot whose exterior is a solid torus (fibered by disks); as such, it is a knot in a lens space.

\begin{thm}\label{thm:main-cable}
Let $(L,\pi)$ be a rational open book decomposition supporting the contact structure $\xi$ on $M$. 
Order the components $L_1,\ldots, L_n,$ of $L$ and for each component $L_i$ choose pairs of integers $(p_i,q_i)$ such that the slope $\frac{q_i}{p_i}$ is neither the meridional slope nor the Seifert slope of $L_i$. Assume all the $p_i$ have the same sign and set $(\mathbf{p},\mathbf{q})=((p_1,q_1),\ldots, (p_n,q_n)).$
Then the contact structure $\xi_{(\mathbf{p},\mathbf{q})}$ associated to the $(\mathbf{p},\mathbf{q})$--cable of $(L,\pi)$ is
\begin{enumerate}
\item contactomorphic to $\xi$ if the $p_i$ are positive and all the $(p_i,q_i)$ with $p_i \neq 1$ are positive,
\item contactomorphic to $-\xi$ if the $p_i$ are negative and all the $(p_i,q_i)$ with $p_i \neq -1$ are positive, 
\item virtually overtwisted or overtwisted if any of the $(p_i,q_i)$ with $p_i \neq \pm1$ are negative and $L$ is not a rational unknot having Seifert slope $\frac sr$ with $rq-ps=-1$, and
\item overtwisted if any of the $(p_i,q_i)$ with $p_i \neq \pm1$ are negative and not an exceptional cabling and $L$ is not a rational unknot having Seifert slope $\frac sr$ with $rq-ps=-1$.
\end{enumerate}
Furthermore, in the last case, if $L_{i_1},\ldots,L_{i_k}$ are the components of $L$ for which $(p_i,q_i)$ is negative  then $\xi_{(\mathbf{p},\mathbf{q})}$ is contactomorphic to the contact structure obtained from $\xi$ (respectively $-\xi$) by performing Lutz twists on $L_{i_j}$ (respectively $-L_{i_j}$) followed by a Lutz twist on the $(p_{i_j}, q_{i_j})$--Lutz cable of $L_{i_j}$ (respectively $-L_{i_j}$) if the $p_i$ are positive (respectively negative).
\end{thm}

\begin{rem}
Notice the exceptions for rational unknots in the above theorem.  They are the only bindings of rational open books with disk pages.  This allows them to have non-trivial cables (in particular negative, non-exceptional cables) that are again rational unknots.  See Example~\ref{ex:ratlunknot} (2) and Remark~\ref{rem:ratlunknot}. 
\end{rem}

\begin{rem}
Observe that in the above theorem if $p_i=1$, then the component $L_i$ is effectively not cabled and the Seifert slope of the page on that component remains the same (though the multiplicity with which a page meets that component may increase).   To cable just a subset of the binding components of an open book where all the $p_i$ are positive, simply do $(1,1)$--cables on the components that are to be left unaltered.
\end{rem}

\begin{rem}
We will see in the proof of this theorem that the operation of cabling a binding component of an open book affects the contact structure by removing a standard neighborhood of the binding and replacing it with a solid torus having a possibly different contact structure. When the cabling is positive the replaced contact structure is the same as the original contact structure but  when the cabling is exceptional it is a virtually overtwisted contact structure (leading to the delicate issue  of when gluing two tight contact structures along a compressible torus yields a tight contact structure). When the cabling is sufficiently negative the replaced contact structure is  overtwisted. 
\end{rem}

Some of the results in this theorem regarding when cabling preserves tightness or induces overtwistedness have been obtained by Ishikawa, \cite{ishikawaPre??}.

The statement is much cleaner in the case of integral open book decompositions with connected binding.   In particular, since a pair of integers $(p,q)$ is positive precisely when $pq>0$ and negative when $pq<0$ for integral open books,  we have the following result.  
\begin{cor}\label{cor:integralcable}
Let $(L,\pi)$ be an integral open book decomposition with connected binding supporting $\xi$ on $M$. Let $(p,q), |p| > 1, q\not=0,$ be a pair of integers. Then the contact structure $\xi_{(p,q)}$ supported by the integral open book with binding $L_{(p,q)}$ is
\begin{enumerate}
\item contactomorphic to $\xi$ if $pq>0$ and $p>0$,
\item contactomorphic to $-\xi$ if $pq>0$ and $p<0$, 
\item contactomorphic to $\xi\#\xi_{(1-|p|)(2g+|q|-1)}$ if $pq<0$ and $p>0$ and $L$ is not the unknot with $q=-1$, and
\item contactomorphic to $-(\xi\#\xi_{(1-|p|)(2g+|q|-1)})$ if $pq<0$ and $p<0$ and $L$ is not the unknot with $q=1$
\end{enumerate}
 where $g$ is the genus of the knot $L$ and $\xi_n$ is the overtwisted contact structure on $S^3$ with Hopf invariant $n$.  
 If $p = \pm1$ then $L_{(p,q)} = \pm L$ and $\xi_{(p,q)}$ is contactomorphic to $\pm \xi$.
\end{cor}

\begin{rem} We note that Corollary~\ref{cor:integralcable} recovers a result of Hedden \cite{HeddenCable} when the ambient manifold $M$ is $S^3$.
\end{rem}

\begin{rem}
Notice the exceptions for the unknot in $S^3$ in the above corollary. Because the unknot is the only binding for an integral open book with disk pages, each $(p,\pm1)$--cables of the unknot is still an unknot (for any $p$). This is the only integrally fibered knot with this property.  
\end{rem}

There are negative cables of binding components of open books that support tight contact structures. In addition, when tightness is preserved by a negative cable the contact structure can change, unlike in the case of positive cables. In particular we have the following two results. 

\begin{prop}\label{prop:ratunknots}
Let $(L,\pi)$ be a rational unknot in a lens space $M$.  Then all exceptional cables of $L$ support a tight contact structure on $M$. If $L$ is not the unknot in $S^3$ then the contact structures supported by the cabled knots types do not have to be the same as the one supported by $L$. Moreover, the exact contact structures can be determined. 
\end{prop}

We can also see that tightness can be preserved when negatively cabling a link that is not a rational unknot. 

\begin{prop} \label{prop:tightneg}
There are negative cables of rational open books other than rational unknots which remain tight (in fact, fillable).  In particular, the (2,-1)--cable of a very general family of (3,-1)--open books are Stein fillable.
\end{prop}

Very general here means that there are no restriction on genus, and no restrictions on the monodromy other than it being suitably positive at the boundary.  For a more precise formulation of this result see Section~\ref{exceptionalsurg}.

There are conditions on a fibered link that imply that negative cabling with the exceptional cabling slopes will never yield  tight contact structures (and hence all negative cables yield overtwisted structures). 
\begin{prop}\label{prop:otexceptional}
If $(L,\pi)$ is a rational open book decomposition of $M$ that has a component $L'\subset L$ that is contained in a solid torus $S$ with convex boundary having dividing slope greater than or equal to any longitudinal slope that is non-negative with respect to the page of the open book, then all exceptional cables along $L'$ will support overtwisted contact structures. 

Moreover, any negative $(p,q)$--cabling where $p$ and $q$ are not relatively prime will yield an overtwisted contact structure.
\end{prop}
Notice that for any integral open book, except the unknot in $S^3$, one can always find a solid torus neighborhood of a binding component with convex boundary having dividing slope $0$ with respect to the page framing.  Hence this proposition gives an indication as to why one cannot have exceptional slopes when considering integral open books. 

\smallskip

Rational open book decompositions can be difficult to work with, so in Section~\ref{sec:resolution} we show how to use the above cabling operations to resolve a rational open book decomposition. That is we give a construction that takes a rational open book decomposition and produces an honest open book decomposition that supports the same contact structure.

It is useful to understand the monodromy of a cable in terms of the monodromy of the original fibered link. In particular our corollaries discussed above are based on this. So in Section~\ref{sec:monodromy} we discuss how to compute the monodromy of certain positive ``homogeneous'' cables of open book decompositions. Given an integral open book decomposition with binding $L$ we give an explicit description of the monodromy of the integral open book decomposition obtained from $L$ by $(p,1)$--cabling each binding component of $L$.  From this one can obtain a presentation for the $(p,q)$--cables of $L$ by positive stabilizations.

\subsection{Surgery and open book decompositions}
In Section~\ref{surgeryont} we observe that Dehn surgery on binding components of open books naturally yield induced rational open books. These rational open books may then be resolved to integral open books. Thus we have a procedure for constructing integral open books for manifolds obtained from Dehn surgeries. 

Recall that a surgery on a transversal knot $K$ is called admissible if the surgery coefficient is smaller than the slope of the characteristic foliation on the boundary of a standard neighborhood of a transverse knot.  Gay shows there is a natural contact structure on a manifold obtained from admissible surgery on a transverse knot and, in the case of integral surgeries, there is a symplectic cobordism from the original manifold to the surgered one, \cite{Gay02a}.  This leads to the following result which can be thought of as a generalization of a result of Gay to the case of rational open books. 
\begin{thm}
Let $(L,\pi)$ be an open book decomposition for $(M,\xi)$ and $K$ one of the binding components. The induced open book for any admissible surgery on $K$ that is negative with respect to the framing on $K$ given by a page of $(L,\pi)$ supports the contact structure obtained from the admissible surgery on the surgered manifold. 
\end{thm}
This result, and generalizations to rational open books, follows immediately from Lemmas~\ref{lem:admissible} and~\ref{lem:admissiblesup}.

We also discuss in Section~\ref{surgeryont} how to put any transverse knot in the binding of an open book decomposition so that we can apply the above theorem to construct open books for contact structures obtained via admissible surgeries. 

{\em Acknowledgments:} The authors thank Vincent Colin, Emmanuel Giroux, and Paolo Lisca for useful discussions during the preparation of this paper and Burak Ozbagci and the referees for helpful comments on a first draft of the paper. The first author was partially supported by NSF Career Grant DMS-0239600.
The second author was partially supported by NSF Career Grant (DMS-0239600), FRG-0244663 and DMS-0804820.

\section{Rational open book decompositions and contact structures.}\label{sec:rob}

We begin by establishing some notation for curves on the boundary of a neighborhood of knot.  A standard neighborhood of a knot $K$ is a solid torus $N_K=S^1\times D^2$.  Let $\mu$ be a \dfn{meridian}, the boundary of a meridional disk $\{pt\}\times D^2$;  let $\lambda$ be a \dfn{longitude} or \dfn{framing curve}, that is a curve on $\partial N$ that is isotopic in $N_K$ to the core of the solid torus. We can choose the product structure so that $\lambda$ is $S^1\times\{pt\}$.  

Fix an orientation on $K$.  Orient $\mu$ as the boundary of the meridional disk $\{pt\}\times D^2$ where $\{pt\}\times D^2$ is oriented so that it has positive intersection with $K$.  Orient $\lambda$ so that $\lambda$ and $K$ are isotopic as oriented knots in $N_K$.  Together $([\lambda],[\mu])$ forms a basis for $H_1(\partial N_K; \Z)$. With respect to this longitude-meridian basis, a pair of integers $(p,q)\not=(0,0)$ defines a collection $K_{(p,q)}$ of coherently oriented essential simple closed curves on $\partial N_K$ representing the homology class $p [\lambda] + q [\mu]$.  If $p$ and $q$ are relatively prime then a $(p,q)$--curve is a single curve.  If $p$ and $q$ are not relatively prime then a $(p,q)$--curve is $\gcd(p,q)$  mutually disjoint copies of the $(p/\gcd(p,q),q/\gcd(p,q))$ curve on $\partial N_K$.   

The \dfn{slope} of a $(p,q)$--curve and of its homology class $p [\lambda] + q [\mu]$ is  $\frac{q}{p}$.  
This is chosen so that  the meridian $\mu$ has slope $\infty$, the chosen framing curve $\lambda$ has slope $0$, and every longitude has integral slope.  This choice of convention for the slope is further discussed in Section~\ref{sec:cabling}.

\subsection{Rational open book decompositions}\label{ss:robd}

A \dfn{rational open book decomposition} for a manifold $M$ is a pair $(L, \pi)$ consisting of an oriented link $L$ in $M$ and a fibration $\pi \colon (M\setminus L)\to S^1$ such that no component of $\pi^{-1}(\theta)$ is meridional for any $\theta\in S^1$.  In other words, if $N$ is a small tubular neighborhood of $L$ then no component of $\partial N\cap \pi^{-1}(\theta)$ is a meridian of a component of $L$. We note that a rational open book can differ from an {\em honest} open book in two ways:
\begin{enumerate}
\item[$(\star)$] a component of  $\partial N\cap \pi^{-1}(\theta)$ does not have to be a longitude to a component of $L,$ and
\item[$(\star \star)$] a component of $\partial N$ intersected with $\pi^{-1}(\theta)$ does not have to be connected.
\end{enumerate}
In particular, if $L$ is a knot then it is rationally null-homologous. This indicates the reason for the name ``rational open book''. As usual $\overline{\pi^{-1}(\theta)}$ is called a \dfn{page} of the open book for any $\theta\in S^1$ and $L$ is called the \dfn{binding} of the open book. We will usually put the word ``rational'' in front of ``open book'' when referring to the above concept.   Sometimes to emphasize that we are referring to the original meaning of ``open book'' we will use the phrase ``honest open book'' or ``integral open book''.

We note that just as for honest open books, one may describe rational open books using their monodromy map. That is, given $(L,\pi)$ a rational open book for $M$, the fibration $\pi \colon (M\setminus L)\to S^1$ is a mapping torus of a diffeomorphism $\phi:\Sigma\to\Sigma$ where $\Sigma=\overline{(\pi^{-1} (\theta) )}$ for some $\theta\in S^1$. We call $\phi$ the monodromy of the open book. 
For an honest open book one demands that $\phi$ is the identity in a neighborhood of the boundary, but for rational open books we allow $\phi$ to be the identity in a neighborhood of the boundary, to be a rigid rotation in either direction (of order less than $2\pi$), or to identify the neighborhood of one boundary with another. In particular we require that some power of $\phi$ is the identity on each boundary component.

\subsection{Torus knots and other examples of rational open books.}
In this subsection we discuss various basic examples and constructions of rational open book decompositions. 

\subsubsection{Torus knots in lens spaces}\label{Ex:torusknots}
Torus knots in lens spaces provide a fundamental class of rational open books.  Fix an oriented longitude-meridian basis $([\gamma], [\alpha])$ for the boundary $T$ of an oriented solid torus $U_\alpha$ (viewing $U_\alpha$ as a standard neighborhood of a knot as above).  With respect to this basis, let $\beta$ be a simple closed curve on $T$ of slope $\frac sr$ for coprime integers $0\leq s < r$.  Attaching another solid torus $U_\beta$ to $U_\alpha$ along $T$ so that $\beta$ is a meridian of $U_\beta$ forms the lens space $-L(r,s)$.  For coprime integers $k$ and $l$ we define the $(k,l)$--torus knot in $-L(r,s)$ to be the simple closed curve on $T$ of slope $\frac lk$ and denote it as $T_{(k,l)}^{(r,s)}$ or simply $T_{(k,l)}$ when the ambient lens space is understood.   If $k \neq 0$ then we may orient $T_{(k,l)}$ so that it is homologous to $k \cdot \gamma$ in $U_\alpha$.  (If $k = 0$, then the torus knot is the meridian of $U_\alpha$ so it bounds an embedded disk and the two orientations yield isotopic knots.)  If $\pm (k,l) = (0,1)$ or $(r,s)$, then $T_{(k,l)}^{(r,s)}$ is a meridian of $U_\alpha$ or $U_\beta,$ respectively, and hence bounds a disk.  We say such torus knots are \dfn{trivial}.  

\begin{rem}{We use the lens space conventions most common to 4-manifold, contact and symplectic topologists and opposite that used by most 3-manifold topologists.  With this convention, $L(r,s)$ is given by $-\frac{r}{s}$ surgery on the unknot, rather than $\frac{r}{s}$.  }
\end{rem}

\begin{lem}\label{lem:torusknotopenbook}\label{lem:torusknotfiber}
The non-trivial torus knot $T_{(k,l)}^{(r,s)}$ in the lens space $-L(r,s)$ is the binding of a rational open book. The page of the open book $\Sigma_{(k,l)}^{(r,s)}$  is a surface of Euler characteristic 
\[\frac{|k|+|ks-lr| - |k (ks-lr)|}{\gcd(r,k)}
\] 
and 
\[
\frac{\gcd(r,k^2)}{\gcd(r,k)}
\]
boundary components. Moreover, as an element of $H_1(-L(r,s);\IZ)$ the knot $T_{(k,l)}^{(r,s)}$ has order   
\[
\frac{r}{\gcd(r,k)}.
\]
Notice that this implies that the total boundary of a fiber in a fibration of the complement of $T_{(k,l)}^{(r,s)}$ wraps $\frac{r}{\gcd(r,k)}$ times around $T_{(k,l)}^{(r,s)}$ and each boundary component of the fiber wraps $\frac{r}{\gcd(r, k^2)}$ times around $T_{(k,l)}^{(r,s)}.$
\end{lem}

\begin{proof}
Consider the torus knot $T_{(k,l)}$ in $-L(r,s)$.
The exterior of $T_{(k,l)}$, $(-L(r,s) \setminus N_{T_{(k,l)}})$, may be viewed as the union of $U_\alpha$ and $U_\beta$ glued together along the complementary annulus $T \setminus (T\cap N_{T_{(k,l)}})$.  When $T_{(k,l)}$ is non-trivial then this annulus is essential in each $U_\alpha$ and $U_\beta$ and hence the exterior of $T_{(k,l)}$ is a small Seifert fiber space over the disk with two exceptional fibers.  Thus there exists a fibration $\pi \colon (-L(r,s) \setminus N_{T_{(k,l)}}) \to S^1$. This fibration can be seen as a (multi-section) of the Seifert fibration or can be constructed directly as we do below.  No component of $\bdry N_{T_{(k,l)}} \cap  \pi^{-1}(\theta)$ is a meridian since that would imply that $[T_{(k,l)}]$ has infinite order in $H_1(-L(r,s);\Z)=\Z/r\Z.$ 

In direct analogy with torus knots in $S^3$, a fiber $\Sigma_{(k,l)}^{(r,s)}$ of $\pi \colon (-L(r,s) \setminus N_{T_{(k,l)}}) \to S^1$ may be viewed as the union of $\frac{| ks-lr|}{\gcd(r,k)}$ meridional disks of $U_\alpha$ and $\frac{| k |}{\gcd(r,k)}$ meridional disks of $U_\beta$ joined together by $\frac{| k (ks-lr)|}{\gcd(r,k)}$ bands in $T \times (-\epsilon, \epsilon) - N_{T_{(k,l)}}$.  The number of disks is due to 
\[
(x,y,z) = (\frac{ks-lr}{\gcd(r,k)},\frac{k}{\gcd(r,k)},\frac{r}{\gcd(r,k)})
\] 
giving the ``smallest'' non-trivial integral solution to $x[\alpha]+y[\beta] = z[T_{(k,l)}]$ in $H_1(T;\Z)$.  The number of bands then may be seen as resulting from $T_{(k,l)}^{(r,s)}$ intersecting $\alpha$ minimally $|k|$ times and using $\frac{|ks-lr|}{\gcd(r,k)}$ meridional disks of $U_\alpha$.  The surface $\Sigma_{(k,l)}^{(r,s)}$ is verified to be a fiber of a fibration by either direct inspection or using Gabai's sutured manifold theory \cite{Gabai83}: the complement in the Heegaard torus $T$ of $T_{(k,l)}$ and the bands of the surface $\Sigma_{(k,l)}^{(r,s)}$ give rise to a complete set of product disks for a sutured manifold decomposition of the sutured manifold $(M \setminus (\Sigma_{(k,l)}^{(r,s)}\times I), \bdry (M \setminus (\Sigma_{(k,l)}^{(r,s)} \times I)))$ where $M = -L(r,s) \setminus N_{T_{(k,l)}^{(r,s)}}$ is the torus knot exterior.  Thus the fiber has Euler characteristic $\frac{|k|+|ks-lr| - |k (ks-lr)|}{\gcd(r,k)}$.  

From this description we may also calculate that the fiber $\Sigma_{(k,l)}^{(r,s)}$ has $\frac{\gcd(r,k^2)}{\gcd(r,k)}$ boundary components as follows.   
The order of $[T_{(k,l)}^{(r,s)}]$  in $H_1(-L(r,s);\Z)$ gives the number of times $\bdry \Sigma_{(k,l)}^{(r,s)}$ intersects its meridian $\mu$.  The number of bands joining the meridional disks in the construction of $\Sigma_{(k,l)}^{(r,s)}$ gives the number of times its boundary intersects $\lambda$.  Therefore the homology class of $\bdry \Sigma_{(k,l)}^{(r,s)}$ in $H_1(\bdry N_{T_{(k,l)}};\Z)$ is $\frac{k(ks-lr)}{\gcd(r,k)}[\mu] +\frac{k}{\gcd(r,k)}[\lambda]$ with respect to the meridian $\mu$ of $T_{(k,l)}^{(r,s)}$ and longitude $\lambda$ induced from $T$.    Thus $\bdry \Sigma_{(k,l)}^{(r,s)}$ has 
\[
\gcd \left( \frac{k}{\gcd(r,k)}, \frac{k(ks-lr)}{\gcd(r,k)} \right) = \frac{\gcd(r,k^2)}{\gcd(r,k)}
\]
components.  (Obtaining this equality makes use of the facts that $\gcd(r,s)=1$ and $\gcd(k,l)=1$.)
\end{proof}

\begin{Ex} \label{ex:ratlunknot} Let $N$ be a small tubular neighborhood of the non-trivial torus knot $T_{(k,l)}^{(r,s)}$.  Let $\pi$ be the fibration of its exterior $-L(r,s) \setminus N$.  The following examples illustrate how a rational open book may differ  from an honest open book by just one of properties $(\star)$ and $(\star\star)$ or both.
\begin{enumerate}
\item  The torus knot $T_{(1,n)}^{(r,s)}$, for any integer $n$,  has disk pages with $\bdry N \cap \pi^{-1}(\theta)$ connected and running $r$ times longitudinally on $\bdry N$.  Indeed, $T_{(1,n)}^{(r,s)}$ is isotopic to the core of $U_\alpha$ and its exterior is a solid torus.  The fibration $\pi \colon (-L(r,s) \setminus T_{(1,n)}) \to S^1$ may be identified with the fibration of $U_\beta$ by meridional disks.  These knots are called \dfn{rational unknots}.

 (Note that the only rational unknot that is also a trivial knot is the standard unknot in $S^3$.  In contrast to trivial knots in other manifolds, it is the binding of an open book.)

\item Similarly, the torus knot $T_{(t,u)}^{(r,s)}$, for any $(t,u)$ such that $ru-ts = \pm 1$, is a rational unknot.  It is isotopic to the core of $U_\beta$ and so its exterior is a solid torus too.  Observe that $T_{(p,1)}^{(1,0)}$, the $(p,1)$--torus knot in $S^3$, is an unknot. 

\item The torus knot $T_{(2,1)}^{(4,1)}$ has annular pages.  Hence $\bdry N \cap \pi^{-1}(\theta)$ has two components.  Since the knot has order $2$, each of these components is a longitude.

\item The torus knot $T_{(2,1)}^{(8,1)}$ has twice-punctured torus pages.  Again $\bdry N \cap \pi^{-1}(\theta)$ has two components.  But since the knot has order $4$, each of these components run twice longitudinally on $\bdry N$.

\end{enumerate}
\end{Ex}

The entire above discussion may be extended to the torus links $T_{(k,l)}^{(r,s)}$ where $\gcd(k,l) \neq 1$.   They give examples of rational open books, as long as no component is a trivial knot.

\subsubsection{Rational open books produced by Dehn surgery.} 
Given an honest open book decomposition $(L, \pi)$ for a manifold $M$, let $\bar{\gamma} = (\gamma_1, \dots, \gamma_n)$ where $\gamma_i$ is a slope on the boundary of a small tubular neighborhood $N_{L_i}$ of $L_i$, the $i^\text{th}$ component of $L$.  Fix a $\theta \in S^1$.  Assume
\begin{itemize}
\item $\gamma_i$ is not isotopic on $\bdry N_{L_i}$ to $\bdry N_{L_i} \cap \pi^{-1}(\theta)$ for any $i = 1,\dots, n$, and 
\item there exists some $i$ for which $\gamma_i$ minimally intersects $\bdry N_{L_i} \cap \pi^{-1}(\theta)$ more than once.
\end{itemize}
Then the Dehn surgered manifold $M' = M_L(\bar{\gamma})$ has a rational open book decomposition $(L',\pi')$ where $L'$ is the link in $M'$ obtained from the cores of the Dehn surgery solid tori, and $\pi' \colon (M' \setminus L') \to S^1$ is the same fibration as $\pi$ since $M' \setminus L' = M \setminus L$.  The two properties imposed upon $\bar{\gamma}$ ensure that $(L', \pi')$ is a rational open book but not an honest open book. (If one ignored the second property then one still obtains a rational open book, but it might in fact be an honest open book.)

\begin{Ex} Let $K$ be a fibered knot in $S^3$ with fibration $\pi \colon (S^3\setminus K) \to S^1$.  Then the surgered manifold $S^3_K(q/p)$ with $q/p \neq 0$ admits an open book decomposition $(K', \pi)$ where $K'$ is the core of the surgery solid torus.  This is an honest open book decomposition if $p=\pm1$ and a rational open book decomposition otherwise.

In particular, any knot in $S^3$ with a Dehn surgery yielding a lens space is known to be fibered \cite{OzsvathSzabo05}, \cite{Ni07}.  Thus they confer a rational open book decomposition upon the resulting lens space.
\end{Ex}

\subsubsection{Grid number one knots in lens spaces}
The known non-torus knots in lens spaces admitting a Dehn surgery yielding $S^3$ all have grid number one (see \cite{BergeXX}, \cite{BakerGrigsbyHedden08}).  Indeed all grid number one knots in lens spaces that represent a generator of $H_1(L(r,s);\Z)$ have fibered exterior \cite{OzsvathSzabo05} and hence are the bindings of rational open books.  Grid number one knots that do not represent a generator of $H_1(L(r,s);\Z)$ are not always fibered.  This may be easily observed through several straightforward calculations of their knot Floer homology, \cite{BakerGrigsbyHedden08}, \cite{Ni07}.  Also, this is explicitly catalogued for grid number one knots with once-punctured torus rational Seifert surfaces in \cite{Baker06}.

\subsection{Inducing contact structures from rational open book decompositions.}\label{sec:thirstywinkle}

Generalizing a definition of Giroux, we say a rational open book $(L,\pi)$ for $M$ {\em supports} a contact structure $\xi$ if there is a contact form $\alpha$ for $\xi$ such that 
\begin{enumerate}
\item $\alpha(v)>0$ for all positively pointing tangent vectors $v\in TL,$ and
\item $d\alpha$ is a volume form when restricted to each page of the open book.
\end{enumerate}

The  existence part of the proof of Theorem~\ref{thm:support} is a small modification of Thurston and Winkelnkemper's original proof for honest open books \cite{ThurstonWinkelnkemper75}, {\em cf} \cite{Etnyre06}. Uniqueness readily follows as in \cite{Giroux02} with the  appropriate definition of ``support'' given above.
\begin{proof}[Proof of Theorem~\ref{thm:support}]
Observe
\[
M_\phi=\Sigma_\phi \cup_\psi \left(\coprod_{|\partial \Sigma_\phi|} S^1\times D^2\right),
\]
where $\Sigma_\phi$ is the mapping torus of $\phi,$ and $\psi$ is a diffeomorphism used to glue the solid tori to $\Sigma_\phi.$   Note we use $|\bdry \Sigma_\phi|$ rather than $|\bdry \Sigma|$ because the monodromy $\phi$ may permute the components of $\bdry \Sigma$. (See Subsection~\ref{ss:robd} for a discussion of the monodromy of rational open books.)

We first construct a contact structure on
$\Sigma_\phi.$  Let $\lambda$ be a 1--form on $\Sigma$ such that $d\lambda$ is a volume form on $\Sigma$ and $\lambda = s\,d\theta$ in the coordinates $(s,\theta)\in [-1,-1+\epsilon]\times S^1$ near each boundary component of $\Sigma$ for some sufficiently small $\epsilon >0$. (Here $s=-1+\epsilon$ corresponds to $\partial \Sigma$). Consider 
the 1--form
\[
\lambda_{(t,x)}= t\lambda_x +(1-t)(\phi^*\lambda)_x
\]
on $\Sigma\times[0,1]$ where $(x,t)\in\Sigma\times[0,1]$ and set
\[\alpha_K=\lambda_{(t,x)}+ Kdt.\]
For sufficiently large $K$ this form is a contact form and it 
is clear that this form descends to a contact form on the mapping torus $\Sigma_\phi.$ (For details on 
the existence of $\lambda$ or the above construction see \cite{Etnyre06, ThurstonWinkelnkemper75}.)

We now
want to extend this form over the solid tori neighborhood of the binding. To this end consider the
map $\psi$ that glues the solid tori to the mapping torus. Using coordinates $(\varphi,(r,\vartheta))$
on $S^1\times D^2$ where $D^2$ is the unit disk in the $\R^2$ with polar coordinates and coordinates $(s, \theta, t)$ as above on the component of $N(\bdry \Sigma_\phi)$ at hand, we have
\begin{align*}
\psi\colon S^1 \times N(\bdry D^2) &\to N(\bdry \Sigma_\phi) \\
(\varphi, r, \vartheta) &\mapsto  (-r, p\varphi+q\vartheta, -q\varphi+p\vartheta).
\end{align*}
This is a map defined near the boundary of $S^1\times D^2$ where $N(\bdry D^2)$ contains the annulus $r\in[1-\epsilon,1]$. Pulling back the contact form $\alpha_K$
using this map gives
\begin{align*}
\alpha_\psi&=-r(p \, d\varphi+q\, d\vartheta) +  K(-q\, d\varphi + p \, d\vartheta)\\
&=(-rp-Kq)\,d\varphi + (-rq+pK)\, d\vartheta.
\end{align*}
We now need to extend $\alpha_\psi$ over all of $S^1\times D^2$. We will extend using a form of the form
\[
f(r)\, d\varphi + g(r)\, d\vartheta.
\]
This form is a contact form if and only if $f(r)g'(r)-f'(r)g(r)>0.$ 
Near $\bdry S^1 \times D^2$,  $\alpha_\psi$ is defined with $f(r)= -rp-qK$ and $g(r)=-rq+pK.$ Near the core of
$S^1\times D^2$ we would like $f(r)=1$ and $g(r)=r^2$. One may easily extend $f(r)$ and $g(r)$ so that  $\alpha_\psi$ is a contact form on the solid torus. Moreover, it is easy to check that $f(r)$ and $g(r)$ can be chosen so that $d\alpha_\psi$ is non-zero on the extension of the pages to the core of the solid torus.

For uniqueness suppose that $\alpha_1$ and $\alpha_2$ are contact forms for two contact structures supported by $(L,\pi).$ 
Consider the form $\pi^* d\theta$ on $M-L,$  where $\theta$ is the angular coordinate on $S^1.$ Let $f:M\to \R$ be a function of the distance to $L$ that is 1 outside a neighborhood of $L,$ vanishes to order 2 on $L$ and increasing in between. Now set $\eta$ to be $f\, \pi^* d\theta$ extended to be $0$ over $L.$ This is a global 1--form on $M$ that acts like ``$dt$'' above, outside a neighborhood of $L.$
Then for any positive $K$ the form $K\eta+\alpha_i$ is a contact form for a contact structure isotopic to $\ker \alpha_i.$ For $K$ sufficiently large the family of forms $K\, dt + (t\alpha_1+ (1-t) \alpha_2)$ on $M\setminus N$ are all contact forms. From this one easily constructs the isotopy between the contact structures. For more details see \cite{Etnyre06}.
\end{proof}

\subsection{Stabilization} 
For later use in the paper we recall the notion of stabilizing an honest open book decomposition. The intrinsic definition is as follows. If $(L,\pi)$ is an open book decomposition supporting $\xi$ on $M$ then choose an arc $\alpha$ properly embedded in a page of the open book and perform a Murasugi sum of $(L,\pi)$ with the negative Hopf link (this is an open book for $S^3$ supporting the tight contact structure) along $\alpha.$ More specifically choose an arc $\beta$ in a page of the Hopf link open book and identify a neighborhood of $\alpha$ with a neighborhood of $\beta$ so that the pages are plumbed together. This results in a new open book $(L_\alpha,\pi_\alpha)$ supporting $\xi$ on $M.$ See \cite{Etnyre06}. The open book decomposition $(L_\alpha,\pi_\alpha)$  is said to be obtained from $(L,\pi)$ by \dfn{positive stabilization along $\alpha.$} 
An open book decomposition can be described via its monodromy presentation. That is the complement of an open neighborhood of $L$ is a surface bundle over $S^1$ and hence there is some diffeomorphism $\phi$ of a fiber $\Sigma$ of $\pi$ that fixes the boundary of $\Sigma$ so  that the complement of a neighborhood of $L$ in $M$ is diffeomorphic to 
\[
\Sigma\times[0,1]/\sim,
\]
where $(1,x)\sim (0, \phi(x)).$ If we add the relation $(t,x)\sim (t',x)$ for all $x\in \partial \Sigma$ then we get back $M.$ So the pair $(\Sigma, \phi)$ will be called the monodromy presentation of the open book decomposition of $M.$ 
In terms of the monodromy presentation $(\Sigma, \phi)$ of the open book  $(L,\pi)$ we can describe a positive stabilization as follows. Again fix an arc $\alpha$ properly embedded in $\Sigma.$ Let $\Sigma_\alpha$ be the surface obtained from $\Sigma$ by attaching a 1-handle along $\partial \alpha\subset \partial \Sigma.$ Let $\phi_\alpha=\phi\circ D_{\alpha'}$ where $\alpha'$ is the simple closed curve  $\alpha$ union the core of the 1-handle and $D_{\alpha'}$ is a right handed Dehn twist along $\alpha'.$ The open book decomposition $(\Sigma_\alpha,\phi_\alpha)$ also supports $\xi$ on $M$ and is called the \dfn{positive stabilization of $(\Sigma,\phi)$ along $\alpha.$}

\section{Cables of open books}\label{sec:cabling}

Refer to the beginning of Section~\ref{sec:rob} for our conventions about curves on the boundary of a standard neighborhood of a knot.  Given a (framed, oriented) knot $K$ in a manifold $M$, replacing $K$ with the curve $K_{(p,q)}$ on $\bdry N_K$ forms the \dfn{$(p,q)$--cable} of $K$, so long as $p \neq 0$.  Note $K_{(0,q)}$ is just a collection of meridians of $K$.  Each $(1,q)$--cable of $K$ is isotopic to $K$ in $N_K$ as an oriented curve, and each $(-1,q)$--cable of $K$ is isotopic to $-K$.  Mind that the curves $K_{(p,q)}$ and $K_{(-p,-q)}$ differ by a reversal of orientation.

If $L$ is an oriented link with $n$ components we can order the components, choose $n$ pairs of integers $(p_i,q_i)$, and then $(p_i,q_i)$--cable the $i^\text{th}$ component of $L$.  This will be denoted $L_{({\mathbf{p}},{\mathbf{q}})}$ where $(\mathbf{p},\mathbf{q})$ is the $n$-tuple of pairs $((p_1,q_1),\ldots, (p_n,q_n))$.  If $L$ is the binding of an honest open book then each component of $L$ has a natural framing coming from a page of the open book. Throughout this paper we assume this framing is used when discussing the bindings of an integral open book.  If $L$ is the binding of a rational open book then the pages might not induce a framing on $L$.  

In this case one must simply choose a framing for each component. Given a choice of framing, a page of the rational open book approaches a component $K$ of $L$ as a cone on an $(r,s)$--curve for some $r>0$ and $s$. (In other words, the page intersects a neighborhood $N$ of $K$ in the obvious collection of annuli each with one boundary component  a component of a $(r,s)$--curve on $\partial N$ and the other boundary component wrapping $r/\gcd(r,s)$ times around $K$.)   In either case, the slope $\frac{s}{r}$ at which a page encounters the standard neighborhood of a binding component is called the \dfn{Seifert slope} for that component.  For binding components of an integral open book using its natural framing, the Seifert slope corresponds to $0$.  Moreover, regardless of what framing is used, the Seifert slope of an integral open book is integral.

The Seifert slope and the meridian partition the remaining slopes on $\bdry N$ into positive and negative.
Relative to the Seifert slope $\frac{s}{r}$, a pair of integers $(p,q)$ with $p \neq 0$ defines a slope $\frac{q}{p}$ that is \dfn{positive}  \label{posnegdef} if $\frac{q}{p}>\frac{s}{r}$ and is \dfn{negative} if $\frac{q}{p}<\frac{s}{r}$.   We also say the pair of integers and the curve on $\bdry N$ it represents are positive or negative accordingly.  In particular, this permits us to speak of a $(p,q)$--cable as being positive or negative.

Given an open book, a cabling of its binding components naturally produces another open book --- except when a component is cabled along its Seifert slope (or its meridional slope).

\begin{lem}\label{lem:cable}
Let $(L,\pi)$ be a (rational) open book decomposition for a 3--manifold $M.$ For each component $L_i, i=1,\ldots, n,$ of $L,$ let $(p_i,q_i)$ be a pair of integers for which $q_i/p_i$  is neither the Seifert slope nor meridional slope of $L_i$, and all the $p_i$ have the same sign, and then set $(\mathbf{p},\mathbf{q})=((p_1,q_1),\ldots, (p_n,q_n)).$  The link $L_{(\mathbf{p},\mathbf{q})}$ is the binding of a rational open book for $M$, and this open book is naturally induced from the open book $(L, \pi)$. If the $p_i$'s all have the same magnitude and $(L,\pi)$ is an honest open book then $L_{(\mathbf{p},\mathbf{q})}$ is the binding of an honest open book.
\end{lem}
\begin{proof}
There are several ways to prove this statement. A simple way that will be useful in what follows is to notice that the $(p,q)$ torus knot $T_{(p,q)}$ sits on a standardly embedded torus $T\subset S^3$ that bounds solid tori $V_0$ and $V_1$ such that $S^3=V_0\cup V_1.$ It is well known that $T_{(p,q)}$ is the binding of an open book for $S^3.$ Moreover it is easily checked that if $C$ is the core of $V_1,$ say, then $C$ intersects the pages of this open book transversely and it intersects each page in $p$ points. (We view $T_{(p,q)}$ as a $(p,q)$--curve on the boundary of $V_0$.)  Thus the complement of a small tubular neighborhood of $C$ is a solid torus $S_{(p,q)}=D^2\times S^1$ containing $T_{(p,q)}$ such that the open book structure on $S^3$ gives a fibration $S_{(p,q)}\setminus T_{(p,q)}\to S^1.$ This fibration induces a fibration of $\partial S_{(p,q)}$ by curves $\{pt\}\times S^1$ and the preimage of any point in $S^1$ intersected with $\partial S_{(p,q)}$ is $p$ curves. 

Suppose the knot $K$ is the binding of an honest open book, $N$ is a small tubular neighborhood of $K$ and $\pi\colon M\setminus N\to S^1$ is the fibration of the complement of $N.$ Letting $f_p\colon S^1\to S^1$ be the $p$-fold covering map one sees that $S_{(p,q)}$ may be glued to $M\setminus N$ to recover $M$ and so that the fibration $f_p\circ \pi\colon M\setminus N\to S^1$ and  $S_{(p,q)}\setminus T_{(p,q)}\to S^1$ glue together to give an open book decomposition of $M$ with binding $K_{(p,q)}.$ In the case of an open book $(L,\pi)$ with multiple binding components it is clear that if all the $p_i$ have the same magnitude then the same construction yields an open book structure with binding $L_{(\mathbf{p},\mathbf{q})}$. If the $p_i$ have different magnitudes then let $p$ be the least common multiple of the $p_i$. Now using the fibration $f_p\circ \pi$ on the complement of a neighborhood of the binding, one easily sees a rational open book structure with binding $L_{(\mathbf{p},\mathbf{q})}.$ Notice that in this construction we have not been paying attention to the orientation on the page.  Taking page orientations into account forces all the signs on the $p_i$ to be the same. (The construction really deals with the case when all the $p_i$ were positive. For the negative case, reverse the orientation on $L.$)

The rational open book case can be similarly considered. Specifically for binding components with multiplicity one (that is one boundary component of a page contains the binding component) we begin with the torus knots in lens spaces from Examples~\ref{Ex:torusknots} above. Using the same argument for $T_{(k,l)}^{(r,s)}$, the $(k,l)$--torus link in the lens space $L(r,s)$, one can construct a solid torus $S_{(k,l)}^{(r,s)}$ containing $T_{(k,l)}^{(r,s)}$ that can be used to $(k,l)$--cable a binding component for which a page approaches as a cone on an $(r,s)$--curve. If a binding component has higher multiplicity then we compose the fibration of the complement of $T_{(k,l)}^{(r,s)}$ in the solid torus with a covering map of the circle. 
\end{proof}

\begin{rem}
Observe that the open book constructed for the cable $L_{(\mathbf{p},\mathbf{q})}$ above may be obtained by a sequence of cablings, cabling one component of $L$ at a time in any order.
\end{rem}

\begin{rem}
Notice that the $((2,1),(3,1))$--cable of the positive Hopf link is a (non-integral) rational open book for $S^3$.  Its page wraps along the first component with multiplicity $3$ and the second with multiplicity $2$.  
\end{rem}

The following relationship between the $(p,q)$--cable of a fibered knot and its $(p,\pm1)$--cable ({\em cf.\ }\cite[Figure 4.2]{NeumannRudolph87}) will be particularly useful in the proof of Theorem~\ref{thm:main-cable} for integral open books.
\begin{lem}\label{lem:gencable}
Let $L$ be a fibered knot in a manifold.
The cable $L_{(p,q)}$ is obtained from $L_{(p,\text{sgn} (q))}$ by $(|p|-1)(|q|-1)$ negative, respectively positive, stabilizations when $pq<0,$ respectively $pq>0,$ where $\text{sgn}(q)$ is $+1$ if $q>0$ and $-1$ if $q<0.$ Equivalently, $L_{(p,q)}$ is obtained from $L_{(p,\text{sgn}(q))}$ by Murasugi summing with the $(p,q)$--torus link $T_{(p,q)}$ in $S^3.$
\end{lem}
\begin{proof}
In the left hand side of Figure~\ref{fig:p1topq} we show how to go from a $(p,1)$--cable of a knot to the $(p,q)$--cable. The horizontal sheets are copies of the Seifert surface for $L$. To go from the top to the bottom pictures in the figure one simply attaches $(|p|-1)(|q|-1)$ bands with a right handed twist. It is easy to see that this corresponds to a positive stabilization of the open book.
\begin{figure}[ht]
 {\epsfysize=2.5truein\centerline {\epsfbox{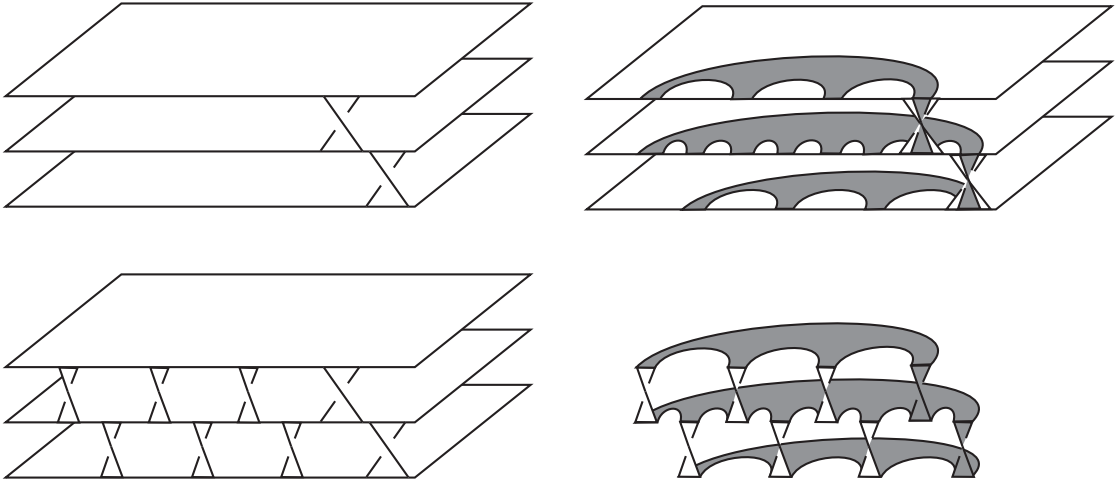}}} 
        \caption{Left: the top picture is the $(3,1)$--cable of $L$ and the bottom picture is the $(3,4)$--cable of $L.$ Right: the top picture is the $(3,1)$--cable of $L$ with a Murasugi summing disk shaded and the bottom picture is the $(3,4)$--torus knot with a Murasugi summing disk shaded.}
        \label{fig:p1topq}
\end{figure}
For the second statement notice that the two shaded disks on the right hand side of Figure~\ref{fig:p1topq} are Murasugi summing disks that show the $(p,q)$--cable of $L$ is the Murasugi sum of the $(p,1)$--cable and the $(p,q)$--torus link.  Similar arguments work for $(p,q)$--cables when $q<0.$
\end{proof}

\section{The proof of Theorem~\ref{thm:main-cable}}\label{proofofmain}

In this section we give the proof of Theorem~\ref{thm:main-cable}.   We break this proof into a few parts for clarity.  In Subsection~\ref{sec:parts1and2} we show that positive cablings of open books preserve the supported contact structures, parts (1) and (2) of the theorem.  As these are the main results needed for all our applications, the reader primarily interested in the applications only needs to read the first subsection; though, the surprising rich and subtle structure of negative cables is  interesting in its own right. In Subsection~\ref{sec:part3OT} we show that all non-exceptional negative cablings of open books support overtwisted contact structures.  In Subsection~\ref{sec:lutz} we show how the homotopy class of the plane field of the contact structure supported by a negative cabling may be induced by Lutz twists.   Finally in Subsection~\ref{sec:thm}  we pull these together to prove the theorem and its corollary, as well as several propositions about exceptional cables.

Since a cabling of a link may be done one component at a time, we focus our attention on cabling just one binding component of an open book.  To this end, throughout this section let $(L,\pi)$ be a rational open book for $M$ with binding components $L_1=K, L_2, \dots, L_n$ and supporting the contact structure $\xi$.   Consider the $(p,q)$--cable of $(L,\pi)$ --- that is, taking the $(p,q)$--cable of $K$ and not cabling the remaining components of $L$.   We will use $K_{(p,q)}$ to denote only the $(p,q)$--cable of $K$ while $L_{(p,q)}$ is the entire cabled link $K_{(p,q)} \cup L_2 \cup \dots \cup L_n$.

Assume a page of $(L,\pi)$ approaches $K$ as an $(r,s)$ curve, having chosen a framing of $K$ such that $0\leq s< r$.
Assuming the $(p,q)$--cabling slope  is neither the trivial slope nor the Seifert slope, then the cabled open book supports a contact structure $\xi_{(p,q)}$.

\subsection{Compatibility of positive cablings}\label{sec:parts1and2} 

For simplicity, we first argue in the case that $(L,\pi)$ is an integral open book.  Thereafter the case of a rational open book is a straightforward generalization.

\begin{lem}\label{main-cable-lemma1}
Let $(L, \pi)$ be an integral open book compatible with $\xi$.  If $(p,q)$ is a positive slope and $p>0$ then the cabled open book $L_{(p,q)}$ is also compatible with $\xi$. 
\end{lem}
\begin{proof}

Recall in the proof of Lemma~\ref{lem:cable} we replaced a small neighborhood $N$ of a binding component $K$ of $L$ with $S_{(p,q)}=S^3\setminus N'$ where $N'$ was a neighborhood of the unknotted curve $C$ intersecting a fiber of the fibration of the $(p,q)$--torus link in $p$ points  (see that proof for notation). The contact structure on $N=D_{\epsilon}^2\times S^1,$ with $D_{\epsilon}^2$ being the disk of radius $\epsilon$ in $\R^2,$ can be assumed to be given in coordinates $((r,\theta),\phi)$ by $f(r)\, d\theta+ d\phi,$ where $f(r)=r^2$ near 0 and $f(r) \gg 0$ near $\epsilon$. Similarly the contact structure on $N'=D^2_{\epsilon'}\times S^1$ in the coordinates $((r,\theta),\phi)$ is given by $g(r)\, d\theta + d\phi$ where $g(r)=r^2$ near 0 and $0<g(r) \ll \epsilon'$ near $\epsilon'$.  

\begin{claim}\label{claim:Spqctctstrhonest}
The contact structure on $S_{(p,q)}=S^3\setminus N'$ is contactomorphic to  the one on $N$ for the appropriate choice of $\epsilon.$ 
\end{claim}

\begin{proof} 
To see this we examine the contact structure on $S_{(p,q)}$.   Let 
\[
S^3=\{(z_1,z_2)\in \C^2\, |\; \,  p|z_1|^2+q|z_2|^2=1\}.
\]
As $S^3$ is transverse to the radial vector field on $\C^2$ we see that $r_1\, d\theta_1+r_2\, d\theta_2$ restricts to a contact form on $S^3$ giving the standard tight contact structure, where $z_j=r_je^{i\theta_j}.$ One may easily check that there are two closed unknotted trajectories $C_0, C_1$ to the Reeb field corresponding to $\{z_j=0\}$ for $j=0,1$. In addition, $C_0 \cup C_1$ is a Hopf link with complement fibered by $(p,q)$--torus knots which are orbits of the flow of the Reeb vector field. One may also check that if we fix one of these fibers $T_{(p,q)}$ then $S^3\setminus T_{(p,q)}$ is fibered by surfaces transverse to the Reeb trajectories. This shows that the standard tight contact structure is supported by the open book with binding $T_{(p,q)}.$ Moreover, we see that $C_1$ is an unknot that intersects the pages of this open book $p$-times; in particular, $C_1$ is $C$ from the proof of Lemma~\ref{lem:cable} (and Lemma~\ref{main-cable-lemma1} above) and thus $C$ is a transverse unknot with  self-intersection $-1.$ The complement of a neighborhood of such an unknot in $S^3$ is easily seen to be contactomorphic to the one on $N$, since they are both universally tight and the neighborhoods can be chosen so that  they have the same characteristic foliation on their boundaries, see \cite{Giroux00, Honda00a}.
\end{proof}

\begin{claim}\label{claim:glueOBhonest}
The contact structure on $M$ resulting from gluing $S_{(p,q)}$ in place of $N$ is supported by the image of $T_{(p,q)}$ in $M$, that is by the $(p,q)$--cable of the original binding. 
\end{claim}

\begin{proof}
Break $M$ into three regions $M\setminus N'',$ where $N''$ is a neighborhood of $K$ slightly larger than $N$ but with contact structure still given by a form as described above, $S_{(p,q)}$ and a region $T^2\times[0,1]=N''\setminus N$.   If $\alpha$ is a contact form for the original contact structure and $\alpha'$ is a contact form supported by $T_{(p,q)}$ on $S^3$ then it is clear by construction that $\alpha|_{M\setminus N''}$ is ``supported'' by the new open book on $M\setminus N''$ and $\alpha'|_{S_{(p,q)}}$ is ``supported'' by the new open book on $S_{(p,q)}$.  We need to extend $\alpha|_{M\setminus N''}$ and $\alpha'|_{S_{(p,q)}}$ to $N''\setminus N$.  This is easily accomplished as at the end of the proof Theorem~\ref{thm:support} above. More specifically, in coordinates $(\theta,\phi, t)$ on $T^2\times[0,1]=N''\setminus N$ we notice the above forms $\alpha|_{M\setminus N''}$ and $\alpha'|_{S_{(p,q)}}$ near $\partial (N''\setminus N)$ can be assumed to be of the form $f(t)\, d\theta + g(t)\, d\phi.$ One now extends the functions $f(t)$ and $g(t)$ across $[0,1]$ so they are compatible with the fibration of $N''\setminus N$ given by constant $\phi$'s.
\end{proof}

By Claim~\ref{claim:Spqctctstrhonest}, gluing $S_{(p,q)}$ in place of $N$ does not change the contact structure on $M$.   Then by Claim~\ref{claim:glueOBhonest}, the open book $L_{(p,q)}$ supports the original contact structure.
\end{proof}

\begin{lem}\label{lem:main-cable-lemma2}
Let $(L,\pi)$ be a rational open book compatible with $\xi$.  If $(p,q)$ is positive and $p>0$ then the cabled open book $L_{(p,q)}$ is also compatible with $\xi$. 
\end{lem}
\begin{proof}
Recall we think of the lens space $-L(r,s)$ as being obtained by gluing $V_1=S^1\times D^2$ to $V_0=S^1\times D^2$ so that $\{pt\}\times \bdry D^2$ maps to the $(r,s)$--curve in the boundary of $V_0$. We denote the core of $V_i$ by $C_i$. Moreover $T_{(p,q)}^{(r,s)}$ is the $(p,q)$--curve on the boundary of a neighborhood of $C_0$ in $V_0.$ 

In the proof of Lemma~\ref{lem:cable} we replaced a small neighborhood $N$ of a binding component $K$ of $L$ with $S^{(r,s)}_{(p,q)}=-L(r,s)\setminus N'$ where $N'$ was a neighborhood of the unknotted curve $C_1$ intersecting a fiber of the fibration of $T_{(p,q)}^{(r,s)}$ in $r/\gcd(p,r)$ points. The contact structure on $N=D_{\epsilon}^2\times S^1,$ with $D_{\epsilon}^2$ being the disk of radius $\epsilon$ in $\R^2,$ can be assumed to be given in coordinates $((r,\theta),\phi)$ by $f(r)\, d\theta+ d\phi,$ where $f(r)=r^2$ near 0 and $f(r) \gg 0$ near $\epsilon.$ Similarly the contact structure on $N'=D^2_{\epsilon'}\times S^1$ in the coordinates $((r,\theta),\phi)$ is given by $g(r)\, d\theta + d\phi$ where $g(r)=r^2$ near 0 and $0<g(r) \ll \epsilon'$ near $\epsilon'.$

\begin{claim}\label{claim:Spqctctstr}
The contact structure on $S^{(r,s)}_{(p,q)}=-L(r,s)\setminus N'$ is contactomorphic to  the one on $N$ for the appropriate choice of $\epsilon.$ 
\end{claim}

\begin{proof} 
We begin by recalling that $-L(r,s)=L(r, r-s)$ can be constructed from the unit sphere $S^3$ in $\C^2,$ with coordinates $(z_1,z_2),$  as follows. Let $g=e^{\frac{2\pi i}{r}}$ and define the $\Z_r$ action on $S^3$ by $g\cdot (z_1,z_2)=(gz_1, g^{(r-s)} z_2).$ Then $-L(r,s)$ is the quotient of $S^3$ under this action.  Notice that $T_{(p,q)}^{(r,s)}$ can be thought of as a (multiple) of a fiber in a Seifert fibration of $-L(r,s).$ More precisely, $-L(r,s)\setminus (C_0\cup C_1)$ is diffeomorphic to $T^2\times \R$ and can be fibered by $T_{(p,q)}^{(r,s)}$ curves in such a way that this extends to a Seifert fibration of $-L(r,s)$ with singular fibers $C_0$ and $C_1.$ There is a positive integer $n$ such that $nT_{(p,q)}^{(r,s)}$ lifts to a closed curve in $S^3$ that sits on a standardly embedded torus. Thus it lifts to a (multiple) of a torus knot which we denote by $T_{(p',q')}.$ Moreover, one may now easily check that the Seifert fibration of $S^3$ by $T_{(p',q')}$ curves covers the given Seifert fibration of $-L(r,s).$ Moreover, there is a Hopf link $C_0'\cup C_1'$ in $S^3$ that covers $C_0\cup C_1$ in $-L(r,s).$   We also note that the condition that $T_{(p,q)}^{(r,s)}$  is a positive torus knot in $-L(r,s)$ implies that $T_{(p',q')}$ is a positive torus knot in $S^3.$

We now represent $S^3$ as
\[
S^3=\{(z_1,z_2)\in \C^2\, |\; \,  p'|z_1|^2+q'|z_2|^2=1\}.
\]
As $S^3$ is transverse to the radial vector field on $\C^2$ we see that $r_1\, d\theta_1+r_2\, d\theta_2$ restricts to a contact form on $S^3$ giving the standard tight contact structure, where $z_j=r_je^{i\theta_j}.$ The action above clearly preserves this $S^3$ and hence gives a model for the universal cover of $-L(r,s).$ One may easily check that there are two closed unknotted trajectories $C'_0, C'_1$ to the Reeb field corresponding to $\{z_j=0\}$ for $j=1,2.$ In addition, the complement of the Hopf link $C_0'\cup C_1'$ is fibered by $(p',q')$--torus knots which are orbits of the flow of the Reeb vector field. As the form $r_1\, d\theta_1+r_2\, d\theta_2$ is equivarient  with respect to the $\Z_r$ action above it is clear that the $1$--form and Reeb vector field descend to $-L(r,s).$ This gives a Reeb vector field for a contact structure on $-L(r,s)$ whose orbits consist of $C_0, C_1$ and $T_{(p,q)}^{(r,s)}$ curves. We also note that the fibers of the fibration $\pi\colon (-L(r,s)\setminus T_{(p,q)}^{(r,s)})\to S^1$ can be made positively transverse to the Seifert fibration of $-L(r,s)\setminus T_{(p,q)}^{(r,s)}$ by Reeb orbits. (One way to see this is to notice that a fiber of $\pi$ is incompressible in $-L(r,s)\setminus T_{(p,q)}^{(r,s)}$ and thus can be made ``horizontal'' or ``vertical''. As it cannot be vertical, it must be horizontal, that is transverse to the fibers. Moreover, homologically we can see that it is positive transverse. It is now easy to make the other fibers of $\pi$ transverse.) This shows that the universally tight contact structure on $-L(r,s)=L(r,r-s)$ constructed above is supported by the open book with binding $T_{(p,q)}^{(r,s)}.$ Moreover, we see that $C_1$  intersects the pages of this open book $r/\gcd(p,r)$ times. The complement of $C_1$ is an open solid torus that is covered by the open solid torus in $S^3$ that is the complement of an unknot with self-linking $-1,$ and thus is universally tight. The complement of such an unknot in $S^3$ is easily seen to be contactomorphic to the one on $N$, since $N$ can be chosen so that the characteristic foliation is arbitrarily close to the pages slope, see \cite{Giroux00, Honda00a}.
\end{proof}

\begin{claim}\label{claim:glueOB}
The contact structure on $M$ resulting from gluing $S^{(r,s)}_{(p,q)}$ in place of $N$ is supported by the image of $T^{(r,s)}_{(p,q)}$ in $M$, that is by the $(p,q)$--cable of the original binding. 
\end{claim}

\begin{proof}
With the work done above, this proof is nearly identical to the proof of Claim~\ref{claim:glueOBhonest} and is left to the reader. 
\end{proof}

By Claim~\ref{claim:Spqctctstr}, gluing $S_{(p,q)}$ in place of $N$ does not change the contact structure on $M$.   Then by Claim~\ref{claim:glueOB}, the open book $L_{(p,q)}$ supports the original contact structure.
\end{proof}

\subsection{Most negative cables are overtwisted}\label{sec:part3OT}
In this section we consider when a negatively cabled open book necessarily supports an overtwisted contact structure.  One would expect that any non-trivial, negatively cabled open book supports an overtwisted contact structure except when both the original open book and the cabled open book have disk pages.  (Note that the rational unknots in lens spaces provide rational open books that necessarily support a universally tight contact structure.)  We show this holds for integral open books in  Lemma~\ref{lem:xipqOT}.  One may use a similar, but more complicated analysis to show that  sufficiently negative cables of rational open books are overtwisted, but to understand when they are not overtwisted a much more delicate argument is needed. The argument is presented in Lemma~\ref{exceptional}. We observe that the case for integral open books follows from our more detailed analysis, but we present our argument in this case separately as it is particularly easy and likely to be the most interesting to many readers. 

\begin{lem}\label{lem:xipqOT}
Statement (4) of Theorem~\ref{thm:main-cable} is true for integral open books.
\end{lem}
\begin{proof}
We first assume $L$ is the connected binding of an integral open book (where $(r,s)=(1,0)$), and take $p>1$ and $q<0$.   The case with $p<0$ is similar.   The case of more components clearly follows from the argument below. 
By Lemma~\ref{lem:gencable}, the open book for the cable $L_{(p,q)}$ (with $p>1$ and $q<0$) may be obtained as a Murasugi sum of the cable $L_{(p,-1)}$ and the $(p,q)$--torus link.  Thus the contact structure $\xi_{(p,q)}$ is contactomorphic to $\xi_{(p,-1)}\# \xi'_{(p,q)},$ where $\xi'_{(p,q)}$ is the contact structure on $S^3$ supported by the $(p,q)$--torus link.  Hence we are left to show $\xi_{(p,-1)}$ is overtwisted.

Recall the notion of \dfn{right veering} from \cite{HondaKazezMatic07}.  Given an open book decomposition of a manifold $(\Sigma, \phi)$ we say a properly embedded arc $\gamma\colon [0,1]\to \Sigma$ on $\Sigma$ is right veering if either $\phi(\gamma)$ and $\gamma$ are isotopic rel end points or when $\phi(\gamma)$ has been isotoped, rel endpoints, to intersect $\gamma$ transversely and minimally, the (inward pointing) tangent vector of $\phi(\gamma)$ at $\gamma(i)$ followed by the (inward pointing) tangent vector of $\gamma$ at $\gamma(i)$ form an oriented basis for $T_{\gamma(i)}\Sigma, i=0,1$.   We say the open book is right veering if all properly embedded arcs on a page are right veering. The main result of \cite{HondaKazezMatic07} is that if a contact structure is tight then any open book supporting it will be right veering. 

We shall see $\xi_{(p,-1)}$ is overtwisted by finding an arc on the open book for $L_{(p,-1)}$ that is not right veering.  To this end notice that if $\Sigma$ is a page of the open book associated to $T_{(p,-1)}$ in $S^3$ (notice that $\Sigma$ is a disk) the monodromy for $T_{(p,-1)}$ preserves $\Sigma\cap C,$ where $C$ is the unknot used in the construction above (and Lemma~\ref{lem:cable}). More precisely, we can think of $\Sigma\cap C$ as $p$ points sitting equally spaced on a circle about the center of the disk $\Sigma$ and the monodromy rotates this circle clockwise by $\frac{2\pi}{p}.$ Let $\gamma$ be an arc in $\Sigma$ that separates $\Sigma\cap C$ into two non-empty sets. Denote the page of the open book for $L_{(p,-1)}$ by  $\Sigma'$ and notice that $\Sigma'$ is obtained by removing disjoint disks about each point in $\Sigma\cap C$ and gluing copies of the page of $L$ in their places. One easily sees that if the page of $L$ is not a disk then $\gamma$ on $\Sigma'$ is not right veering and thus $\xi_{(p,-1)}$ is overtwisted. 
\end{proof}

\subsubsection{Relative open books}
In order to analyze negative cables of rational open books we need to recall the notion of a relative open book from  \cite{VanHornMorrisThesis}. A \dfn{relative open book decomposition} for a manifold $M$ with torus boundary components is a pair $(L,\pi)$ where $L$ is an oriented link in $M$ and $\pi:(M-L)\to S^1$ is a locally trivial fibration such that the closure of each fiber approaches each component of $L$ as a longitude. Notice that the fibration induces a fibration of each torus boundary component of $M$ by circles. We say a contact structure $\xi$ is compatible with $(L,\pi)$ if there is a contact form $\alpha$ for $\xi$ that is positive on vectors tangent to $L$ in the direction of the orientation of $L,$ $d\alpha$ is a positive volume form on each page of the open book and the characteristic foliation of $\xi$ on each boundary component of $M$ agrees with the foliation of by circles induced by $\pi.$ A slight generalization of this allows for the characteristic foliation on the boundary and the foliation induced by $\pi$ to differ but then the characteristic foliation should be linear and the Reeb vector field for $\alpha$ should leave the boundary invariant and be positively transverse to both the characteristic foliation and the foliation given by $\pi.$

One may easily check, see Proposition 3.0.7 in \cite{VanHornMorrisThesis}, that two boundary components of a relative open book can be glued together by an orientation reversing diffeomorphisms that preserves the fibration on the boundary components being glued. There is a natural (possibly still relative) open book  on the manifold obtained by the gluing, and it supports the natural contact structure obtained via the gluing. A slight generalization of this will allow us to glue relative open books and adapted contact structures along torus boundary components using diffeomorphisms that preserve both the characteristic foliations and foliations by the pages of the open book. 

\subsubsection{A few words about slope conventions.} \label{sec:slopeconventions}
We will need to use the classification of tight contact structures on solid tori and $T^2\times[0,1]$ from \cite{Honda00a}.   Unfortunately, as mentioned earlier, the convention used there (and in much of contact geometry) for denoting curves on tori and their slopes does not agree with the one used in this paper (see the introduction) and in the Dehn surgery literature.    Given a pair of oriented curves $\mu$ and $\lambda$ forming a positive basis for $T^2$, the convention used by Honda has a $(p,q)$--curve representing $p[\mu]+q[\lambda]$ with slope $q/p$ whereas our convention (used by Rolfsen for cables \cite{Rolfsen}) has a $(p,q)$--curve representing $p[\lambda]+q[\mu]$ with slope denoted $q/p$ as well.  We will use Rolfsen's convention.  To translate to Honda's convention, take the reciprocal of the slopes and then swap the letters.  Thus in translation, a $q/p$ slope is again a $q/p$ slope, but $0$ and $\pm\infty$ are swapped and inequalities flip.  

Following Rolfsen's convention, assuming our curves always have homology class with positive $[\lambda]$ coefficient,   we have the benefit of working with slopes from $-\infty$ to $\infty$ where a slope $q/p$ is positive with respect to a slope $s/r$ if $q/p>s/r$.  It does cause  $\mu$ to have slope $\infty$ and $\lambda$ to have slope $0$ (and any other longitude to have integral slope), but this is common when viewing $\mu$ as a meridian of a solid torus.  The main drawback is that when viewing $T^2$ as $\R^2/\Z^2$ associating $\mu$ and $\lambda$ with the $x$-axis and $y$-axis of $\R^2$ our slopes are ``run over rise'' and hence a ``horizontal'' curve has slope $\infty$ while a ``vertical'' curve has slope $0$.  

With our choice of convention it is convenient to orient the labeling of the Farey tessellation so that $-\infty, -1, 0$ appear in clockwise order as in Figure~\ref{fig:farey}.  This orientation is opposite what is common, but with our convention for slopes it permits us to preserve the notion of the slopes of three curves being in a clockwise order.

\begin{figure}[ht]
\relabelbox \small
{\epsfysize=2truein\centerline {\epsfbox{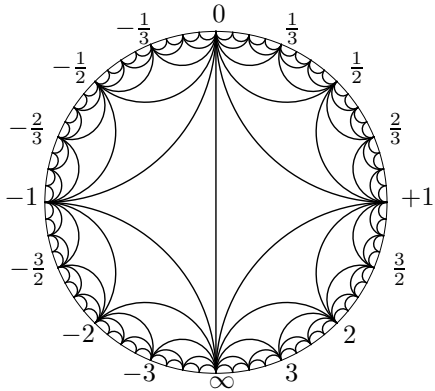}}}
\relabel{inft}{$\infty$}
\relabel{+1}{$+1$}
\relabel{2}{$2$}
\relabel{3}{$3$}
\relabel{3/2}{$\frac 32$}
\relabel{1/2}{$\frac 12$}
\relabel{1/3}{$\frac 13$}
\relabel{2/3}{$\frac 23$}
\relabel{0}{$0$}
\relabel{-1}{$-1$}
\relabel{-2}{$-2$}
\relabel{-3}{$-3$}
\relabel{-3/2}{$-\frac 32$}
\relabel{-1/2}{$-\frac 12$}
\relabel{-1/3}{$-\frac 13$}
\relabel{-2/3}{$-\frac 23$}
\endrelabelbox
	\caption{The Farey tesselation oriented for use with our convention of slopes.}
	\label{fig:farey}
\end{figure}

\subsubsection{Contact structures on solid tori and $T^2 \times I$.}\label{exception-section}
We now briefly recall the classiÞcation of tight contact structures on solid tori and $T^2 \times [0,1]$ from \cite{Honda00a}.

To this end we begin by discussing continued fractions. 
Given a rational number $-1< \frac sr < 0$ let  $[r_0,\ldots, r_k]$ be the continued fraction representation of $\frac sr$ such that  
\[
\frac sr =\cfrac{1}{r_0-\cfrac{1}{r_1-\cfrac{1}{\ddots -\cfrac{1}{r_k}}}}
\]
with $r_i\leq -2$ for each $i$.
 
Using the convention that $[r_0,\ldots, r_{j},-1]=[r_0,\ldots, r_{j}+1]$,
we define the set of {\em exceptional cabling slopes associated to $\frac sr$} as follows: Start with $e_1=[r_0,\ldots, r_k+1]$. Once $e_l$ is defined then $e_{l+1}$ is defined by adding 1 to the last term in the continued fraction expansion of $e_l$. Stop at $e_n = [-1] = -1$.  This produces a finite set of rational numbers $e_1,\ldots, e_n$.    Notice that the $e_i$ are precisely the vertices in a minimal path from $-1$ to $\frac sr$ in the Farey tessellation (not including the vertex $\frac sr$).  The set of exceptional cabling slopes for an integral open book (with Seifert slope $0$) may be taken to consist of  just $-1$.  Finally, an {\em exceptional cable} of a binding component of a rational open book with Seifert slope $\frac sr$ with $-1< \frac sr <0$ is a $(p,q)$--cable with $p,q$ coprime where $\frac qp$ is an exceptional cabling slope associated to $\frac sr$.

\begin{rem}\label{rem:ratlunknot}
If the Seifert slope $\frac sr$ of a rational unknot shares an edge in the Farey tessellation with $\frac qp$, then its $(p,q)$--cable with $p,q$ coprime (so that $rq-ps = \pm1$) gives another rational unknot.  See Example~\ref{ex:ratlunknot} (2).  
\end{rem}

Suppose $\xi$ is a tight contact structure on a solid torus such that the boundary is convex with two dividing curves of slope $ -1< \frac qp \leq 0$.  Take a shortest clockwise path in the Farey tessellation from $-1$ to $\frac qp$ and notice that the vertices in this shortest path are precisely the exceptional slopes $e_1,\ldots, e_n$ associated to $\frac qp.$  A tight contact structure on the solid torus is determined by a choice of sign on each edge. (Notice that some of these contact structures may be the same due to ``shuffling'' in a continued fraction block, but this will not be important for us in this paper).  Similarly if $\xi'$ is a tight (minimally twisting) contact structure on $T^2\times[0,1]$ with convex boundary, each torus $T^2\times\{i\}$ having two dividing curves of slope $s_i,$ with $s_0\not= s_1,$ then signs assigned to the shortest clockwise path in the Farey tessellation from $s_0$ to $s_1$ determine a tight contact structure on $T^2\times [0,1]$.  Now suppose $s_0=\frac qp$ and we glue the contact structures $\xi$ and $\xi'$ together. We now have a path from $-1$ to $s_1$ in the Farey tessellation obtained by concatenating the paths corresponding to $\xi$ and $\xi'.$ The resulting contact structure will be tight if and only if while shortening this path to a minimal path from $-1$ to $s_1$ we never have to merge two paths with different signs. 

\begin{lem}\label{exceptional}
Let $(B,\pi)$ be the relative open book on $S=S^1\times D^2$ with binding $B=S^1\times\{(0,0)\}$ and pages annuli of slope $-1 < \frac sr \leq 0$.
  Let $(p,q)$ be a pair of integers such that $p>0$ and $-1 < \frac qp < \frac sr$. 
 
  Let $\xi$ be the contact structure supported by $(B,\pi)$ and assume that $\partial S$ has linear characteristic foliation with slope $\frac {s'}{r'}$ with $\frac qp <\frac{s'}{r'}\leq\frac sr$. (We can perturb $\partial S$ in a $C^\infty$-small manner to arrange it to be convex with two dividing curves of slope $\frac {s'}{r'}.$)
Let $(B', \pi')$ be the open book obtained by cabling $(B,\pi)$ and $\xi'$ the contact structure supported by it.  After perturbing the boundary of the solid torus so that it is convex with respect to $\xi$ and $\xi'$ the contact structure $\xi'$ can be described as follows: take a shortest path in the Farey tessellation from $-1$, clockwise, to $\frac {s'}{r'}$ that passes through $\frac qp$. The contact structure $\xi$ is described by putting a $+$ on each jump. The contact structure $\xi'$ is described by putting a $-$ on all jumps from $-1$ to $\frac qp$ and a $+$ on the rest. 
\end{lem}

\begin{rem}
As an example, the shortest path from $-1$ to $-\frac 13$ does not go through $-\frac23$, so the $(3,-2)$--cable of a $(3,-1)$--open book yields an overtwisted contact structure, since the path in the Farey tessellation describing this cable can be shortened along edges with inconsistent signs. 
\end{rem}

\begin{proof}
We begin by considering a neighborhood $N$ of the boundary of $S$ and the topology of the open book in this region. Recall that the cable of $B$ can be taken to be a (multi)curve on a torus contained in $N$. We will construct an open book for $S$ with binding given by the cable of $B$ and so that the pages agree with those of $(B,\pi)$ on $\partial S$. We do this by first constructing the open book on $N$ with binding the cable of $B$ that extends over $S$. We then construct the contact structure supported by this open book in several steps and observe that is the contact structure described in the lemma.

\smallskip
\noindent 
{\bf Step I} --- {\em Construct an open book on $N$ with binding the $(p,q)$--cable of $B$.}
Think of $N$ as an annulus $A$ times $S^1$ where $A$ is an annulus on the disk $D^2$.  We can break $A$ into three successively larger annuli $A_1$, $A_2$ and $A_3$. The open book $(B,\pi)$ in each of the regions $A_i\times S^1$ is a foliation by annuli of slope $\frac sr$.   To construct $(B',\pi')$ we foliate $A_3\times S^1$ by annuli of slope $\frac sr$ (oriented so that they intersect the $S^1$-fibers positively) and $A_1\times S^1$ by annuli of meridional slope $\infty$ (oriented so they intersect the $S^1$-fibers negatively). Finally, $A_2\times S^1$ can be though of as the union of two solid tori $S_1\cup S_2$ each of which has slope $\frac qp$.  We foliate $S_1$  by meridional disks so that the disks in $S_1\cap (A_i\times S^1),$ for $i=1,3,$ agree with the foliation already defined on $A_i\times S^1.$  Thus we have fibered the complement of the neighborhood $S_2$ of a $(p,q)$--curve in $A\times S^1$.  This fibration on $\partial (S_2=S^1\times D^2)$ (the product structure is induced by the framing induced by the torus on which the $(p,q)$--curves sits) is by curves of slope $\frac lk>0.$ We can cone these curves to the core of $S_2$ to obtain an open book decomposition for $A\times S^1$ with binding a $(p,q)$--curve  and pages intersecting the boundary in curves of slope $-\infty$ (co-oriented downward) and slope $\frac sr$ (co-oriented upwards). (We notice that if the $(p,q)$--cable has multiple components then we actually have to break $A_2\times S^1$ into $2\gcd(p,q)$--solid tori. To avoid confusing notation we ignore this issue in the rest of the proof but note that none of the arguments are affected by this omission.)

Turning this construction around we can start with an open book for the solid torus $S_2$ with the core curve being the binding and the pages having slope $\frac lk>0.$ We can then glue $S_1,$ foliated by meridional disks, to $S_2$ (along a pair of annuli that are longitudinal on both $S_1$ and $S_2$) so that the pages of the open book on $S_2$ are extended over $S_1$ by the meridional disks.  We then can trivially extend this open book over $(A_1\cup A_3)\times S^1$ by annuli. This results in the same open book decomposition for $A\times S^1.$  

\smallskip
\noindent
{\bf Step II} --- {\em Construct a contact structure on $A_2\times S^1$ supported by the cabled open book.}
For convenience, first apply an orientation preserving diffeomorphism to $T^2\times [0,1]=A\times S^1$ so that the $(p,q)$--curves are vertical, recall this means their slope is 0.  (In particular, this makes the solid tori $S_1$ and $S_2$ vertical too.)  We may choose this diffeomorphism  so that the slope $\frac sr$ becomes a positive number $t$ and the slope $\infty$ becomes a negative number $t'$.  The slope $\frac lk$ on $\bdry S_2$ is still $\frac lk$ since the framing on the $(p,q)$--curve is unchanged after the diffeomorphisms (since the framing is determined by the torus on which the curve sits).

To begin the construction of the contact structure, let $S_2=S^1\times D^2$ be a solid torus and let $(B',\pi')$ be an open book for $S_2$ with binding $B'=S^1\times\{(0,0)\}$ and annular page that foliates $\partial S_2$ by $\frac lk$--curves (recall $\frac lk>0$). This supports the contact structure that starts positively transverse to the binding and uniformly rotates clockwise to any slope $0<\frac{l'}{k'}<\frac lk$.  In particular, let $S'_2$ be a concentric solid torus within $S_2$ so that the characteristic foliation on $\partial S_2'$ is linear of longitudinal slope $0$. Perturb $S_2'$ slightly, still denoting it by $S_2',$ so that $\partial S_2'$ is broken into 8 vertical annuli $\hat A_1,\ldots,\hat A_8.$  See Figure~\ref{fig:absres0}. On $\hat A_{i},$ with $i$ even, the characteristic foliation has a single vertical singular set and horizontal ruling curves. On $\hat A_1$ and $\hat A_5$ the foliation is non-singular and each leaf has slope varying between two barely negative values. On each of $\hat A_3$ and $\hat A_7$ we have an inner sub-annulus $\hat A_i'$ on which the foliation has slope exactly $\frac {l''}{k''}$ where $\frac{l'}{k'}<\frac{l''}{k''}<\frac lk$ and on $\hat A_i-\hat A_i'$ the slope is between $0$ and $\frac{l''}{k''}$. 
\begin{figure}[ht]
  \relabelbox \small 
 {\epsfysize=1.5truein\centerline {\epsfbox{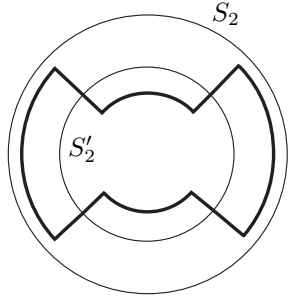}}} 
  \relabel{1}{$S_2'$} 
  \relabel {2}{$S_2$} 
  \endrelabelbox
        \caption{ The inside circle (times $S^1$) is the torus in $S'_2$ with linear vertical slope. The outside circle has some positive slope. The heavier line with corners rounded (times $S^1$) is the convex $\bdry S_2'$.} 
	\label{fig:absres0}
\end{figure}

Now let $S_1=S^1\times I\times I$ where $I=[0,1]$ is an interval. The boundary $\partial S_1$ is made of four longitudinal annuli $B_1,\ldots, B_4.$  Gluing $S_1$ to $S'_2$ so that $B_1$ maps to a sub-annulus of $\hat A_3$ and similarly for $B_3$ to $\hat A_7$ yields the manifold  $C=A_2\times S^1= T^2\times [0,1].$ We can choose the gluing maps as discussed above so that the open book decomposition on $S_2'$ extends (by meridional disks in $S_1$) to an open book decomposition for $C.$ Thinking of $S_1$ as $B_1\times [0,1],$ we can assume the slope of the pages of $S_2'$ on $B_1\times\{0\}$ is  $\frac lk$ and on $B_1\times \{1\}$ it is  $-\frac lk.$ Moreover the meridional disks that make up the pages in $S_1$ intersect $B_1\times \{t\}$ in lines of slope uniformly rotating (clockwise) from $\frac lk$ to $- \frac lk.$ The characteristic foliation on $B_1\times \{0\},$ respectively $B_1\times \{1\},$ induced by the contact structure on $S_2'$ and the gluing maps has slope  $\frac{l''}{k''},$ respectively  $-\frac{l''}{k''}.$ Thus we can extend the contact structure from $S_2'$ to $C$ so that it induces linear characteristic foliations on $B_1\times \{s\}$ of slope uniformly rotating (clockwise) from  $\frac{l''}{k''}$ to  $-\frac{l''}{k''}.$ One can easily see that this extended contact structure is supported by the open book on $C.$ (Notice that one can easily arrange the characteristic foliation on the meridional disks of $S_1$ to contain a single hyperbolic singularity.)

Note that (1) $C$ is a $T^2\times [0,1],$ (2) the binding of the open book is a vertical $S^1$, (3) the dividing slope on both boundary components of $C$ is vertical (since the vertical dividing curves on $S_2$ give dividing curves on both boundary components of $C$),  and (4) a page of the open book intersects $\partial (T^2\times[0,1])$ in curves of slope $t'$ on $T^2\times\{0\}$ and $t$ on $T^2\times\{1\}.$

\smallskip
\noindent
{\bf Step III} --- {\em Identify the contact structure on $A_2\times S^1$.}
We claim the contact structure on $C=A_2\times S^1$ supported by this open book is tight. Moreover the contact structure is non-rotative. To see this consider the $I$ invariant contact structure on $T^2\times I$ with two vertical Legendrian divides and ruling slope some small positive number on, say, $T_0 = T^2 \times \{0\}$. If $\overline A$ is a vertical annulus on $T_0$ that contains both the Legendrian dividing curves and $I'$ is a sub-interval of $I$ that is contained in the interior of $I$ then one can round the corners of $\overline A\times I'$  to obtain a solid torus contactomorphic to $S_2'.$ One may also easily check that $\overline A$ and $I'$ can be chosen so that $(T_0\setminus \overline A)\times I''$ (where $I''$ is a sub-interval of $I'$) is contactomorphic to $S_1.$ Thus we can embed $C$ into an $I$ invariant contact structure on $T^2\times I$ with vertical dividing curves on each $T^2 \times \{0\}$ and $T^2 \times \{1\}$. This establishes our claim.   

\smallskip
\noindent
{\bf Step IV} --- {\em Glue the open book decomposition and contact structure on $A_2\times S^1$ to $A_3\times S^1$ and a solid torus and observe that this is the contact structure described in the lemma.} 
From the discussion above we would like the characteristic foliation on the boundary of our manifolds to be linear.
To this end we see how to slightly extend the contact structures near the boundary of $C=A_2\times S^1$ so that they become ``rotative'' and in particular have linear characteristic foliations. 

Notice that $S_2'$ sits inside a solid torus $S_2''$ that has a linear foliation on its boundary of some very large positive slope. We can use $S_2''$ to enlarge $C$ to $C'=T^2\times [0,1]$ so that on $T^2\times \{0\}$ we have a linear characteristic foliation of slope between $t'$ and $0$ and on $T^2\times \{1\}$ the characteristic foliation is also linear and of slope between $0$ and $t$. The open book constructed above can easily be seen to support this contact structure too. 

Using the inverse of the diffeomorphisms of $T^2$ discussed above we can convert the above construction to give a contact structure $\xi'$ on $T^2\times [0,1]$ and an open book decomposition $(B',\pi')$ with the following properties: (1) the binding $B'$ has slope $\frac qp,$ (2) the pages intersecting $T^2\times \{0\}$ have slope $-\infty$ with downward co-orientation, (3) the pages intersect $T^2\times\{1\}$ have slope $\frac sr$ with upward co-orientation,  and (4) on tori $T_x=T^2\times \{x\}$ near the boundary of  $T^2\times [0,1]$ the characteristic foliations are linear with slopes between $-\infty$ and $\frac qp$ near $T_0$ and between $\frac qp$ and $\frac sr$ near $T_1$.  Moreover, on a subset $T^2\times[\epsilon, 1-\epsilon]$ the contact structure is non-rotative and both boundary components have dividing curves of slope $\frac qp$.  

We can now take the contact structure on $S^1\times D^2$ that is radially symmetric and rotates to a slope near $\frac qp$ on the boundary and notice that it is ``supported'' by the fibration by meridional disks. If we reverse the orientation on the contact planes and the meridional disks we can glue this to the contact structure on $C'$ above. 
From this we see that the contact structure on the solid torus now bounded by $T_\epsilon$ is described by a shortest path in the Farey tessellation from $\infty$ to $\frac qp$ with $-$ signs on all the jumps (except the first). (In other words it is the standard contact structure on the solid torus with its orientation reversed.)

Gluing $A_3\times S^1$ to the boundary of the above contact structure on $S^1\times D^2$ we get a new contact structure on the solid torus which is easily seen to be the one described in the lemma and is supported by the open book decomposition obtained by $(p,q)$--cabling the core of the solid torus.
\end{proof}

To apply Lemma~\ref{exceptional} to the proof of Theorem~\ref{thm:main-cable} we need to understand neighborhoods of binding components. 
\begin{lem}\label{get-good-nbhd}
Let $(L,\pi)$ be an open book decomposition and $K$ a component of $L.$ Choose a framing of $K$ so that so that the Seifert slope is $-1<\frac sr\leq 0.$  If the page of $(L,\pi)$ is not a disk then there is a neighborhood $N$ of $K$ so that $N$ has convex boundary with two dividing curves of slope $\frac sr$ and the contact structure induced on $N$ is determined by only positive jumps in the Farey tessellation. Moreover given any $\frac{s'}{r'}$ such that $-1<\frac {s'}{r'}<\frac sr,$ there is another neighborhood $N'\subset N$ of $K$ whose boundary has linear characteristic foliation of slope $\frac{s'}{r'}$ and the intersection of the open book with $N'$ gives a relative open book on $N'$ that supports the contact structure on $N'$ and the intersection of the open book with the complement of $N'$ gives a relative open book that supports the contact structure there.
\end{lem}
\begin{rem}
We notice that this theorem constructs a large ``standard neighborhood'' of a binding component. Understanding the size of a standard neighborhood of a transverse (or Legendrian) knot has been an important, recurring theme in contact geometry, see \cite{EtnyreHonda05, EtnyreLafountainTosun, Gay02a}. Such considerations have also become important in higher dimensional contact geometry \cite{NiederkrugerPresas10}. 

We also notice that the hypothesis that the page is not a disk is clearly necessary since if such a neighborhood of the binding could be constructed in this case then one would have an overtwisted disk on one of the pages of the open book. 
\end{rem}

\begin{proof}
Let $P$ be a page of the open book and let $A$ be a small neighborhood of the boundary of $P.$ Set $P'=\overline{P-A}.$ Since there is a Reeb vector field $v$ transverse to the pages of the open book we see that $P'$ is convex. Now let $P'\times[-\epsilon, \epsilon]$ be an invariant neighborhood of $P'$ constructed using the flow of $v.$ One can round the corners of this neighborhood so to obtain a convex surface $\Sigma$  where $\Sigma=P'\times\{\epsilon\}\cup P'\times{-\epsilon}\cup B$ where $B$ is a union of annuli. Moreover the dividing set can be taken to be $B\cap P$ and $P'\times\{\pm\epsilon\}$ is contained in the $\pm$-region of the convex surface. 

On $P'\times\{\epsilon\}\subset \Sigma$ we can Legendrian realize a curve $L'$ parallel to the $(r,s)$--cable of $K.$ If $L\not=K$ then we can use the standard Legendrian realization principle \cite{Honda00a}. If $L=K$ then we must first Legendrian realize another non-separating curve on  $P'\times\{\epsilon\}\subset \Sigma.$ Now using a local model for this realized curve we can ``fold'' the surface to create a new convex surface with two new dividing curves. We can now use the Legendrian realization principle on this new convex surface to realize $L'$ as a Legendrian curve. 

Notice that we can use the annulus on $\Sigma$ that $L'$ and a dividing curve cobound and an annulus on $P$ to create an annulus that $L'$ and $K$ cobound and that contains no singular points and no closed leaves other than $L'.$ Let $N$ be a neighborhood of this annulus and having $L'$ on its boundary. We can build a model of this annulus in $S^1\times \R^2$ with contact structure $d\phi + r^2\, d\theta$ (where $\phi$ is the coordinate on $S^1$ and $(r,\theta)$ are polar coordinates on $\R^2$). Moreover we can identify the annulus with an annulus in $S^1\times \R$ so that $K$ maps to $K'=S^1\times \{(0,0)\}.$ Thus we can assume that $N$ is contactomorphic to a neighborhood of $K'.$ This implies that the contact structure on $N$ is universally tight. Now making the boundary of $N$ convex and possible taking a sub-torus of $N$ we can assume that $\partial N$ has just two dividing curves and they have slope $\frac s r$ (that is, they  are parallel to $L'$). The classification of contact structures on solid tori implies that $N$ is the neighborhood claimed in the lemma. 

Since $N$ has the unique universally tight contact structure on the torus one may easily use a standard model for $N$ to construct $N'.$
\end{proof}

We can now prove item (3) and (4) in Theorem~\ref{thm:main-cable} in complete generality.
\begin{proof}[Proof of Statements (3) and (4)  in Theorem~\ref{thm:main-cable}]
We begin by assuming that the page of our open book is not a disk. (The case where the page is a disk is dealt with below.) Notice that Lemma~\ref{get-good-nbhd} allows us to assume $\frac{s'}{r'}=\frac{s}{r}$ in Lemma~\ref{exceptional} for the purposes of determining the effect of cabling on the contact structure. 

Given a binding component $K$ of an open book $(L,\pi)$ and a framing on $K$ chosen so that the pages of the open book approach $L'$ as $(r,s)$--curves  where $-1<\frac sr \leq 0$ then choose a pair of integers $(p,q)$ such that $\frac qp <\frac sr$.   We have neighborhoods $N$ and $N'$ given in Lemma~\ref{get-good-nbhd}. Applying Lemma~\ref{exceptional} to $N'$ with its relative open book induced from $(L,\pi)$ we see the effect of cabling on the contact structure restricted to $N'.$ From this we also see the effect for the contact structure on $N.$ 
Thus it is clear from Lemma~\ref{exceptional} that the result of $(p,q)$--cabling $K$ will be to replace the contact structure on $N$ with the one described in the lemma. If $\frac qp$ is not an exceptional cabling slope then we may shorten the path that describes the contact structure on $N$ by removing the $\frac qp$-vertex, but since the signs describing the contact structure changed at $\frac qp$ the contact structure must be overtwisted. 

To see that all negative cables are virtually overtwisted, except for the $(p, q)$--cables of a rational unknot of Seifert slope $\frac sr$ with $rq-ps=-1$, we notice that the solid torus neighborhood of the binding component $L'$ can be unwrapped in a cover of the manifold until the slope of the page becomes longitudinal while the cabling remains negative.   
One may easily see that this makes the contact structure on the solid torus overtwisted. By taking a further cover if needed so that the lift of the open book is integral, Lemma~\ref{lem:xipqOT} applies too.

Now consider the case when the page of the open book is a disk. In this case the manifold is a lens space with universally tight contact structure.  We cannot apply Lemma~\ref{get-good-nbhd} above to get a nice neighborhood of the binding; however, if we perform a $(p,q)$--cable of the binding with $rq-ps=-1$ then one may easily check in the Farey tessellation that the resulting path given in Lemma~\ref{exceptional} cannot be shortened. Moreover, as the contact structure on a lens space is determined by its restriction to the solid torus used in the statement of Lemma~\ref{exceptional} we see that the effect on the contact structure is to reverse the co-orientation. 
\end{proof}

\subsection{Lutz twists and Homotopy classes of negative cablings}\label{sec:part3planefield}\label{sec:lutz} 
We will identify the overtwisted contact structure supported by the open book of a negative cable with a modification of the original contact structure by Lutz twists along the component being cabled and along its cable.  This will facilitate an understanding of how the Hopf invariant changes under cabling too. (We note that this analysis also determines the change in the homotopy type of supported contact structures when an exceptional cabling is done to a binding component.)

Let us recall and discuss Lutz twists.  A transverse simple closed curve $K$ in a contact manifold $(M,\xi)$ has a neighborhood $N=S^1\times D^2$ with coordinates $(\phi, (r, \theta))$, where $(r,\theta)$ are polar coordinates on the disk $D^2$ of radius $c>0$ in the plane, on which the contact form can be written $d\phi+r^2 \, d\theta$.  Fix $0<\delta$ such that $4\delta \ll c$.  

Define functions $f(r)$ and $g(r)$ on the interval $[0,c]$ satisfying $g'f-f'g>0$  and such that
\[
f(r) = 
\begin{cases} -1 & r \in [0,\delta]\\ \cos (\pi(\frac{r-\delta}{c-2\delta})+\pi) & r \in [2\delta , c-2\delta]\\ 1 & r \in [c-\delta, c] \end{cases},
\]
and
\[
g(r) =  \begin{cases} -r^2 & r \in [0,\delta]\\ r^2 \sin (\pi(\frac{r-\delta}{c-2\delta})+\pi) & r \in [2\delta , c-2\delta]\\ r^2 & r \in [c-\delta, c] \end{cases}.
\]
(Note we really think of the functions as defined by the trigonometric functions on all of $[\delta, c-\delta]$ but altered on $[\delta,2\delta]\cup[c-2\delta, c-\delta]$ to make it smooth.)  The \dfn{Lutz twist} of $\xi$ along $K$ is the contact structure $\xi_{K}^{\lutz}$ obtained from $\xi$ by changing the contact form from $d\phi+ r^2\, d\theta$ to $f(r)d\,\phi+ g(r)\, d\theta$ on $N$.  In effect, a Lutz twist introduces a half rotation to the contact planes in $N$ as one travels radially inward to $K$.   This is also known as a half Lutz twist or a $\pi$--Lutz twist.  

When dealing with plane fields that are either foliations or contact near a curve transverse to the plane field one can add a positive or negative twist along the curve. We will call these positive and negative Lutz twists (even though they are not technically Lutz twists).

Notice that there is a radius $r_0 \in [0,c]$ such that the leaves of the 
characteristic foliation of the torus  $\{r=r_0\}$ induced from $\xi_{K}^{\lutz}$ are meridional. The $(p,q)$--curve on this torus is transverse to the characteristic foliation and, with orientation induced by the contact planes, wraps around $N$ in the opposite direction as $K$.  We define a \dfn{$(p,q)$--Lutz cable} of $K$ to be this transverse link  in $\xi_{K}^{\lutz}$.

\begin{lemma} \label{lem:OTlutztwist} 
Let $\xi$ be the contact structure supported by the open book $(L,\pi).$
If $\frac{q}{p} <\frac{s}{r}$, then the contact structure, $\xi_{(p,q)},$ compatible with $L_{(p,q)}$ is homotopic to the contact structure obtained from $\xi$ by a Lutz twist along the component $K$ to be cabled followed by a Lutz twist along each component of its Lutz cable $K_{(p,q)}$. 
\end{lemma}

\begin{proof} 
From the Thurston-Winkelnkemper construction (even the rational one), any contact structure compatible with an open book is homotopic to a standard positive confoliation (see \cite{EliashbergThurston} for information about confoliations) given by the foliation of the fibers matched to the rotative positive contact structure in tubular neighborhoods of the binding.  As shown in Lemma~\ref{lem:main-cable-lemma2}, a positive cable of a binding component of an open book induces an open book supporting a contact structure identical to the original outside a neighborhood of the binding and isotopic to the original in this neighborhood.  Similarly,  the standard positive confoliation of a positive cable is homotopic (actually confoliation-isotopic) to the standard positive confoliation of the original open book.  By mirroring, we have the same statements for negative cables and negative confoliations.

Given the standard positive confoliation associated to an open book, applying negative Lutz twists along each binding component produces a plane field homotopic to the standard negative confoliation.   Similarly, applying positive Lutz twists along each binding component of the standard negative confoliation of an open book produces a plane field homotopic to the standard positive confoliation.  Also, recall that the plane field obtained from performing two positive Lutz twists along the same curve is homotopic to the original as is performing a positive and a negative Lutz twist along the same curve.  Thus the plane field obtained from a positive Lutz twist is homotopic to the plane field obtained from a negative Lutz twist.

Now consider the standard positive confoliation $\eta$ associated to the open book $(L,\pi)$ that is the foliation of the pages outside tubular neighborhoods of the binding components.  Perform a (positive) Lutz twist on $K.$ The plane field is homotopic to a plane field $\eta'$ associated to $(L,\pi)$ by taking the foliation given by the fibers of the open book in the complement of a neighborhood of $L,$ matched to a positive rotative contact structure in a neighborhood of $L-K$ and a negative rotative contact structure in a neighborhood of $K.$ Moreover the cable $K_{(p,q)}$ of $K$ can be kept transverse to the plane fields throughout the homotopy from $\eta$ after the Lutz twist to $\eta'.$  From above this $\eta'$ is homotopic to a plane field $\eta''$ that is a negative rotative contact structure near $K_{(p,q)},$ a positive rotative contact structure near $L-K$ and a foliation by fibers of the cabled open book elsewhere. Now perform a positive Lutz twist along $K_{(p,q)}.$ This yield a plane field $\eta'''$ that is homotopic to the standard positive confoliation associated to $L_{(p,q)}$ and hence to the contact structure $\xi_{(p,q)}$ supported by $L_{(p,q)}.$ 
\end{proof}

\subsection{Contact structures of cabled open books}
\label{sec:thm}

\begin{proof}[Proof of Theorem~\ref{thm:main-cable}]
Assume $(L,\pi)$ is an open book supporting $(M,\xi)$ with binding components $L_1, \dots, L_n$. 
Let $({\bf p}, {\bf q}) = ((p_1,q_1), \dots, (p_n,q_n))$ be $n$ pairs of integers such that all the $p_i$ have the same sign and the slope $\frac{q_i}{p_i}$ is neither the meridional slope nor Seifert slope of $L_i$.
 Form the cable $L_{({\bf p}, {\bf q})}$ successively, taking the $(p_i,q_i)$--cable of the $i$th component of $L$ one at a time.  Let $\xi_{({\bf p}, {\bf q})}$ be the contact structure it supports.

Assume all the $(p_i,q_i)$ are positive.   If the $p_i$ are all positive, then apply Lemma~\ref{lem:main-cable-lemma2} upon cabling each component of $L$ to show that the open book at each step continues to induce the contact structure $\xi$.  This proves (1).  If the $p_i$ are all negative, apply the same construction to the open book $(-L,-\pi)$ which supports $-\xi$.  With respect to this new open book the corresponding cabling uses the integers $-p_i$ and $-q_i$ so that $L_{({\bf p},{\bf q})}$  supports $-\xi$.  This proves (2).

The case when $(p_i,q_i)$ is negative for some $i$ is dealt with at the end of Subsection~\ref{sec:part3OT}.

Since overtwisted contact structures are determined by their homotopy class (Eliashberg's theorem \cite{Eliashberg89}), applications of Lemma~\ref{lem:OTlutztwist} to the each of the negative cablings completes the proof.
\end{proof}

\begin{proof}[Proof of Corollary~\ref{cor:integralcable}]
Parts (1) and  (2) of the corollary are clear form Theorem~\ref{thm:main-cable} when $pq>0.$  If $p=1$ then $K_{(p,q)}$ is isotopic to $K$ so the corollary follows in this case too.  For parts (3) and (4) notice that Theorem~\ref{thm:main-cable} implies the contact structure is overtwisted and obtained from $\xi$ or $-\xi$ from performing Lutz twists. Since the binding is null-homologous, all the Lutz twists performed are on null-homologous curves and thus the $\text{spin}^c$ structure of the contact structure is unaffected, so we only need to see how the Hopf invariant of the resulting contact structure compares to $\xi$ or $-\xi$.  Let us focus on part (3) and $\xi$ as part (4) is similar.

Let $\xi'$ be obtained from $\xi$ by a Lutz twist along $K$, viewed as the transversal binding of the open book.  By Lemma~\ref{lem:OTlutztwist}, $\xi_{(p,q)}$ is homotopically obtained from $\xi'$ by a Lutz twist along $K_{(p,q)}$.   Since a Lutz twist followed by another Lutz twist on the (now orientation reversed) core is homotopic to the identity, $\xi'$ is homotopically obtained from $\xi_{(p,q)}$ by a Lutz twist along $K_{(p,q)}$, viewed as the transversal binding of the cabled open book.  Because $\xi'$ and $\xi_{(p,q)}$ are both overtwisted, this is actually an isotopy.

Section~4.3 of \cite{Geiges08} shows that a Lutz twist on a transversal knot adds the self-linking number of the knot to the Hopf invariant of the contact structure.
In the contact structure supported by an integral open book with connected binding, the binding is naturally transversal and has self-linking number equal to the negative of the Euler characteristic of its page.  
Hence in obtaining $\xi'$ from $\xi$, we add $-\chi(K)$ to the Hopf invariant of $\xi$. Similarly, in obtaining $\xi'$ from $\xi_{(p,q)}$ we add $-\chi(K_{(p,q)})$ to the Hopf invariant of $\xi_{(p,q)}$.  Therefore, passing from $\xi$ to $\xi_{(p,q)}$, we add $(-\chi(K)) - (-\chi(K_{(p,q)})) = -\chi(K) + |p|\chi(K)+|q|-|pq| = (1-|p|)(-\chi(K)+|q|)$ to the Hopf invariant of $\xi$.  Letting $g$ be the genus of $K$, the Hopf invariant changes by $(1-|p|)(2g+|q|-1)$ as we were to show.
\end{proof}

We now prove two of our propositions that determine what happens at exceptional surgeries. 

\begin{proof}[Proof of Proposition~\ref{prop:ratunknots}]
The rational unknot in $L(p,q)$ supports the contact structure obtained by gluing together two solid tori. More precisely, $L(p,q)$ is obtained from $T^2\times [0,2]$ by collapsing the $\infty$ curves on $T^2\times\{0\}$ and the $-q/p$ curves on $T^2\times \{2\}$. The contact structure on $L(p,q)$ is tangent to the $[0,2]$ factor and rotates from $-\infty$ to $-q/p$. We can split $L(p,q)$ along $T\times\{1\}$ and assume the characteristic foliation on this torus is by curves of slope $-1$. Now perturb the torus to be convex. The contact structure on the solid torus with $\infty$ meridians is unique. The contact structure on the other torus is described by only positive jumps in the Farey tessellation description. An exceptional cable will give two solid tori glued together with a combination of positive and negative  jumps in the Farey tessellation description of the second solid torus. It is well known that all of these are precisely the tight contact structures on $L(p,q)$ as described in \cite{Honda00a}.
\end{proof}

\begin{proof}[Proof of Proposition~\ref{prop:otexceptional}]
Given the hypothesis of the Proposition it is clear that the path in the Farey tessellation describing the contact structure on a solid torus neighborhood containing the cabled binding component can be shortened so that positive and negative edges must be merged.   (Compare with the proof of Statement (3) in Theorem~\ref{thm:main-cable}.)

For the second statement on the proposition one may observe that when one negatively $(p,q)$--cables a binding component with $(p,q)$ not relatively prime then the change in the relative Euler class on a neighborhood of the binding component results in a relative Euler class that is not the relative Euler class of a tight contact structure on a solid torus \cite{Honda00a}. (One may also explicitly locate an overtwisted disk in the cabling solid torus.)
\end{proof}

Our last result, Proposition~\ref{prop:tightneg} will be proven in Section~\ref{exceptionalsurg} below once we consider monodromies of cables.

\section{Integral resolution of a rational open book}\label{sec:resolution}

Suppose $K$ is a rational (and non-integral) binding component of an open book $(L,\pi)$ for a manifold $M$ whose page approaches $K$ in a $(r,s)$--curve, $r>1$.   
  Set $n = \gcd(r,s)$.  We say $K$ has \dfn{order} $r$ and \dfn{multiplicity} $n$.  

For any $l \neq s$, replacing $K$ in $L$ by $K_{(r,l)}$, the $(r,l)$--cable of $K$, gives a new link $L_{K_{(r,l)}}$ where the components of $K_{(r,l)}$ all have multiplicity $1$.   This is called the \dfn{$(r,l)$--resolution of $L$ along $K$}.  If $K$ were the only component of $L$, then the $(r,l)$--resolution of the rational open book would yield an integral open book.

Using the same analysis as in Lemma~\ref{lem:torusknotfiber} notice that in the resolution, the new fiber is created using just one copy of the old fiber.  Thus the data along the other binding components of the open book remain unchanged and we may continue to resolve the other boundary components in a similar way without affecting the boundary components that have already been resolved.

\begin{thm}[Resolution of rational open books]\label{thm:resolution}
Let $(L, \pi)$ be a rational open book decomposition of $M$ compatible with a contact structure $\xi$ and let $\{K_i \}$ be the collection of binding components, each of which has order $r_i$ greater than one.  
Choose framings on the $K_i$ so that the pages approach $K_i$ as a $(r_i,s_i)$--curve and choose integers $l_i>s_i.$
Then the open book decomposition $(L', \pi')$ obtained by $(r_i,l_i)$--resolving $L$ along $K_i$ is an integral open book, is also compatible with $\xi$, and agrees with $(L, \pi)$ outside of a neighborhood of $\{ K_i \}$.
\end{thm}
\begin{proof}
The fact that $(L',\pi')$ is an integral open book decomposition of $M$ that agrees with $(L, \pi)$ outside of a neighborhood of $\{ K_i \}$ follows from the discussion above. That the resolved open book $(L',\pi')$ supports the same contact structure follows from statement (1) of Theorem~\ref{thm:main-cable} since the $l_i>s_i.$
\end{proof}

\begin{Ex}\label{exs-1}
The easiest example of a resolution is when $s = -1$ (for some choice of longitude) and we do the $(r,0)$--resolution. 
Specifically let $K$ be a binding component of an open book decomposition $(L,\pi)$ of a manifold $M$ and choose a framing on $K$ so that the pages approach $K$ as $(r,-1)$--curve. There is a natural way to describe the abstract $(r,0)$--resolution of this open book given the abstract open book for $(L,\pi).$ Specifically if $\Sigma$ is a page of $L$ and the monodromy of $L$ is $\phi,$ then the page of the resolved open book $\Sigma'$ is obtained by gluing an $r$-punctured disk to $\Sigma$ along $\widetilde{K}\subset \bdry \Sigma$ (where $\widetilde{K}$  is the component of $\partial \Sigma$ that $r$--fold covers $K$) and composing the extension of $\phi$ to $\Sigma'$ with a positive Dehn twist about a curve parallel to each boundary component in the punctured disk (except for the component $K$). See Figure~\ref{fig:absres}.

To see that this is indeed the correct description of the resolution we notice that we can remove a neighborhood of $K$ from $M$ and reglue it to obtain a manifold $M'$ and an open book $(L',\pi')$ where the core of the reglued solid torus is a knot $K'$ contained in the link $L'$, $K$ has multiplicity 1, and $-r$--surgery on $K'$ yields $M$ and the open book $(L,\pi)$. 
(Then on the page $\Sigma$, $K' = \widetilde{K}$.) Notice that the surgery takes a meridian of $K'$ to the framing of $K$.
Stabilizing the open book $(L',\pi')$ by $r$ connect sums of positive Hopf bands to $K'$ produces a new open book with binding $L'$ union $r$ unknots linking $K'$ as meridional curves.  The new page gives $K'$ a framing $r$ less than the old page.  Thus $0$--surgery on $K'$ (using the framing from the new page) returns $M$.  Since the $r$ unknots from the stabilization are isotopic to meridian of $K'$, after the surgery they are $r$ parallel copies of the framing of $K$ and hence form the $(r,0)$--cable of $K$.  Abstractly, the stabilization of $(L',\pi')$, effectively adds $r$ punctures to $\Sigma$ near $K'$ with a positive Dehn twist around each. The $0$--surgery then caps off the boundary component $K'$.  This produces the open book as claimed above.
\end{Ex}

\begin{figure}[ht]
  \relabelbox \small 
 {\epsfysize=1.5truein\centerline {\epsfbox{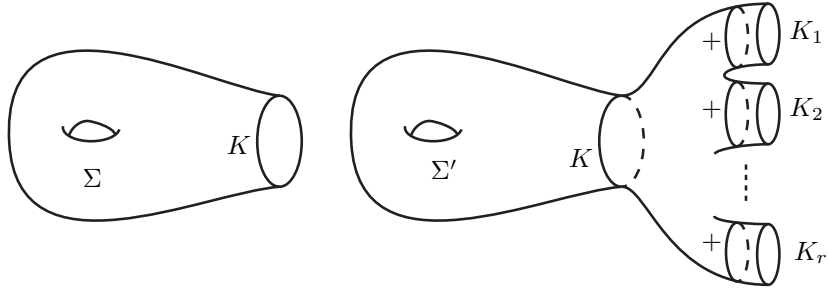}}} 
  \relabel{1}{$K_1$} 
  \relabel {2}{$K_2$} 
  \relabel{l}{$K$} 
  \relabel {s}{$\Sigma$} 
  \relabel{r}{$K_r$} 
  \relabel {k}{$K$} 
  \relabel{p1}{$+$} 
  \relabel {p2}{$+$} 
  \relabel{p3}{$+$} 
  \relabel {sp}{$\Sigma'$}  
  \endrelabelbox
        \caption{The $(r,0)$--resolution of a rational open book with binding component $K$ being approached by a page as a $(r,-1)$--curve. }
	\label{fig:absres}
\end{figure}

\begin{Nt} Every rational open book $(L,\pi)$ has description as an abstract open book.  When each boundary component has trivial multiplicity (i.e., $\mathrm{gcd}(r,s) = 1$) the monodromy is particularly straightforward to describe.  Let our link components be $L = K_1, \dots, K_n$ and denote by $\Sigma$ the rational fiber surface of $L$.  To give an abstract presentation for this rational open book, we prescribe the boundary behavior of the monodromy near an $(r,s)$ binding component to be a right-handed $\frac{|s|}{r}$ partial or fractional Dehn twist, where we have chosen our framing so that $-r < s \leq 0$.  By right-handed we mean a right-handed fractional rotation of the surface as you move toward the boundary along a cylindrical neighborhood of the boundary component.  (In particular, this rotation is ``right-veering" when measured at the boundary.)  We will denote this fractional Dehn twist by $\delta_{\frac s r }$, adding the superscript $i$ to indicate it acting on the $i$th binding component $K_i$.  Given these fractional Dehn twists we can describe the open book abstractly (uniquely) as $(\Sigma, \prod_{i=0}^n \delta_{\frac{|s|}{r}}^i \circ \phi)$ where $\phi$ is a(n) (isotopy class of) diffeomorphism(s) supported away from $\partial \Sigma$.
\end{Nt}

One particular advantage of this description is that it allows us to determine the monodromy of certain resolutions nicely. In particular the discussion of Example~\ref{exs-1} (specifically the construction of $(L,\pi)$ from $(L',\pi')$ yields the following result.

\begin{prop} Let $(L, \pi)$ be a rational open book obtained from $(L', \pi') = (\Sigma, \phi)$ by $\bf{r}$ surgery on $L'$ with each $r_i$ a negative integer.  Then we can describe $(L, \pi)$ abstractly as $(\Sigma, \prod_{i=0}^n \delta_{\frac{1}{r_i}}^i \circ \phi)$.  \qed
\end{prop}

We can now restate Example~\ref{exs-1} as follows. 
\begin{prop} \label{prop:r-1 resolution} Let $(L,\pi)$ be a $({\bf r},-1)$ open book with abstract description \mbox{$(\Sigma, \prod_{i=0}^n \delta_{\frac{1}{r_i}}^i \circ \phi)$}.  The $({\bf r},0)$--resolution of $(L,\pi)$ can be described abstractly as $(\Sigma_{\bf{r}}, \phi \circ M_{\partial})$ where $\Sigma_{\bf{r}}$ is built from $\Sigma$ by 
gluing a disk with $r_i$ holes to the $i$th boundary component so that $\phi$ acts on $\Sigma$ as a subsurface of $\Sigma_{\bf{r}}$.  The map $M_\partial$ is a Dehn multitwist about all the boundary components of $\Sigma_{\bf{r}}$. \qed
\end{prop}

\section{Surgery on transversal knots}\label{surgeryont}

We recall the notion of an admissible surgery along a transverse knot. Given a transverse knot $K$ in a contact manifold $(M,\xi)$ with a fixed framing $F$ then  $\frac qp$ surgery on $K$  is \dfn{admissible} if there is a neighborhood $N$ of $K$ in $M$ that is contactomorphic to a neighborhood $N'_{r_0}=\{(r,\theta, z)| r\leq \sqrt{r_0}\}$ of the $z$-axis in $\R^3/(z\equiv z+1)$ with the contact structure $\xi'=\ker(dz+r^2\, d\theta)$ such that $F$ goes to the product framing on $N_{r_0}'$ and $-\infty<\frac qp<-\frac 1{r_0}.$ (We remind the reader that we are using a different convention for representing slopes of curves on tori that is usual in contact geometry.  This is to agree with conventions used when describing Dehn surgery coefficients.  See also Section~\ref{sec:slopeconventions}. )

Note that if $M_K(\frac qp)$ is obtained from $M$ by an admissible surgery then there is a natural contact structure $\xi_K(\frac qp)$ on it. The contact structure comes  from a contact reduction process. Specifically let $T_{a}=\{(r,\theta,z)| r={\frac 1{\sqrt a}}\}$ and $S_{a,b}=\{(r,\theta,z)| {\frac 1{\sqrt a}}<r<\frac1{\sqrt b}\}$ in $\R^3/(z\equiv z+1)$ with the contact structure $\xi'=\ker(dz+r^2\, d\theta).$ Then in $N'_{r_0}$ above we have the torus $T_{-\frac qp}$ that divides $N'_{r_0}$ into two pieces $S_{-\frac qp, r_0}$ and $N'_{-\frac qp}.$ If we remove $N'_{-\frac qp}$ from $N'_{r_0}\cong N\subset M$ then we have a manifold $M'$ with a torus boundary component $T_{-\frac qp}$ and the characteristic foliation on this boundary component has slope (with respect to the framing $F$) $\frac qp.$ Notice that if we form the quotient  space of $M'$ with each leaf of the characteristic foliation identified to a point then the resulting manifold is $M_K(\frac qp).$ Let $K'$ be the knot formed from points in the quotient space where nontrivial identifications have been made (that is $K'$ is the core of the surgery torus). Notice that $M_K(\frac qp)-K'$ has a contact structure on it since it is a subset of $M.$ We claim this contact structure extends over $K'.$ To see this we consider $S_{-\frac qp, r_0}$ and let $S$ be the solid torus obtained by identifying the leaves of the characteristic foliation on $T_{-\frac qp}$ to a point. We claim that the contact structure on $S_{-\frac qp,r_0}-T_{-\frac qp}$ extends over the core of $S.$ This is easily seen by applying the contactomorphism 
\[
\Psi:(S_{-\frac qp,r_0}-T_{-\frac qp})\to (N'_{r_1}-Z):(r,\theta, z)\mapsto (f(r),p'\theta-q'z,-p\theta+ qz),
\]
for some $r_1$ where $p',q'$ satisfy $qp'-q'p=1,$ $Z$ is the core of $N'_{r_1},$ and $f(r)$ is a smooth increasing function such that the torus $T_{f(r)}$ has characteristic foliation with slope $\frac {-p'-rq'}{p+rq}.$ One may easily check that since $\Psi$ preserves the radial direction and preserves the characteristic foliations on the tori $T_r$ it is a contactomorphism. It is also clear that this map extends to a homeomorphism from $S$ to $N'_{r_1}.$ Thus we may consider $M_K(\frac qp)$ as being build from $M\setminus N'_{r_0}$ and $N'_{r_1}$ by using $\Psi$ to glue their boundaries together and hence the contact structure on $M_K(\frac qp)$ clearly extends over $K'.$

We recall Gay showed that if $(M,\xi)$ is symplectically fillable then so is $(M_K(\frac qp),\xi_K(\frac qp)),$ \cite{Gay02a}. So admissible surgery on transverse knots is analogous to Legendrian surgery on Legendrian knots (in fact, if the transverse knot is the push off of a Legendrian knot with Thurston-Bennequin invariant greater than the surgery slope, then the admissible surgery is precisely a sequence of Legendrian surgeries). 

The following result is a simple consequence of the construction of compatible contact structures for open books decompositions.
\begin{lem}\label{lem:admissible}
Let  $(L,\pi)$ be an open book decomposition of a manifold $M$ that supports the contact structure $\xi.$ Then if $K$ is a component of the binding of $L$ with framing given by the page of the open book then any negative Dehn surgery is admissible. 

Let $(L,\pi)$ be a rational open book decomposition for $(M,\xi)$ and $K$ be a binding component with order larger than one. Fix a framing on $K$ so that the pages approach $K$ as a $\frac sr$--curve. Then a Dehn surgery coefficient  is admissible if it is less than $\frac sr.$ \qed
\end{lem}

Notice that if $K$ is the binding of an (rational) open book decomposition $(L,\pi)$ and we do any surgery to $K,$ except the one corresponding to the slope of a page approaching $K,$ then letting $L'$ be $L-K$ union the core $K'$ of the surgery torus clearly is the binding of a rational open book decomposition for the new manifold. We will call this an \dfn{induced open book decomposition}.

\begin{lem}\label{lem:admissiblesup}
Let $(L,\pi)$ be a rational open book decomposition for $(M,\xi)$ and $K$ be a binding component. For the admissible surgeries described in Lemma~\ref{lem:admissible} the induced open book decomposition supports the contact structure obtained by the admissible surgery. 
\end{lem}

\begin{proof}
Notice that the compatibility of $(L,\pi)$ with $\xi$ gives a contact form $\alpha$ that is positive on oriented tangents to $L$ and such that $d\alpha$ is a positive area form on the pages. After the admissible surgery $\alpha$ restricted to the complement of a small neighborhood of $K$ gives a contact form on the surged manifold minus a small neighborhood of $K'.$ This contact form shows compatibility with the induced open book everywhere except the neighborhood of $K'.$ Using the contactomorphism $\Psi$ above one may easily use the construction at the end of the proof of Theorem~\ref{thm:support} to extend $\alpha$ over this torus so as to demonstrate compatibility. 
\end{proof}

\begin{Ex}
Let $(L_0,\pi_0)$ be an open book for a contact manifold $(M,\xi)$ and let $K_0$ be a component of the binding $L_0.$ For a positive integer $r$ the $-r$ surgery is an admissible surgery on $K_0$ and satisfies the hypothesis of Lemma~\ref{lem:admissible}. Thus the rational open book $(L,\pi)$ induced on the admissibly surgered manifold supports the resulting contact structure. The binding component of this open book has the page approaching it as a $(-1,r)$--curve. Moreover Figure~\ref{fig:absres} shows how to resolve this rational open book decomposition into an honest open book decomposition. 

Notice that this gives the same open book we would obtain by Legendrian surgery on a Legendrian copy of $K$ given by $r$ right-handed stabilizations of a realization of $K$ on the page of the open book (assuming this is possible, which it is if there are other boundary components).
\end{Ex}

\begin{Ex}  The open book on $S^3$ with binding the left-handed trefoil $K$ supports an overtwisted contact structure.  In particular it supports the contact structure $\xi_{-2}$ with Hopf invariant $-2$. The $-5$ surgery is an admissible surgery and satisfies the hypothesis of Lemma~\ref{lem:admissible} so the induced rational open book decomposition on $S^3_K(-5)$ supports the contact structure obtained by admissible surgery on $K.$ 
It is known that $-5$ surgery on $K$ gives the lens space $L(5,4)$.  Looking at the $(5,0)$--resolution of the rational open book decomposition, we obtain the open book given by Figure~\ref{fig:lensspace}.  Using relations in the mapping class group given in  \cite{KorkmazOzbagci08} one can see that this gives the (unique) Stein fillable contact structures on $L(5,4)$.  Thus, the rational open book supported by the $-5$ surgery on the left-handed trefoil is tight.
\begin{figure}[ht]
  \relabelbox \small 
 {\epsfysize=2truein\centerline {\epsfbox{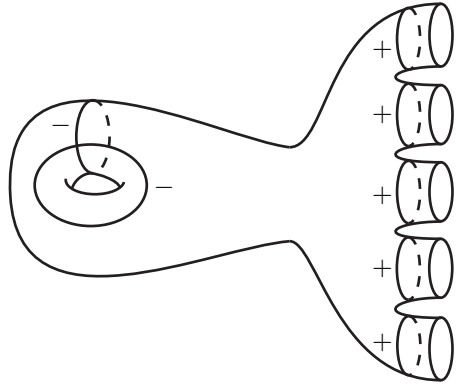}}}
  \relabel{1}{$+$} 
  \relabel {2}{$+$} 
  \relabel{3}{$+$} 
  \relabel {4}{$+$} 
  \relabel{5}{$+$} 
  \relabel {a}{$-$} 
  \relabel{b}{$-$} 
  \endrelabelbox
\caption{The resolution of a rational open book on $L(5,4)$ given by surgery on the left-handed trefoil.}
\label{fig:lensspace}
\end{figure}
\end{Ex} 

It is clear from the results above that to understand the open book decomposition associated to admissible surgeries on a transverse knot $K$ it is helpful to have $K$ in the binding of an open book decomposition. One can always do this as the following lemma, whose potential existence was first observed during conversations between the authors and Vincent Colin, shows. 

\begin{lem}
Let $(L,\pi)$ be an integral open book decomposition for the contact manifold $(M,\xi).$ Assume $K$ is an oriented  Legendrian knot on a page of the open book decomposition (so that the framing given by $\xi$ and by the page agree). Let $\gamma$ be an arc on the page running from one binding component of the open book decomposition to the knot $K$ and approaches $K$ from the right. (The page of the open book and $L$ are both oriented. We say $\gamma$ approaches $K$ from the right if the orientation on $K$ followed by the orientation on $\gamma$ induces the orientation on the page where $K$ and $\gamma$ intersect.) Set $\alpha$ to be a curve that runs from $L$ along $\gamma$ around $K$ and then back to $L$ along a parallel copy of $\gamma.$ The open book $(L_\alpha,\pi_\alpha)$ obtained from $(L,\pi)$ by positively stabilizing along $\alpha$ has a binding component $B$ that is the transverse push off of $K.$ See Figure~\ref{transrealize}.
\end{lem}
\begin{figure}[ht]
  \relabelbox \small 
 {\centerline {\epsfbox{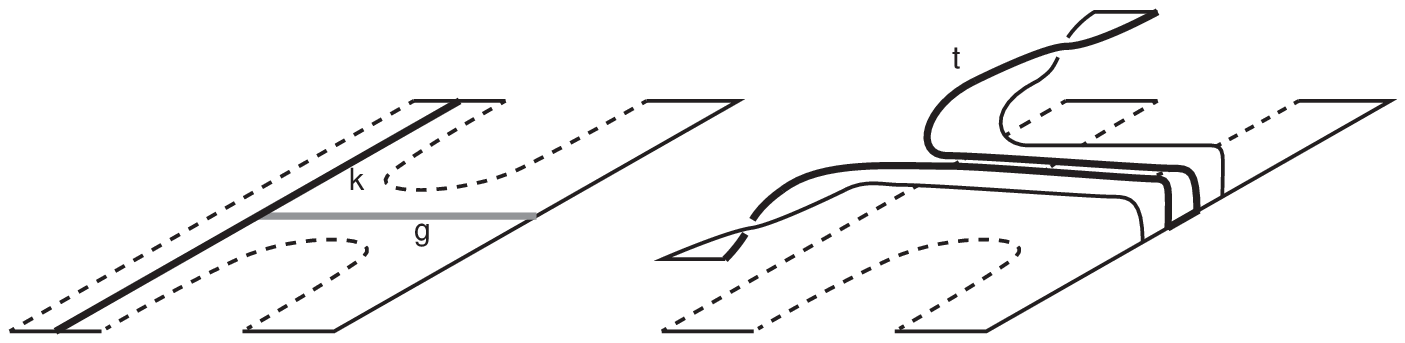}}}
  \relabel{k}{$K$} 
  \relabel {g}{$\gamma$} 
  \relabel{t}{$T$} 
  \endrelabelbox
\caption{On the left is a neighborhood of $K, \gamma$ and a binding component of $L$. (The top and bottom horizontal lines in both figures are identified.) On the right is the stabilized open book with the binding component $T$ drawn in a thicker line, where $T$ is the transverse push off of $K$.}
\label{transrealize}
\end{figure}

\begin{proof}
On the page of the open book $(L_\alpha, \pi_\alpha)$ there is a knot $K'$ that runs over the new 1-handle once, is isotopic to $K$ (in the whole manifold),  and is parallel to one of the new binding components of the open book. One may Legendrian realize both $K$ and $K'$.  Then using a local model for stabilization one can see that together $K$ and $K'$ cobound an embedded annulus where the contact framing of $K$ with respect to this annulus is $0$ and the contact framing of $K'$ with respect to this annulus is $-1.$ One may use this annulus to see that $K'$ is a Legendrian stabilization of $K.$ Moreover, by our choice of $\gamma$ in the lemma, $K'$ will be a negative stabilization of $K.$ Thus $K$ and $K'$ have the same transverse push-offs, see \cite{EtnyreHonda01b}. 

Recall that $K'$ cobounds an annulus with one of the binding components of the open book. Using this annulus one sees that the binding component is the transverse push-off of $K'.$ Thus the binding is also the transverse push-off of $K.$
\end{proof}

\begin{rem}
There are many techniques for putting Legendrian knots on pages of open book decompositions \cite{Etnyre04b, EtnyreOzbagci06} (though this is still more of an art than a science).  Since any transverse knot can be realized as the transverse push off of a Legendrian knot this previous lemma, coupled with resolutions of rational open books,  allows us to find open books for admissible surgeries on transversal knots. The lemma also gives us a convenient way to find open books for all ``rational Legendrian surgeries'' on Legendrian knots sitting on a page of an open book. 
\end{rem}

\section{Monodromies of cables}\label{sec:monodromy}

In Lemma~\ref{lem:cable} we discussed how to construct the page of a cabled fibered link from the page of the original link.  In this section we show how to compute the monodromy, in terms of Dehn twists, of certain cablings of an integral open book decomposition from the monodromy of the original open book decomposition. (Throughout this section we will focus solely upon integral open books.)  It turns out this is most difficult when the binding is connected and one wishes to perform a $(p,1)$--cable. This will be addressed in Subsection~\ref{sec:connected binding}. The somewhat simpler cases will be addressed first in Subsection~\ref{sec:easy case}. We will analyze both cases using branched coverings, so we review a few facts about branched covers in Subsection~\ref{bc}. Subsection~\ref{sec:monsetup} contains a basic proposition that we use throughout our analysis. 

\begin{rem} A quick note regarding notation:  throughout this section and the next, we will use group notation for braids, writing them left to right, and functional notation for mapping class elements, composing them right to left.  Whenever product notation is used, we will assume that the lower indexed elements act first.  In the functional notation, the product of mapping class elements would be written $\prod_{i=1}^n f_i = f_n \circ f_{n-1} \circ \cdots \circ f_1$, whereas for a braid we would write resulting product in the reverse order. \end{rem}
\subsection{Branched covers and open book decompositions}\label{bc}
As we will be using branched covering technology and terminology heavily in this section we refer the reader to \cite{BersteinEdmonds79, Montesinos-AmilibiaMorton91} for a thorough discussion of the relevant results; however, we review a few basic facts used below for the convenience of the reader. 

Let $(U,\pi)$ be the open book for $S^3$ with binding the unknot. 
Recall a link $B$ is braided in $S^3$ it it is transverse to the pages of $(U,\pi)$, or equivalently is isotopic to a closed braid through links transverse to the pages of $(U,\pi)$. The braid index of $B$, seen as a closed braid,  is the number of times $B$ intersects the pages of $(U,\pi)$. If $B$ has braid index $n$ then notice that we can fix $n$ points $x_1,\ldots, x_n$ on a disk $D^2$ then there will be a diffeomorphism $\phi$ of $D^2$ that fixes these points set-wise such that the image of $\{x_1\times[0,1],\ldots, x_n\times[0,1]\}$ in the mapping cylinder of $\phi$ will trace out a link $L_\phi$ that is braid isotopic to $B$ when the mapping cylinder is completed to give the open book $(U,\pi)$ of $S^3$.  

Let $p:\Sigma\to D^2$ be a $k$-fold covering map branched over $x_1,\ldots, x_n$. 
If there is a diffeomorphism $\phi'$ of $\Sigma$ that covers $\phi$, then one may easily check that the manifold associated to the open book $(\Sigma, \phi')$ is a $k$-fold cover of $S^3$ branched along $B$. 

We call a $k$-fold branched cover $p:\Sigma\to \Sigma'$ of surfaces \dfn{simple} if the the pre-image of any point has either $k$ or $(k-1)$ points. Notice if the pre-image has $k$ points then all the pre-image points are regular point (that is, it is not a branched point).  If it has $(k-1)$ points then $(k-2)$ of the points are regular points and the other point has order two ramification, that is there are local (complex) coordinates where $p$ looks like the map $z\mapsto z^2$. 

Suppose that $p:\Sigma\to D^2$ is a simple $k$-fold cover branched over the points $x_1,\ldots, x_n$. Let $\gamma$ be an arc with end points $x_i$ and $x_j$ and not intersecting any $x_k$ on its interior. Let $h_\gamma$ be the diffeomorphism of $D$ that exchanges $x_i$ and $x_j$ via a right handed twist in a small neighborhood of $\gamma$. The pre-image of $\gamma$ in $\Sigma$ will be a collection of arcs and possible a circle. The diffeomorphism $h$ is covered by the composition of right handed half twists between the end points of the arcs and a right handed Dehn twist about the circle (if it exists). Thus $h$ is covered by a diffeomorphism that is either isotopic to the identity or isotopic to a Dehn twist about the circle covering $\gamma$ (if it exists). 

Here are some simple examples. Let $g$ be the genus of  $\Sigma$.  Let $n = |\bdry \Sigma|$ if $\Sigma$ has disconnected boundary, but set $n=2$ if $\Sigma$ has connected boundary.   If $\Sigma$ has connected boundary then quotienting $\Sigma$ by the hyper-elliptic involution shows that it can be realized as a 2-fold cover over the disk branched along $2g+1$ points. If $\Sigma$ has disconnected boundary then 
the surface $\Sigma$ can be built by an $n$-fold simple branched cover of the disk, branched over $d = (2g+2) + 2(n-2)$ points.  To see this, think of the disk as the unit disk in $\R^2$ and the points on the $x$-axis, labeled $x_1$ to $x_d$, left to right.  Each adjacent pair of points $(x_{2k-1}, x_{2k})$ can be connected by an arc $\gamma_k$ in the $x$-axis.  Crossing this arc moves you between sheet 1 and sheet 2 in the cover for $k =1, \dots, g+1$ and between sheet 1 and sheet $k - g+1$  for $k = g+2, \dots, g+n$.  See Figure~\ref{fig:basicbcover}. 
\begin{figure}[ht]\small 
 {\epsfysize=2.2truein\centerline {\epsfbox{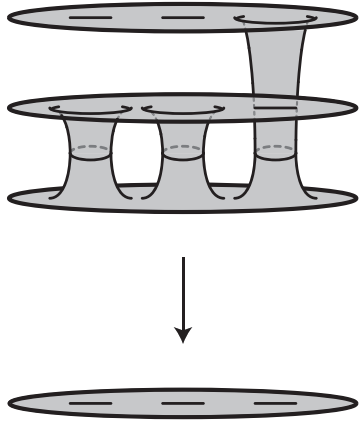}}} 
        \caption{The simple 3-fold branched cover of a genus 1 surface with 3 boundary components $\Sigma$ over the disk. The gray arcs on the disk are the arcs $\gamma_i$ and their preimages are shown in $\Sigma.$ }
	\label{fig:basicbcover}
\end{figure}
Now one easily sees that $M_{\left(\Sigma, \Id\right)}$ can be built as a simple $n$-fold branched cover of $M_{(D^2, \Id)} \cong S^3$, branched over the closure $L_B$ of the trivial $d$ component braid $B$ where we think of $L_B$ as braided about the unknot $U=\partial D^2.$

Denote the standard generators in the braid group by $\sigma_i$, this generator switches the $i$th and $(i+1)$st points by a right handed half twist. Now, for $i < j$, set 
\[
\sigma_{i,j}= \sigma_i^{-1}\ldots \sigma_{j-2}^{-1}\sigma_{j-1}\sigma_{j-2}\ldots\sigma_{i}.
\]
That is to say, $\sigma_i$ is the generator that exchanges the $i$ and $i+1$ strands in the braid by a right handed twist while $\sigma_{i,j}$ is the element in the braid group that switches the $i$ and $j$ strands with a right handed twist along an arc that lies in the front of the braid diagram.  Thinking of the braid as the trace of a set of marked points on the $y$-axis of the disk, the arc lies to the left of the marked points $x_1,\ldots, x_n$, where we index the marked points starting from the bottom strand. It is an interesting and simple exercise to determine what these braid generators lift to in the branched covers mentioned above. Any braid can be realized as a composition of these generators.  

We recall that a \dfn{Markov stabilization} of a braid $B$ of index $n$ is obtained by adding an $(n+1)$st point and composing with $\sigma_n$. 

\begin{lem}[Montesinos-Amilibia and Morton 1991, \cite{Montesinos-AmilibiaMorton91}] \label{lem:stabilize} Let $(D^2,\Id)$ be the standard open book decomposition for $S^3$ and let $L$ be a link braided about $U=\partial D^2.$ Suppose $(\Sigma,\phi)$ is an open book decomposition constructed by a finite sheeted simple branched cover of $(D^2,\Id)$ branched along the link $L.$ Let $L'$ be obtained from $L$ by a positive Markov stabilization. As $L'$ is topologically isotopic to $L$ the branching data from $L$ provides branching data for $L'.$ The open book decomposition $(\Sigma',\phi')$ obtained from  $(D^2,\Id)$ by the simple branched cover of $L'$ is obtained from $(\Sigma, \phi)$ by a positive Hopf stabilization. 
\end{lem}

\begin{proof}
There is neither need to restrict to $S^3$ nor to the trivial open book $(D^2, \Id)$, and so we give a proof that holds for an arbitrary link $L$ transverse to an arbitrary open book $(F, \psi)$ in an arbitrary manifold $M$.  Denote by $\pi$ the covering map $\pi : M_L \rightarrow M$.  Let $(F_L, \psi_L)$ be the open book decomposition on $M_L$ induced from the cover. 

Restricting $\pi$ to a page $F$ of the open book gives finite sheeted cover $\pi_L : F_L \rightarrow F$, branched over the points $\{x_1, \dots, x_d\}$.  We think of $L$ as a braid and hence a map $B_L \in Map^+(F, \{x_1, \dots, x_d\})$ in the mapping class group of $(F,\{x_1, \dots, x_d\})$ (the diffeomorphisms are all the identity in a neighborhood of the boundary). 

We define \dfn{positive Markov stabilization} in this context as follows. Let $\alpha$ be an arc connecting $x_d$ to the boundary and choose a point $x_{d+1}$ on $\alpha$ that is contained in the region where $B_L$ is the identity. The positive Markov stabilization is now the map $B_{L'}$ obtained from $B_L$ by composing with the diffeomorphism that exchanges $x_d$ and $x_{d+1}$ by a right handed twist contained in a small neighborhood of the portion of $\alpha$ between $x_d$ and $x_{d+1}.$ One may easily check that this corresponds to positive Markov stabilization (and conjugation) in the standard braid group. 

If the cover $\pi$ is unramified along the component of $L$ containing $x_d$ then the branched cover of $L$ and $L'$ are the same so $(F_{L'}, \psi_{L'})=(F_L, \psi_L)$.   Thus we assume that $L$ is fully ramified, so that some branching occurs along each component of $L$.   In this case, as $x_{d+1}$ and $x_d$ are on the same component of the link described by the braid, they must be ramified in the same way.  By this we mean that if we use the arc $\alpha$ to provide branch cuts for $x_d$ and $x_{d+1}$, the holonomy of the branched cover about both points will be the same. (Said in a different way, $\alpha$ can be used to provide an explicit relation between the loops about $x_d$ and $x_{d+1}$ in the fundamental group of the complement of the branched points.  With this identification these elements will both map to the same permutation of the sheets of the cover).
Now one easily sees that $F_{L'}$ is obtained from adding a 1-handle to $F_L$ and the arc between $x_d$ and $x_{d+1}$ will lift to a simple closed curve $a$ that intersects the co-core of this 1-handle one time. The extra right handed twist in $B_{L'}$ will lift to a right handed Dehn twist in $F_{L'}$, so  $\psi_{L'}$ is $\psi_L$ composed with a right handed Dehn twist about the simple closed curve $a.$
\end{proof}

\subsection{Monodromies of cables}\label{sec:monsetup}
Let $\left(\Sigma, \phi \right)$ be the page and monodromy of an open book on a manifold $M$ with binding $L$, where $L$ is an $n$ component link, $L = K_1 \cup \dots \cup K_n$.  The goal of this section is to give a Dehn twist presentation of the monodromy of the open book $\left(\Sigma_{(p,\bf{q})}, \phi_{(p,\bf{q})} \right)$, the $(p,\bf{q})$--cable of $\left(\Sigma, \phi \right)$.  As in the previous sections, $p$ is a positive integer and $\bf{q}$ a vector of length $n$ of positive integers.  By $(p,\bf{q})$--cable we mean the open book made by replacing the $i$th component $K_i$ of $L$ with its $(p, q_i)$--cable.  

We recall the construction in Lemma~\ref{lem:cable} of the pages of an open book after cabling. See the proof of that lemma for notation, but briefly recall we took a $(p,q)$--torus knot $T_{(p,q)}$ in $S^3$ and noticed that we could choose a core $C$ of a Heegaard torus for $S^3$ such that the Seifert surface $F_{(p,q)}$ for $T_{(p,q)}$ (which is a page of an open book decomposition for $S^3$ with binding $T_{(p,q)}$) intersected $C$ in $p$ points, which we label $x_1,\ldots, x_p.$ We now observe that the monodromy of the open book $T_{(p,q)},$ which we denote $\psi_{(p,q)},$ takes points $x_{i}$ to $x_{i+1}$ (where $i$ is taken modulo $p$). Let $N_C$ be a small tubular neighborhood of $C.$ This can be chosen so that $N_C\cap F_{(p,q)}$ is $p$ disjoint disks $D_1\ldots, D_{p}$ and $\psi_{(p,q)}|_{D_{i}}$ is a diffeomorphism from $D_{i}$ to $D_{i+1}$. Now $S^3\setminus N_C$ is denoted $S_{(p,q)}$ and is a solid torus $D^2\times S^1$ containing $T_{(p,q)}$ such that the open book structure on $S^3$ gives a fibration $S_{(p,q)}\setminus T_{(p,q)}\to S^1.$ This fibration induces a fibration of $\partial S_{(p,q)}$ by curves $\{pt\}\times S^1$ and the preimage of any point in $S^1$ intersected with $\partial S_{(p,q)}$ is $p$ curves. In addition the fiber of this fibration is $C_{(p,q)}=F_{(p,q)}\setminus (\cup_{i=1}^{p} D_i)$ and the monodromy is $\psi'_{(p,q)}=\psi_{(p,q)}|_{C_{(p,q)}}.$

As in Lemma~\ref{lem:cable} the open book $\left(\Sigma_{(p,\bf{q})}, \phi_{(p,\bf{q})} \right)$ is built by removing small neighborhoods of the binding components of $\left(\Sigma, \phi \right)$ and replacing them with $S_{(p,q_i)}.$ Thus the fiber surface $\Sigma_{(p,\bf{q})}$ is built by taking the surfaces $C_i,$ where $C_i$ is isomorphic to $C_{(p,q_i)}$ and the $p$ cyclicly ordered boundary components, $O_{i,j}$ are ordered so that the $\psi'_{(p,q_i)}$ takes $O_{i,j}$ to $O_{i,j+1}$.   To this collection of surfaces one glues $p$ copies $\Sigma_1, \dots \Sigma_p$ of $\Sigma$, gluing the $i$th boundary component of $\Sigma_j$ to $O_{i,j}$.  
 
If the monodromy $\phi$ is $id_\Sigma$ then the monodromy $\phi_{(q,\bf{q})}$ is simply $\psi_{(p,q_i)}$ on $C_i$ and sends $\Sigma_{j}$ to $\Sigma_{j+1}$, where again, $j=p+1$ is identified with $j=1.$ Since the $\Sigma_i$ can be thought of as sitting in the complement of the binding of the original open book which is a product we can use this product structure to identify $\Sigma_{j}$ with $\Sigma_{j+1}.$ Thus we have an explicit description of the monodromy in this case. We denote this monodromy map as $\rho_{(p,{\bf q})}(\Sigma).$ 

If the the original monodromy map $\phi$ is non-trivial, then we can describe $\phi$ as the identity map followed by a sequence of positive and negative Dehn twists performed on fiber surfaces near $\Sigma_1$, which we then interpret as Dehn surgeries on curves lying on pages near $\Sigma_1$. Thus the monodromy map $\phi_{(q,\bf{q})}$ will differ from $\rho_{(p,{\bf q})}(\Sigma)$ by performing these Dehn surgeries on the curves near $\Sigma_1\subset \Sigma_{(p,{\bf q})}.$ We denote by $\widetilde{\phi}$ the diffeomorphism of $\Sigma_{(p,{\bf q})}$ obtained from these Dehn twists on $\Sigma_1$ and call it the \dfn{lift} of $\phi$ to $ \Sigma_{(p,{\bf q})}.$

Because we will use this decomposition of $\Sigma_{(p,\bf{q})}$ rather heavily, we introduce the term \dfn{nodules} to refer to these distinguished subsurfaces $\Sigma_j$ of $\Sigma_{(p,\bf{q})}$.  The remaining connected components, $C_i$, of $\Sigma_{(p,\bf{q})}$ will be called \dfn{base components}.  
The goal of this section is to find a Dehn twist presentation of the monodromy $\phi_{(p,\bf{q})}$ of the cable, and the following proposition, which is proven above, allows us to do this without referring to a particular open book.

\begin{prop}  \label{prop:monodromy split}  Let $\left(\Sigma_{(p,\bf{q})}, \phi_{(p,\bf{q})}\right)$ be the $(p, \bf{q})$--cable of an open book decomposition $(\Sigma, \phi)$.  The monodromy $\phi_{(p,\bf{q})}$ splits as a product $\phi_{(p,\bf{q})} = \rho_{(p,\bf{q})}(\Sigma) \circ \widetilde{\phi}$, where $\rho_{(p,\bf{q})}(\Sigma)$ is a universal map depending only on $\Sigma$, $p$ and $\bf{q}$, and $\tilde{\phi}$ is a lift of $\phi$ acting on the first nodule $\Sigma_1$.  This factorization holds for any conjugation of the factors by a map of the cable surface $\Sigma_{(p,\bf{q})}$ which preserves the nodules and hence is independent of the identification of $\Sigma$ with $\Sigma_1$.\qed
\end{prop}

Thus, assuming we know the original monodromy $\phi$, we need understand only $\rho_{(p,\bf{q})}(\Sigma)$ in order to understand the monodromy of the cable.  The idea will be to construct Dehn twist presentations of $\rho_{(p,\bf{q})}(\Sigma)$ while keeping track of the first nodule $\Sigma_1$ without a specific identification with $\Sigma$.  

\subsection{Simple branched covers and cablings}\label{sec:easy case}
In this subsection we understand $\rho_{(p,\bf{q})}(\Sigma)$ in the case where $\Sigma$ has more than one boundary component or all the $q_i>1.$ The reason for this restriction is that $\Sigma$ will have a nice branched cover description that can be exploited. We illustrate this basic idea in the next theorem. Theorems~\ref{thm:mon2discon}, \ref{prop:22cable} and~\ref{thm:connected binding} expand on the basic ideas used here.   However, as we will see in Theorem~\ref{thm:counterexample}, this result does not hold for $(p,1)$--cables in general.

\begin{thm} \label{thm:stabilization} 
Let $(\Sigma, \phi)$ be an open book decomposition and $\left(\Sigma_{(p,\bf{q})}, \rho_{(p,\bf{q})}(\Sigma) \circ \tilde{\phi}\right)$ its $(p,\bf{q})$--cable with $p$ and each $q_i$ a positive integer.  If either $\Sigma$ has disconnected boundary or each $q_i>1$ then $\left(\Sigma_{(p,\bf{q})}, \rho_{(p,\bf{q})}(\Sigma) \circ \tilde{\phi}\right)$ can be obtained from $\left(\Sigma, \phi\right)$ by a sequence of positive Hopf stabilizations.
\end{thm}

\begin{proof}  
By Proposition~\ref{prop:monodromy split}, it is enough to prove the theorem when $\phi$ is the identity on $\Sigma,$ as any sequence of positive Hopf stabilizations from $\left(\Sigma, \Id\right)$ to $\left(\Sigma_{(p,\bf{q})}, \rho_{(p,\bf{q})}(\Sigma)\right)$ can be used to build a sequence of positive Hopf stabilizations from $\left(\Sigma, \phi\right)$ to $\left(\Sigma_{(p,\bf{q})}, \rho_{(p,\bf{q})}(\Sigma) \circ \widetilde{\phi}\right)$.  
Following \cite[Section 4.3 and Figure 4.2]{NeumannRudolph87}, and recalled in Lemma~\ref{lem:gencable} above, the ${(p,\bf{q})}$--cable can be obtained from the $(p,\bf{1})$--cable by positive Hopf stabilizations.  
These stabilizations can be done along arcs disjoint from $\Sigma_1$ and thus we may make the further simplification that $q_i=1$ (or in the case of connected boundary we will take $q_1=2$).

Recall the description of $\left(\Sigma, \Id\right)$ as a simple branched cover from Subsection~\ref{bc}. The branched cover description breaks down into two cases. One when $\partial \Sigma$ is connected and one when it is not. 

{\bf Case 1.}  {\em Disconnected boundary.} 
We begin with a simple lemma.
\begin{lem}\label{realizedisconnected}
If $\Sigma$ has disconnected boundary then the $(p,1)$--cable of the binding of the open book $\left(\Sigma, \Id\right)$ can be realized as the $p$-fold branched cover of the braid $B_p$ shown in Figure~\ref{fig:sqpword}. The braid $B_p$ can be expressed as

\begin{align}
 \prod_{i=2}^{p}\prod_{j=1}^{d} \sigma_{(p-i)d+j, (p-i+1)d+j} = 						&(\sigma_{(p-2)d+1,(p-1)(d)+1}\sigma_{(p-2)d+ 2,(p-1)d+2}\cdots \sigma_{(p-1)d,pd}) \cdot \notag\\
	&(\sigma_{(p-3)d+1, (p-2)d+1}\sigma_{(p-3)d+2,(p-2)d+2}\cdots \sigma_{(p-2)d,(p-1)d})\cdots \label{sqp-word}\\
	&(\sigma_{1, d+1}\sigma_{2, d+2}\cdots\sigma_{d, 2d}).\notag
\end{align}
\end{lem}
\begin{figure}[ht]\small 
 {\epsfysize=1.5truein\centerline {\epsfbox{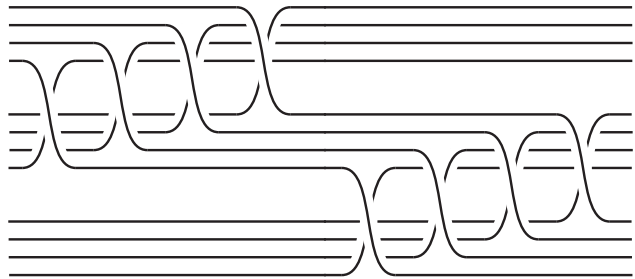}}} 
        \caption{A representation of $B_p$ in terms of the generators $\sigma_{i,j}.$ The figure illustrates $d=4$ and $p=3.$}
	\label{fig:sqpword}
\end{figure}
\begin{proof}
Since the boundary of $\Sigma$ is disconnected, the cover constructing $\Sigma$ is trivial covering along $\partial D^2$ is the trivial $n$-fold cover, and so the $(p,1)$--cable $U_p$ of $\partial D^2$ lifts to the $(p,1)$--cable of every component of $\partial \Sigma$.  The knot $U_p$ is again an unknot and the trace of the branch loci in $D^2$, $L_B$, is now braided about $U_p$ (in particular, it is transverse to the disk fibers in the complement of $U_p$).  Untwisting $U_p$ to make it the braid axis transforms $L_B$ into the closure of the $dp$ stranded braid $B_p$, the branch cover of which reconstructs the fibration on the complement of the lift of $U_p$, i.e., the desired $(p,1)$-cable.  The left hand side of Figure~\ref{fig:unwindB} shows the braid $B$ and the cable $U_p$ of braid axis $U$, while the right hand side shows the braid $B_p$ after unwinding $U_p$. One may easily verify that this braid can also be expressed as in Equation~\eqref{sqp-word}.
\begin{figure}[ht]
  \relabelbox \small 
 {\epsfysize=1.75truein\centerline {\epsfbox{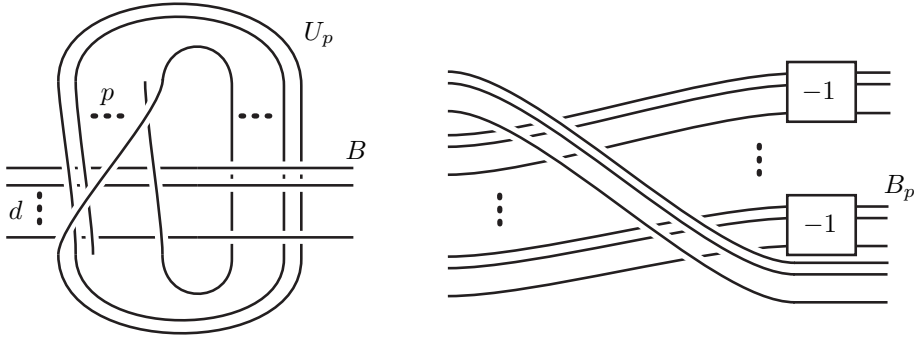}}} 
  \relabel{p}{$p$} 
  \relabel {d}{$d$} 
  \relabel{u}{$U_p$} 
  \relabel {b}{$B$} 
  \relabel{n}{$B_p$} 
  \relabel {t}{$-1$} 
  \relabel{z}{$-1$}  
  \endrelabelbox
        \caption{The cabled unknot $U_p$ and braid $B$ on the left. On the right is $B_p,$ that is $L_B$ when written as a braid about the unknot $U_p.$ The braid $B_p$ has $dp$ strands.  The strand index starts at the bottom. }
	\label{fig:unwindB}
\end{figure}
Since the disk $U_p$ bounds arises by the cabling construction from the disk $U$ bounds, the branched cover lifts it to the surface $\Sigma_{(p,\bf{q})}$.  
\end{proof}

From the presentation of $B_p$ given in the lemma it is easy to see it can be obtained from the $d$-strand trivial braid about $U_p$ by positive Markov stabilizations.
Using Lemma~\ref{lem:stabilize}, it then follows that $\left(\Sigma_{(p,\bf{q})}, \rho_{(p,\bf{q})}(\Sigma)\right)$ is a Hopf stabilization of $(\Sigma, \phi)$.  
 
 {\bf Case 2.}  {\em Connected boundary.}
When $\Sigma$ has connected boundary, the above goes through as stated, and nearly the same as in the previous case, though since the chosen branched cover over $\partial D^2$ is a non-trivial 2-fold cover, $U_p$ now lifts to the $(p,2)$--cable of $\partial \Sigma$.  
\end{proof}

We are now ready to explicitly describe the monodromy of the cabled open book. 
\begin{thm} \label{thm:monpdiscon}  \label{thm:mon2discon} 
The monodromy $\phi_{(p,{\bf 1})}$ of the $(p,{\bf 1})$--cable of an open book $\left(\Sigma, \phi\right)$ with disconnected binding can be written as 
$$\phi_{(p,{\bf 1})} = \displaystyle\prod_{j=1}^{p-1} \displaystyle\prod_{i=1}^{d} D_{c_{p-i,j}} \circ \tilde{\phi},$$
where  $\tilde{\phi}$ is the lift of $\phi$, acting on the first nodule and $D_{c_{i,j}}$ is the right handed Dehn twist along the curve $c_{i,j}$,  that is the simple closed curve component of the lift of the curve $a_{i,j}$ shown in Figure~\ref{fig:basedisk3} to the branch cover $\Sigma_{(p,{\bf 1})}.$ The $c_{i,j}$ can also be thought of as the image of $c_j$ from Theorem~\ref{thm:mon2discon} under the identification of the subsurface of $\Sigma_{(p,{\bf q})}$ lying above $D_{i,i+1}$ shown in Figure~\ref{fig:basedisk3} with $\Sigma_{(2,{\bf 1})}$.
\end{thm}
 \begin{figure}[ht]
  \relabelbox \small 
 {\epsfysize=2truein\centerline {\epsfbox{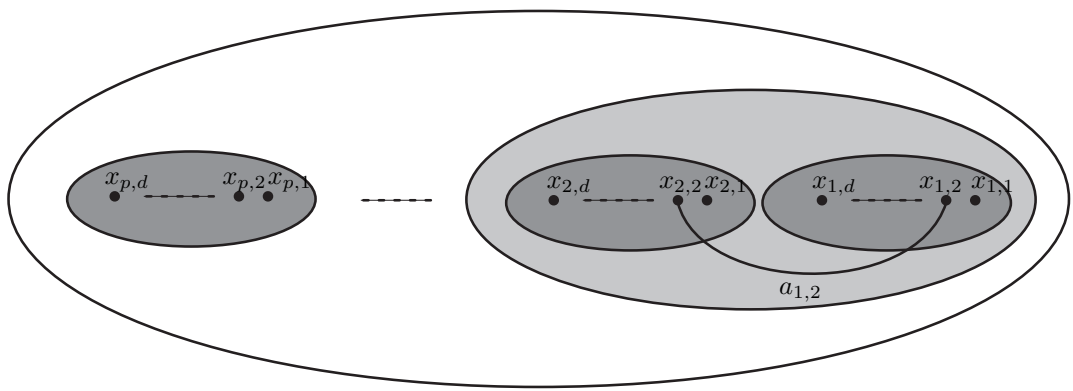}}} 
  \relabel{a}{$x_{2,1}$} 
  \relabel {b}{$x_{2,2}$} 
  \relabel{d}{$x_{1,d}$} 
  \relabel {1}{$x_{1,1}$}  
    \relabel{2}{$x_{1,2}$} 
  \relabel {e}{$x_{2,d}$} 
  \relabel{c}{$a_{1,2}$}     
    \relabel {m}{$x_{p,1}$}  
    \relabel{n}{$x_{p,2}$} 
  \relabel {o}{$x_{p,d}$}    
  \endrelabelbox
        \caption{The disk $D$ with its subdisks $D_1,\ldots D_p$ shaded. The disk $D_{1,2}$ is lighter grey.}
	\label{fig:basedisk3}
\end{figure}

\begin{proof}[Proof of Theorem~\ref{thm:monpdiscon} in the case of $p=2$] 
We give a detailed discussion of the monodromy computation with $p=2$. Later we extend this to all $p>1$. In this case the theorem states:
The monodromy $\phi_{(2,{\bf 1})}$ of the $(2,{\bf 1})$--cable of an open book $\left(\Sigma, \phi\right)$ with disconnected binding can be written as $$\phi_{(2,{\bf 1})} = \displaystyle\prod_{i=1}^d D_{c_i} \circ \tilde{\phi}$$ where $\tilde{\phi}$ is the lift of $\phi$, acting on the first nodule.  We point out that, notationally, this product is a sequence of compositions of Dehn twists with the lowest indexed twists acting first.  The Dehn twists $D_{c_i}$ are Dehn twists along the curves $c_i$, which are the simple closed curve components of the lifts to the branch cover $\Sigma_{(2,{\bf 1})}$ of the curves $a_i$ shown in Figure~\ref{fig:basedisk2}. The curves $c_i$ are also shown in  Figure~\ref{fig:bcover+gen1} and~\ref{fig:bcover+gen2}. See Figure~\ref{fig:nicecover} for a symmetric view of $\Sigma_{(2,{\bf 1})}$ and the curves $c_i.$  

\begin{figure}[ht]
  \relabelbox \small 
 {\epsfysize=3truein\centerline {\epsfbox{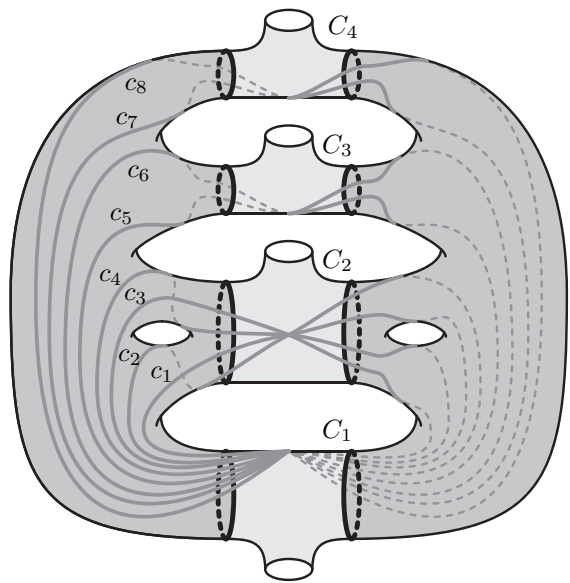}}} 
  \relabel{1}{$c_1$} 
  \relabel {2}{$c_2$} 
  \relabel{3}{$c_3$} 
  \relabel {4}{$c_4$}    
    \relabel{5}{$c_5$} 
  \relabel {6}{$c_6$} 
  \relabel{7}{$c_7$} 
  \relabel {8}{$c_8$}    
  \relabel{9}{$C_4$}
  \relabel{10}{$C_3$}
  \relabel{11}{$C_2$}
  \relabel{12}{$C_1$}
  \endrelabelbox
        \caption{The page $\Sigma_{(2,{\bf 1})}$ drawn symmetrically when $\Sigma$ is genus 1 and has 4 boundary components. The nodules $\Sigma_1$ and $\Sigma_2$ are the right most and left most surfaces and the basic components $C_1,\ldots, C_4$ are the four central pairs-of-pants. }
	\label{fig:nicecover}
\end{figure}

Since Proposition~\ref{prop:monodromy split} allows us to compute $\phi_{(2,{\bf 1})}$ as a product of $\rho_{(2,{\bf 1})}(\Sigma)$ and $\tilde{\phi}$ (provided we keep track of the nodules of $\Sigma_1$), we begin by assuming that $\phi = \Id$.  Now as detailed in the proof of Subsection~\ref{bc}, since $n = |\bdry \Sigma| \geq 2$, the open book decomposition $\left(\Sigma, \Id\right)$ can be thought of as a simple $n$-fold branched cover of $(D^2, \Id)$ branched over a $d$-component unlink that sits transverse to $(D^2, \Id)$ as the trivial braid $B$, thinking of $U = \partial D^2$ as the braid axis.  Moreover, because $n>1$, the cover along $\partial D^2$ is the trivial $n$-fold cover. (When the binding of the open book is connected (i.e., $n=1$), a separate construction is needed.  This will be given in Section~\ref{sec:connected binding}.)  

\begin{figure}[ht]
  \relabelbox \small 
 {\epsfysize=2.75truein\centerline {\epsfbox{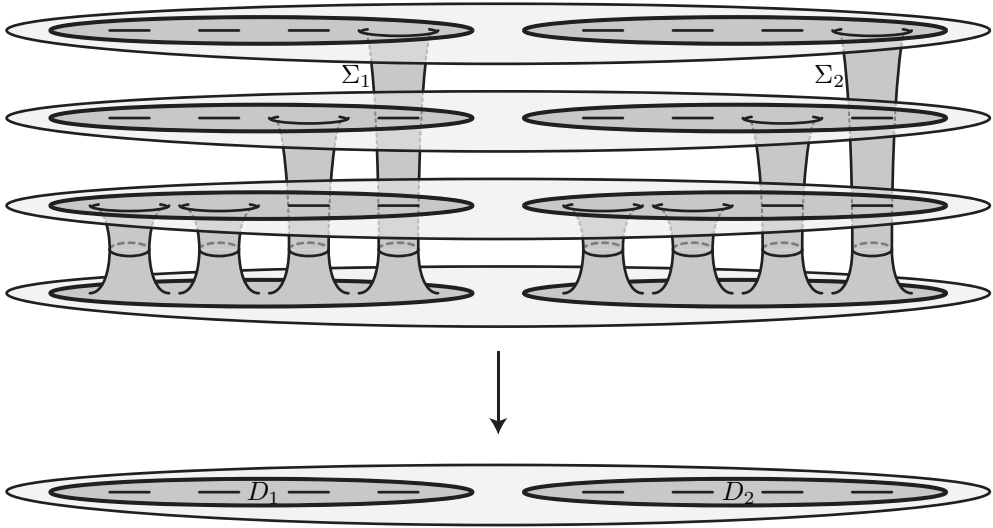}}} 
  \relabel{s}{$\Sigma_1$} 
  \relabel {t}{$\Sigma_2$} 
  \relabel{d}{$D_1$} 
  \relabel {f}{$D_2$}    
  \endrelabelbox
        \caption{A page of the $(2,1)$--cable of $(\Sigma, \Id)$ when $\Sigma$ has genus 1 and 4 boundary components. The two nodules $\Sigma_1$ and $\Sigma_2$ are shown in grey. The white regions are the 4 base components $C_1,\ldots, C_4.$}
	\label{fig:bcovertwo}
\end{figure}

Lemma~\ref{realizedisconnected} shows that 
the open book decomposition for the $(2,\bf{1})$--cable of the open book decomposition $(\Sigma, \Id)$ is obtained as the simple cover branched over the $2d$-braid  $B_2$ given in Equation~\eqref{sqp-word} (with $p=2$). The page $\Sigma_{(2,\bf{1})}$ of the $(2,\bf{1})$--cable is shown in Figure~\ref{fig:bcovertwo} with the nodules and base components labeled. To be specific we think of $D^2$ as a disk in $\R^2$ that contains a segment of the $x$-axis. We then label $2d$-points on the $x$-axis from right to left, $x_{1,1}, \ldots, x_{1,d},x_{2,1},\ldots, x_{2,d}.$ Let $D_1$ and $D_2$ be two disjoint subdisks of $D$ with $D_i$ containing the $x_{i,j}$ with  $i=1,2$ and for  $j=1, \ldots, d$.  The cable surface $\Sigma_{(2,{\bf 1})}$ is the simple cover of $D^2$ branched over the $x_{i,j}$ with ramification data as described in the proof of Theorem~\ref{thm:stabilization} for the $x_{1,j}$ and the same data repeated for the $x_{2,j}.$ Moreover, the nodules $\Sigma_i$ are lifts of the subdisks $D_i.$ Let $a_i$ be the arc embedded in $D$, with negative $y$-coordinate on its interior, that connects $x_{1,i}$ to $x_{2,i}$ as shown in Figure~\ref{fig:basedisk2}. 
\begin{figure}[ht]
 \relabelbox \small 
{\epsfysize=1.5truein\centerline {\epsfbox{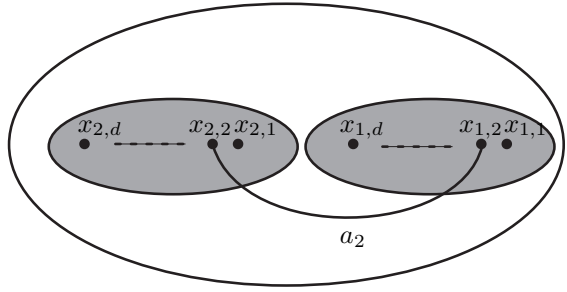}}} 
 \relabel{d}{$x_{2,1}$} 
 \relabel {2}{$x_{2,2}$} 
 \relabel{a}{$x_{1,d}$} 
 \relabel {e}{$x_{1,1}$}  
   \relabel{b}{$x_{1,2}$} 
 \relabel {1}{$x_{2,d}$} 
 \relabel{c}{$a_2$}     
 \endrelabelbox
       \caption{The disk $D$ with its two subdisks $D_1$ and $D_2$ shaded.}
	\label{fig:basedisk2}
\end{figure}
The braid $B_2,$ thought of as an element of the mapping class group, is given as $B_2=\prod_{i=1}^d \tau_i$ where $\tau_i$ is a right handed half twist exchanging $x_{1,i}$ and $x_{2,i}$ in a small neighborhood of $a_i.$ Each $a_i$ lifts to a simple closed curve $c_i$ in $\Sigma_{(p,{\bf 1})}$ (and several arcs). See Figures~\ref{fig:bcover+gen1} and~\ref{fig:bcover+gen2}. 
\begin{figure}[ht]
  \relabelbox \small 
 {\epsfysize=2.75truein\centerline {\epsfbox{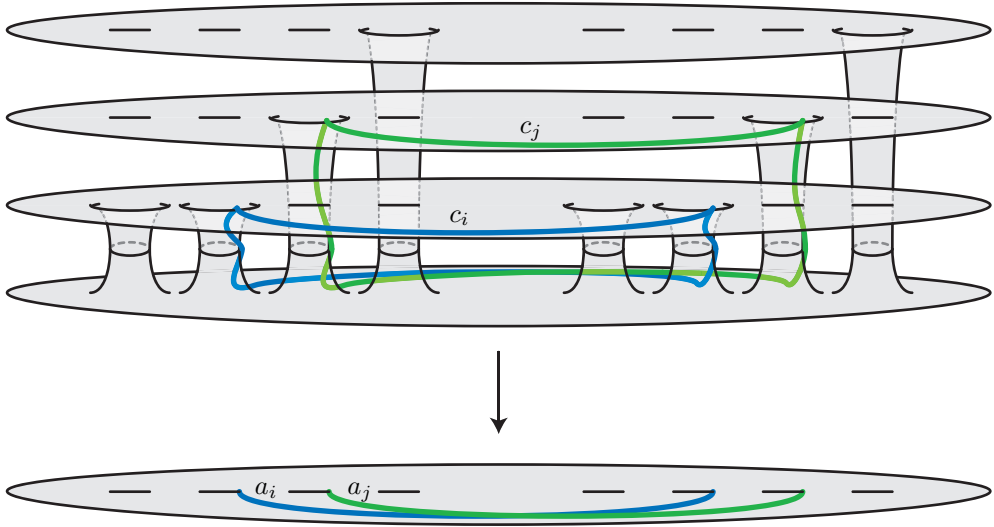}}} 
  \relabel{a}{$a_i$} 
  \relabel {b}{$a_j$} 
  \relabel{c}{$c_i$} 
  \relabel {d}{$c_j$}    
  \endrelabelbox
        \caption{The branched cover $\Sigma_{(2,{\bf 1})}$ with the curves $c_i, 1\leq i\leq 2g+2,$ and $c_j, 2g+2<j\leq d,$ with $i$ and $j$ even, where $g$ is the genus of $\Sigma.$}
	\label{fig:bcover+gen1}
\end{figure}

\begin{figure}[ht]
  \relabelbox \small 
 {\epsfysize=2.75truein\centerline {\epsfbox{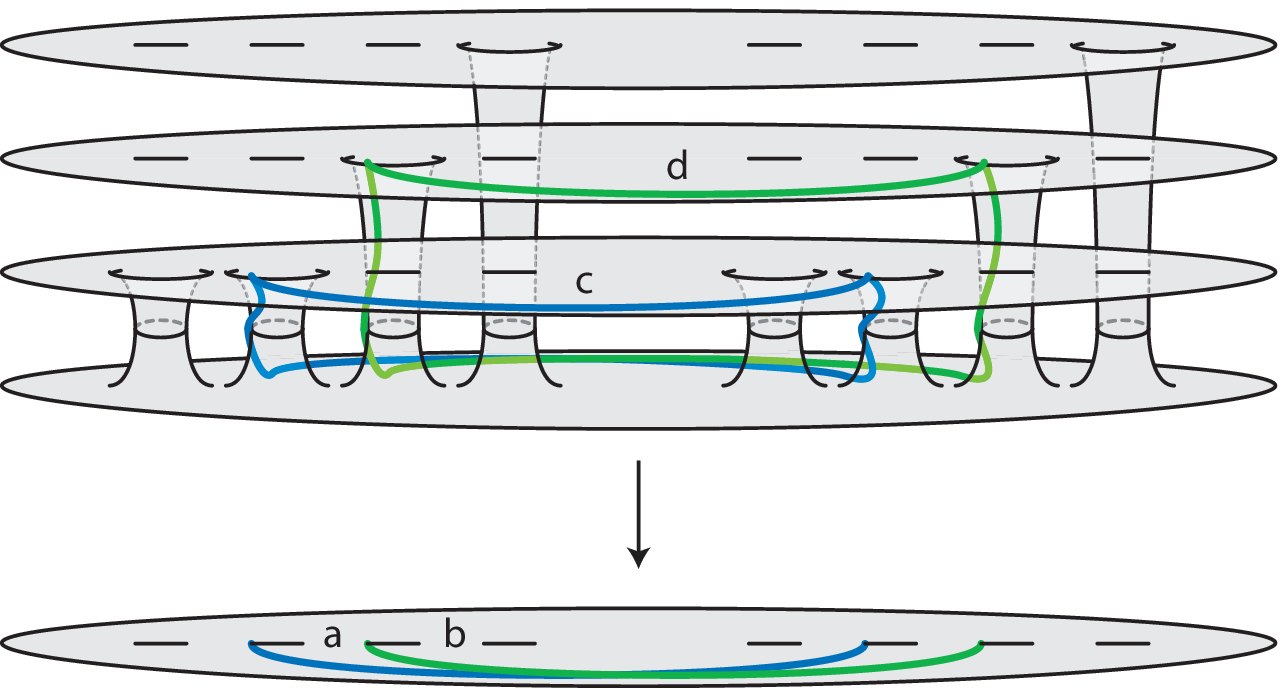}}} 
  \relabel{a}{$a_i$} 
  \relabel {b}{$a_j$} 
  \relabel{c}{$c_i$} 
  \relabel {d}{$c_j$}    
  \endrelabelbox
          \caption{The branched cover $\Sigma_{(2,{\bf 1})}$ with the curves $c_i, 1\leq i\leq 2g+2,$ and $c_j, 2g+2<j\leq d,$ with $i$ and $j$ odd, where $g$ is the genus of $\Sigma.$}
	\label{fig:bcover+gen2}
\end{figure}
Since $\tau_i$ lifts to the right handed Dehn twist $D_{c_i}$ we clearly see that $\rho_{(2,{\bf 1})}=\displaystyle\prod_{i=1}^d D_{c_i}.$
\end{proof}

\begin{proof}[Proof of Theorem~\ref{thm:monpdiscon} in the general case] 
Again, by Proposition~\ref{prop:monodromy split}, it is enough to find a factorization of the cable of the open book with $\phi = \Id$, keeping track of the nodules, and so we make that simplification again.  The factorization is again a lift of a braid factorization of $B_p$ from Lemma~\ref{realizedisconnected}. Specifically consider the disk $D$ in $\R^2$ intersecting the $x$-axis and let 
\[
x_{1,1},\ldots, x_{1,d},x_{2,1},\ldots,x_{2,d},\ldots, x_{p,1},\ldots, x_{p,d}
\] 
be points on the $x$-axis, again ordered from right to left. (See Figure~\ref{fig:basedisk3}.)  Let $D_1,\ldots, D_p$ be disjoint disks in $D$ such that $D_i$ contains the points $x_{i,1},\ldots, x_{i,d}.$  Moreover let $D_{i,i+1}, i=1,\ldots, p-1$ be larger disks in $D$ engulfing adjacent pairs of disks: $D_{i,i+1}$ contains the disks $D_i$ and $D_{i+1}$ and is disjoint from the other $D_j$. Finally let  $a_{i,j}$ be the embedded arc in $D_{i,i+1}$ with negative $y$-coordinate on its interior that connects $x_{i,j}$ to $x_{i+1,j}$ as indicated in Figure~\ref{fig:basedisk3}. The braid $B_p,$ thought of as an element of the mapping class group, is given as $B_p=\prod_{j=1}^{p-1}\prod_{i=1}^d \tau_{p-i,j}$ where $\tau_{i,j}$ is a right handed half twist exchanging $x_{i,j}$ and $x_{i+1,j}$ in a small neighborhood of the arc $a_{i,j}.$ Each $a_{i,j}$ lifts to a simple closed curve $c_{i,j}$ in $\Sigma_{(2,{\bf 1})}$ (and several arcs).  Since $\tau_{i,j}$ lifts to the right handed Dehn twist $D_{c_{i,j}}$, the factorization of $B_p$ gives the desired factorization: $\rho_{(p,{\bf 1})}= \displaystyle\prod_{j=1}^{p-1} \displaystyle\prod_{i=1}^{d} D_{c_{p-i,j}}.$
\end{proof}

\begin{rem}
From the proof of Theorem~\ref{thm:stabilization} we know that the monodromy of the $(p,{\bf q})$--cable of an open book decomposition $(\Sigma, \phi)$ can be constructed from the $(p,{\bf 1})$--cable by stabilization. While it would be nice to have an explicit description of the monodromy it is somewhat difficult to write down and we leave this to future work.
\end{rem}

\subsection{Connected Binding} \label{sec:connected binding}

In this subsection we write find the monodromy of the $(2,2)$--cable and the $(p,1)$ cable of an open book with connected binding. The $(2,2)$--cable is more or less done in the previous subsection, but explicitly derive it here as we will need it in our applications in Section~\ref{applications} (it also helps cement  the ideas from the last subsection before we move onto the more difficult monodromy computations for the $(p,1)$--cable). It is interesting to contrast the monodromies constructed in this section as we see the $(p,1)$--cable requires some explicit left-handled Dehn twists. This observation is a key to construction Stein fillable contact structures supported by open books whose monodromy is not a composition of positive Dehn twists. 

\subsubsection{The $(2,2)$--cables of open books with connected bindings}
As discussed in the proof of Theorem~\ref{thm:stabilization}, the $(2,2)$--cable of an open book with connected binding is the natural object you get by doubling the branch locus as in the proof of Theorems~\ref{thm:stabilization} and~\ref{thm:monpdiscon} and this braid has a positive braid factorization which lifts to a factorization of $\rho_{(2,2)}$, the rotation map in the monodromy of the cable.  To obtain a more convenient and symmetric expression for $\rho_{(2,2)}$ we choose a different conjugacy representative of the braid, see Figure~\ref{fig:22braid} (the conjugation is by a half twist on the lower $(2n+1)$ strands).
\begin{figure}[htp]
  	\labellist 
 	\small \hair 2pt
		\pinlabel 1 at -8, 1
		\pinlabel	2g+2 at -10, 25
		\pinlabel 2g+1 at -10, 17
		\pinlabel	4g+2 at -10, 41
	\endlabellist
 	\includegraphics[width = 2in]{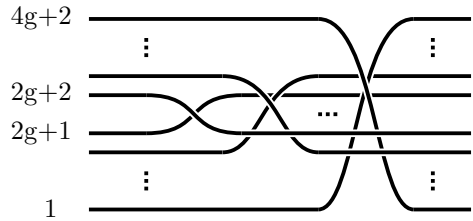}
	\caption{Branch locus of the $(2,2)$-cable of the trivial genus $g$ open book.} \label{fig:22braid}
\end{figure}
This construction is fundamentally different than the construction in Theorem~\ref{thm:connected binding}, below, where it is shown that the rotation map of the $(2,1)$--cable does not admit a positive factorization.  There is, however, a (single)  
positive Hopf stabilization taking the $(2,1)$--cable to the $(2,2)$--cable ({\em cf.\ }Lemma~\ref{lem:gencable}), which gives a factorization of the monodromy of the $(2,2)$--cable.  The equivalence of these two presentations is discussed in Section~\ref{sec:stabilizing to 2,2}.

\begin{prop} \label{prop:22cable}
Let $(\Sigma, \phi)$ be an open book with connected binding and let $g = g(\Sigma)$ be the genus of $\Sigma$.  The $(2,2)$--cable of $(\Sigma, \phi)$ can be described abstractly as $(\Sigma_{(2,2)}, \phi_{(2,2)})$ where $\Sigma_{(2,2)}$ has genus $2g$ and 2 boundary components, and $\phi_{(2,2)} = \rho_{(2,2)} \circ \tilde{\phi}$.  The map $\rho_{(2,2)}$ is a lift of the braid $R_{(2,2)}^g$ shown in Figure \ref{fig:22braid} and has a factorization 
$$\rho_{(2,2)} = D_{d_{2g+1}} \circ \cdots \circ D_{d_{1}},$$
where the $D_{d_i}$ are Dehn twists about the curves $d_i$ shown in Figure \ref{fig:22dehntwist}.
\end{prop}

The proof of the proposition is contained in the following two lemmas.

\begin{figure}[htp]
  	\labellist 
 	\small \hair 2pt
		\pinlabel $d_1$ at 165, 59
		\pinlabel $d_2$ at 203, 61
		\pinlabel $d_{2g}$ at 265, 53
		\pinlabel $d_{2g+1}$ at 286, 53
	
	\endlabellist
	\includegraphics[width = 4in]{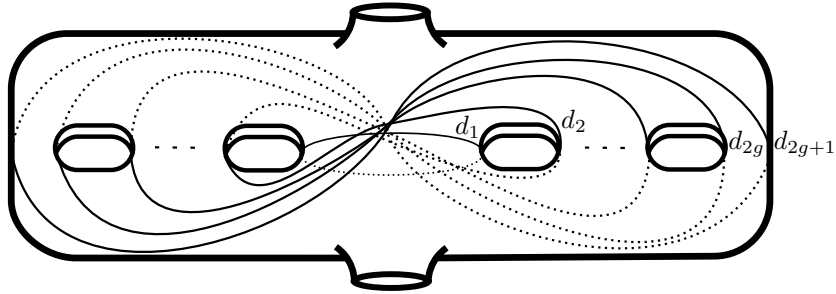}
	\caption{Dehn twists used in the factorization of the rotation map $\rho_{(2,2)}(\Sigma)$ of the $(2,2)$-cable of a genus $g$ open book.} \label{fig:22dehntwist}
\end{figure}

\begin{lem}  \label{lem:22factorization} The braid $R_{(2,2)}^g$ shown in Figure \ref{fig:22braid} has factorizations
$$R_{(2,2)}^g = \Delta \Delta_1^{-2} \Delta_2^{-2}$$
and
$$R_{(2,2)}^g = b_1 \cdots b_{2g+1},$$
where $b_i$ is a braid half twist about the arc $a_i$ shown in Figure~\ref{fig:22arc}, $\Delta$ is the Garside half twist on all $4g+2$ strands, $\Delta_1$ is the  half twist on the first $2g+1$ strands, and $\Delta_2$ is the  half twist on the last $2g+1$ strands.
\end{lem}
\begin{figure}[htp]
  	\labellist 
 	\small \hair 2pt
		\pinlabel $a_1$ at 57, 18
		\pinlabel $a_{2g}$ at 73, 26
		\pinlabel $a_{2g+1}$ at 66, 35
		
	\endlabellist
	\includegraphics[width = 2in]{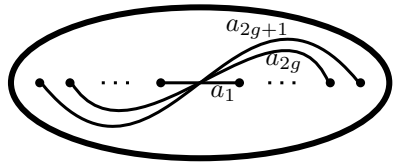}
	\caption{Arcs presenting a braid factorization of the braid $R_{(2,2)}^g$.} \label{fig:22arc}
\end{figure}

\begin{proof} ({\em cf.\ }Lemma~\ref{realizedisconnected})  That the braid half twists give a factorization of $R_{(2,2)}^g$ is obvious.  To see the other factorization recall that Figure~\ref{fig:22braid} is obtained from Figure~\ref{fig:unwindB} by conjugating by $\Delta_1$. From this one may easily see the new factorization. 
\end{proof}

\begin{lem} The rotation map $\rho_{(2,2)}$ has a factorization
\[
\rho_{(2,2)} = D_{d_{2g+1}}\circ \cdots \circ D_{d_{1}},
\]
where the $D_{d_i}$ are Dehn twists about the curves $d_i$ shown in Figure \ref{fig:22dehntwist}.
\end{lem}

\begin{proof} ({\em cf.\ }Theorem \ref{thm:monpdiscon})    
As discussed in the previous subsection, in the two-fold branched cover the braid axis for $R_{(2,2)}^g$ lifts to the $(2,2)$-cable of the braid axis for the trivial  $2g+1$-stranded braid.  This is then the cable of the trivial open book and hence the monodromy is exactly $\rho_{(2,2)}$.  Any factorization of $R_{(2,2)}^g$ then lifts to a factorization of $\rho_{(2,2)}$. The factorization of $R_{(2,2)}^g$ from Lemma~\ref{lem:22factorization} in particular gives the desired factorization of $\rho_{(2,2)}$ since  the braid arcs for $R_{(2,2)}^g$ lift to the curves shown in Figure \ref{fig:22dehntwist}.
\end{proof}

\subsubsection{The $(p,1)$--cables of open books with connected bindings}
The goal of this subsection is to present the monodromy of the $(p,1)$--cable of an open book with connected binding.  In Theorem~\ref{thm:stabilization}, we used a branched cover construction of $\left(\Sigma, \Id\right)$ over $\left(D^2, \Id\right)$  to also build $\left(\Sigma_{(p,{1})}, \phi_{(p,{1})}\right)$.  When $\Sigma$ has only one boundary component, however, the twofold branched cover used to construct $\Sigma$ is non-trivial along the boundary. So while we used the same ideas to construct the monodromy of the $(p,2)$--cables, there is no cable of the unknot which lifts to the $(p,1)$--cable of $\partial \Sigma$. To construct this cable we need a different approach, in particular, a different branched cover.   
\begin{lem}\label{11cable}
Let $M$ be the manifold obtained from the trivial open book $\left(\Sigma, \Id\right)$ and denote the binding by $C$.  The $p$-fold cyclic cover of $M$ branched over $C_{(1,1)}$ is again $M$. Moreover, we can assume $C_{(1,1)}$ is transverse to the pages of the open book and then $C$ lifts to $C_{(p,1)}$ and the pages lift to pages of the cabled open book. 

In other words, the $(p,1)$--cable of $\left(\Sigma, \Id\right)$ can be seen as the $p$-fold cyclic cover of $\left(\Sigma, \Id\right)$ branched over the $(1,1)$--cable of the binding ${C}.$
\end{lem}

\begin{proof}
We begin by commenting that it is essential here that the monodromy is the identity. In this case notice that the $p$-fold cyclic branched cover over the binding of $\left(\Sigma, \Id\right)$ yields the same manifold. Moreover, the branched cover takes the $(1,1)$--cable of the binding to the $(p,1)$--cable of the binding. Now reversing the roles of the binding and its $(1,1)$--cable (which we can do as they are isotopic) yields the desired result.  
\end{proof}

Let us establish some notation.  As discussed at the beginning of this section, the page of the cabled open book decomposition $\Sigma_{(p,{1})}$ is made up of $p$-copies of $\Sigma,$ denoted $\Sigma_i,$ for $i=1,\ldots, p,$ called nodules, and a base component $C$, which is a disk with $p$ subdisks removed. We explicitly realize $\Sigma_{(p,{1})}$ in $\R^3$ so that the nodules have $z$-coordinate non-negative, the base component $C$ is in the $xy$-plane and consists of the unit disk minus $p$ open disks arranged cyclically around the origin, and the entire surface is invariant under a $\frac{2\pi}{p}$ rotation about the $z$-axis.  See Figure~\ref{fig:base disk}.  We are given a reference arc, $d_j$, in $C$ that connects the $j$th and $j+1$st nodules.  Denote a neighborhood of $d_j$ and the nodules $\Sigma_j$ and $\Sigma_{j+1}$ by $\Sigma_{j,j+1}.$  Notice that we can fix an identification of $\Sigma$ with $\Sigma_1$ and then identify $\Sigma$ with the remaining $\Sigma_i$ by rotating about the $z$-axis. Under this realization of $\Sigma_{(p,1)}$ there is an natural identification of $\Sigma_{(2,1)}$ with $\Sigma_{j,j+1}.$
\begin{figure}[ht]
  \relabelbox \small 
 {\epsfysize=2truein\centerline {\epsfbox{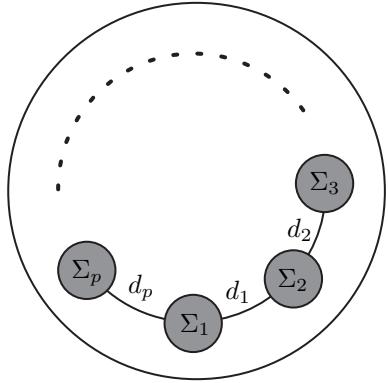}}} 
  \relabel{1}{$\Sigma_1$} 
  \relabel {2}{$\Sigma_2$} 
  \relabel{3}{$\Sigma_3$} 
  \relabel {p}{$\Sigma_p$}  
    \relabel{a}{$d_1$} 
  \relabel {b}{$d_2$}
     \relabel {c}{$d_p$}
  \endrelabelbox
        \caption{The base disk for $\Sigma_{(p,{1})}$.  The arcs $d_j$ are used to determine the subsurfaces $\Sigma_{j, j+1}$. }
\label{fig:base disk}
\end{figure}

\begin{thm} \label{thm:connected binding} 
Let $\left(\Sigma, \phi\right)$ be an open book with connected binding.  Then the monodromy $\phi_{(p,1)}$ of the $(p,1)$--cable of $\left(\Sigma, \phi\right)$ can be written as 
\[
\phi_{(p,1)} = \displaystyle\prod_{j=2}^{p} \partial_j^{-1} \circ \displaystyle\prod_{j=1}^{p-1} T_{p-j} \circ \tilde{\phi}.
\]
Here $\partial_j$ is the Dehn twist about the boundary of the $j$th nodule $\Sigma_j$ and $\tilde{\phi}$ is the lift of $\phi$, acting on the first nodule.  The map $T_{j}$ is the diffeomorphism of $\Sigma_{j,j+1}$ that, when $\Sigma_{j,j+1}$ is identified with the surface $\Sigma_{(2,1)}$ in Figure~\ref{fig:twonodules} as discussed above, is a lift of the Garside braid half-twist and can be written 
\[
T_j = \partial_j^{-1} \circ (D_{2d-1}) \circ (D_{2d-2} \circ D_{2d-1}) \circ \cdots \circ (D_{2} \circ \cdots \circ D_{2d-1})\circ (D_{1} \circ \cdots \circ D_{2d-1})
\]
where $D_i$ is a right handed Dehn twist along $c_i.$
\end{thm}

We point out that, unlike the previous cases, there is in general no positive Dehn twist presentation of $\rho_{(p,1)}(\Sigma)$ (see Theorem~\ref{thm:counterexample}), and so we content ourselves with the presentation given.

\begin{proof}
We again appeal to Proposition~\ref{prop:monodromy split} and focus on determining the Dehn twist presentation for $\rho_{(p,{1})}(\Sigma)$. From Lemma~\ref{11cable} we see that the monodromy of $\left(\Sigma_{(p,{1})}, \rho_{(p,{1})}(\Sigma)\right)$ can be computed from lifting the braid representation of the $(1,1)$--cable $K$ of the binding of $\left(\Sigma, \phi\right)$ to the $p$-fold cyclic cover branched over $K.$ The braid representing $K$ thought of as an element of the mapping class group is $B=D_c\circ D_{c'}^{-1}$ where $c$ is a simple closed curve parallel to $\partial \Sigma,$ $c'$ is a copy of $c$ pushed a little further into $\Sigma,$ and $D_c$ and $D_{c'}$ are Dehn twists about the given curves. If we choose a point $x$ between $c$ and $c'$ then it will trace out the $(1,1)$--cable of the binding. See Figure~\ref{fig:11cable}.
\begin{figure}[htp]
  \relabelbox \small 
 {\epsfysize=1.5truein\centerline {\epsfbox{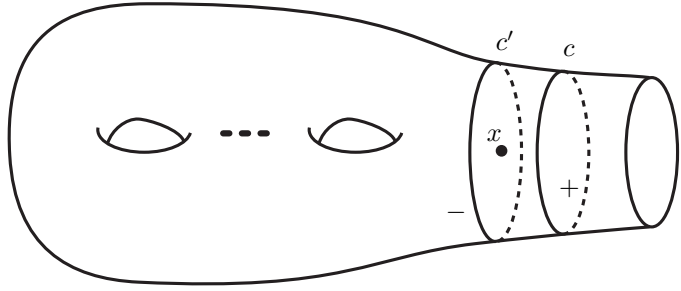}}} 
  \relabel{c}{$-$} 
  \relabel {x}{$x$} 
  \relabel {l}{$+$} 
  \relabel {a}{$c'$} 
  \relabel {b}{$c$} 
  \endrelabelbox
        \caption{Braid picture of the (1,1)-cable of the binding of the page $\Sigma$.}
\label{fig:11cable}
\end{figure}

Lifting $B$ to the $p$-fold cyclic cover, we first note that $c'$ will lift to $p$ simple closed curves $c'_i$, $i =1, \dots p$, with each $c'_i$ parallel to the boundary of the nodule $\Sigma_i$. Thus $D_{c'}^{-1}$ will lift to the diffeomorphism $\partial_1^{-1}\circ\ldots \circ \partial_p^{-1}$ (where we use the notation $\partial_i$ for $D_{c'_i}$ as in the statement of the theorem).  The curve $c$ does not lift to the cover, but we can still lift $D_c.$  Referring back to our description of $\Sigma_{(p,{1})}$ in $\R^3$ above, assume $\epsilon>0$ is chosen so that all the nodules of  $\Sigma_{(p,{1})}$ are contained in the cylinder about the $z$-axis of radius $1-\epsilon.$ Now let $r_p$ the restriction to $\Sigma_{(p,{1})}$ of the map that is rotation by $\frac{2\pi}{p}$ about the $z$-axis for all points within the cylinder of radius $1-\epsilon,$ the identity outside the cylinder of radius $1-\frac \epsilon 2$ and interpolates between the two (keeping $z$-coordinate constant) in between. By noting that the generating deck transform for the $p$-fold cover of $\Sigma_{(p,{1})}$ over $\Sigma$  branched over $x$ is just the restriction to $\Sigma_{(p,{1})}$ of rotation about the $z$-axis by $\frac {2\pi}{p}$ one may easily check that $r_p$ is the lift of $D_c$ to $\Sigma_{(p,{1})}$, {\em cf}.\ \cite[Figure 3.1]{Montesinos-AmilibiaMorton91}. 
Thus we see that 
\[
\rho_{(p,{1})}(\Sigma)= \partial_1^{-1}\circ\ldots \circ \partial_p^{-1}\circ r_p.
\]
So to complete the proof we need a Dehn twist presentation of $r_p.$

As before, the idea will be to find a suitable presentation when $p=2$ and show that the composition of different lifts, acting on each $\Sigma_{j,j+1}$, $j=1, \dots, p-1$, gives the general case. When $p=2$ the rotation $r_2$ is particularly easy to describe.  It occurs as the lift of the {\em Garside half twist braid} under a 2-fold branched cover.  More specifically, Figure~\ref{fig:garside twist} shows the 2-fold cover which describes $\Sigma_{(2,2)},$ the page of a $(2,2)$--cable of the original open book.  Here $\Sigma_{(2,2)}$ is the 2-fold cover of $D^2$ branched over $2d=2(2g+1)$ points. Let $\psi$ be the diffeomorphism of $\Sigma_{(2,2)}$ that fixes the boundary, rotates the the figure (outside a small neighborhood of the boundary)  through an angle $\pi$, and interpolates between the two maps in between.  The surface $\Sigma_{(2,1)}$ is obtained from $\Sigma_{(2,2)}$ by capping off one of its boundary components.  Moreover, $r_2$ is the extension of $\psi$ to $\Sigma_{(2,1)}$. So we are left to give a Dehn twist presentation of $\psi.$

\begin{figure}[htp]
\centering
\labellist\small\hair 1.pt
\pinlabel {$c_1$} at 524 273 
\pinlabel {$c_2$} at 486 288
\pinlabel {$c_{d}$} at 319 270
\pinlabel {$c_{d-1}$} at 385 286
\pinlabel {$c_{d+1}$} at 188 286
\pinlabel {$c_{2d-2}$} at 87 285
\pinlabel {$c_{2d-1}$} [r] at 31 259
\pinlabel {$l$} at 260 245
\pinlabel {$g(l)$} at 317 212
\pinlabel {$F$} at 137 285
\pinlabel {$B$} at 137 235
\pinlabel {$F$} at 435 285
\pinlabel {$B$} at 435 235
\endlabellist
\includegraphics[width=5in]{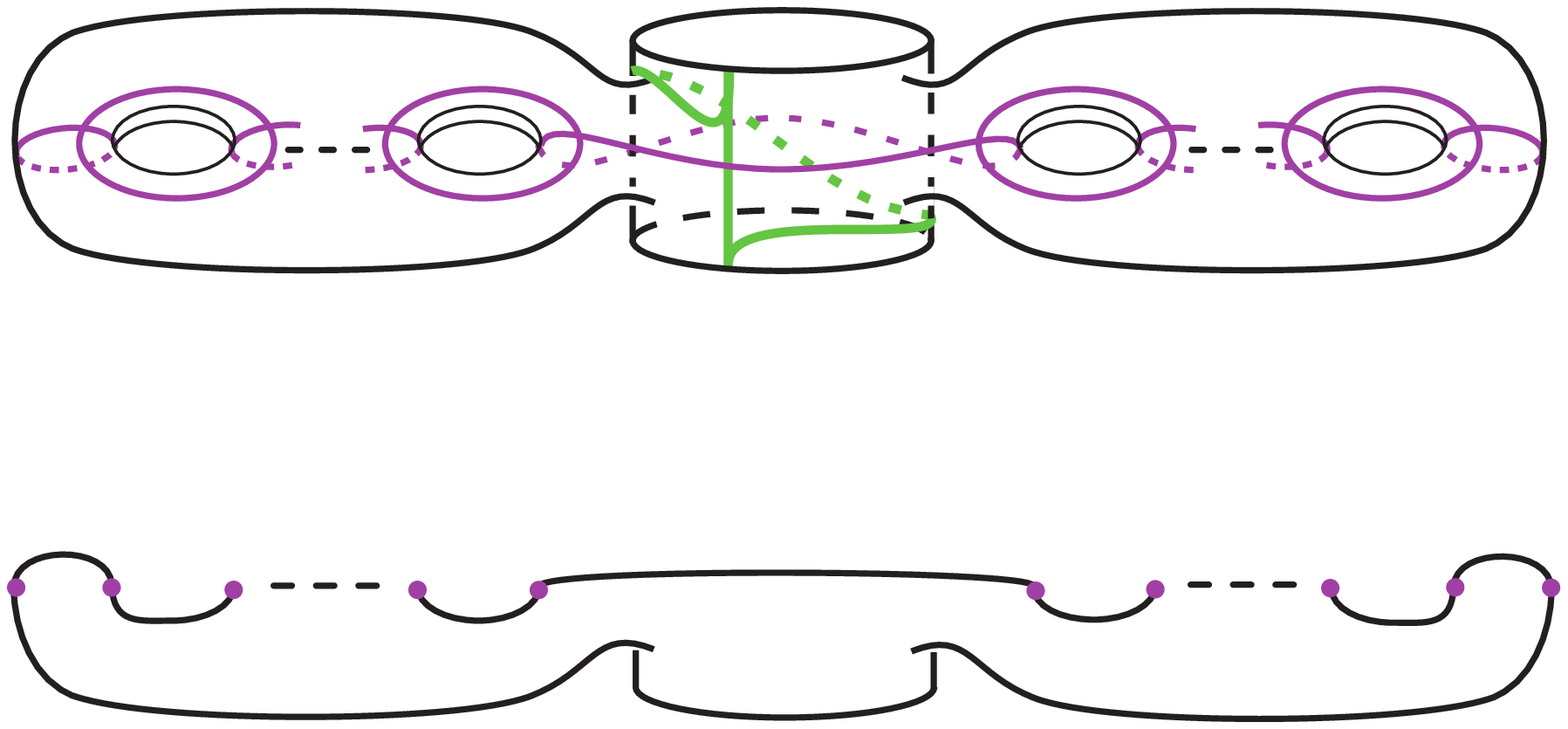}
\caption{Branched cover picture of the (2,2)-cabled surface $\Sigma_{(2,2)}$.  To see the lift $g$ of the Garside half twist rotate both the base and cover $180^\circ$ and then slide each boundary back to where it started.  The arc $l$ and it's image $g(l)$ under the rotation are shown, as are the lifts $c_i$ of the standard arcs $a_i$ connecting adjacent marked points in $D^2$.  The two sheets are labeled F and B and the spin map takes each sheet in nodule $\Sigma_1$ to its counterpart in nodule $\Sigma_2$.}
\label{fig:garside twist}
\end{figure}

Identify the base of the branched covering with the unit disk in $\R^2$ and place the branched point $x_1,\ldots, x_{2d}$ on the $x$-axis, ordered left to right so that they are symmetric about the origin.   Then (up to isotopy) $\psi$ covers the diffeomorphism $\psi'$ of $D^2$ that fixes the boundary, rotates the complement of a small symmetric neighborhood of the boundary (that contains no branched points) counterclockwise by $\pi$, and interpolates between the two maps elsewhere --- if we forget about the boundary, this is just a rotation of the entire disk through an angle $\pi$ --- this is just the Garside half twist braid $\Delta$. If we let $\sigma_i$ be the standard generators of the braid group (that is they are diffeomorphisms of $D^2$ that exchange $x_i$ and $x_{i+1}$ via a right handed twist supported in a neighborhood of an arc $a_i$ on the $x$-axis connecting them) then $\psi' = \Delta$ has factorization 
\[
\Delta=(\sigma_{2d-1} \cdots  \sigma_{1}) (\sigma_{2d-1} \cdots \sigma_{2}) \cdots (\sigma_{2d-1} \sigma_{2d-2} ) \cdot \sigma_{2d-1}.
\]
Each arc $a_i$ lifts to a simple closed curve $c_i$ in $\Sigma_{(2,2)},$ see Figure~\ref{fig:garside twist}, and the lift of the diffeomorphism $\Delta$ is given by $(D_{2d-1}) \circ (D_{2d-2} \circ D_{2d-1}) \circ \cdots \circ(D_{2} \circ \cdots \circ D_{2d-1})\circ (D_{1} \circ \cdots \circ D_{2d-1}),$ where $D_i$ is a right handed Dehn twist about $c_i.$  This gives the factorization 
\[
\phi_{(2,1)}(\Sigma) = \partial_2^{-1} \circ \partial_1^{-1} \circ (D_{2d-1}) \circ (D_{2d-2} \circ D_{2d-1}) \circ \cdots \circ (D_{2} \circ \cdots \circ D_{2d-1})\circ (D_{1} \circ \cdots \circ D_{2d-1}) \circ  \tilde{\phi}.
\]
 
To normalize the presentation of $r_2$ in preparation for the $p\not=2$ case, we pick the chain of curves $c_1, \dots, c_{d-1}$ and a proper arc $a$ in $\Sigma$ shown in Figure~\ref{fig:twonodules}, here $a$ is $c_d$ intersected with the nodule. The surface $\Sigma_{(p,1)}$ as described before the theorem consists of the base surface $C$ and the nodules $\Sigma_1,\ldots, \Sigma_p$ sitting symmetrically around the $z$-axis in $\R^3.$ We identify $\Sigma$ with $\Sigma_1$ and then with the other $\Sigma_i$ by rigid rotation about the $z$-axis. Under this identification we denote by $c_{i,j}$ the curve $c_j$ on $\Sigma_i$  for $1\leq j\leq d-1,$ and the curve $c_{2d-j}$ on $\Sigma_{i+1}$ for $d+1\leq j\leq 2d-1.$ Notice that there is some repetition among the $c_{i,j}$.  In particular, $c_{i,j}=c_{i+1, 2d-j}.$ Finally denote by $c_{i,d}$ the curve obtained by taking the union of $a\subset \Sigma_i,$ $a\subset \Sigma_{i+1}$ and two parallel copies of $d_j$ that connect the end points of the $a$'s.  Notice that for a fixed $i$ the curves $c_{i,j}$ correspond to the curves in Figure~\ref{fig:twonodules} under the identification of $\Sigma_{(2,1)}$ with $\Sigma_{i,i+1}.$
 
\begin{figure}[ht]
\centering
	\labellist 
	\small\hair 2pt
		\pinlabel {$c_{1}$} at 239 239
		\pinlabel {$c_{2}$} at 228 200
		\pinlabel {$c_{d}$} at 173 13
		\pinlabel {$c_{d+1}$} at 120 116
		\pinlabel {$c_{d-1}$} at 226 114
		\pinlabel {$c_{2d-2}$} at 120 177
		\pinlabel {$c_{2d-1}$} at 113 236
		\pinlabel {$F$} at 112 152
		\pinlabel {$B$} at 70  152
		\pinlabel {$F$} at 232 152
		\pinlabel {$B$} at 277 152
		\pinlabel {$S_2$} at 87 48
		\pinlabel {$S_1$} at 260 46
		\pinlabel {$d_1$} at 174 40
	\endlabellist
\includegraphics[width = 4.in]{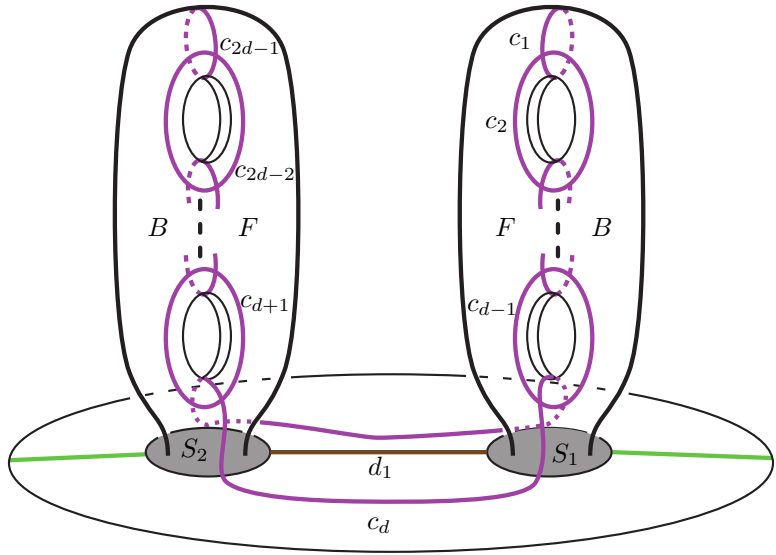}
\caption{The surface $\Sigma_{(2,1)}$. The spin map $r_2$ can be seen by rotating the picture $180^\circ$ and then sliding the boundary clockwise back to where it began.}
\label{fig:twonodules}
\end{figure}

We set
\[
s_i=(D_{2d-1}) \circ (D_{2d-2} \circ D_{2d-1}) \circ \cdots \circ (D_{2} \circ \cdots \circ D_{2d-1})\circ (D_{1} \circ \cdots \circ D_{2d-1})
\]
and notice that $s_i$ is simply $r_2$ acting on $\Sigma_{i,i+1}$ under our above identification. That is $s_i$ acts on $\Sigma_{i,i+1}$ by exchanging the nodules $\Sigma_i$ and $\Sigma_{i+1}$ and on the base component of $\Sigma_{i,i+1}$ it acts as shown in the middle part of Figure~\ref{fig:local models}. 

\begin{figure}[ht]
\labellist \small \hair 2pt
	\pinlabel $S_{j}$ at 108 90
	\pinlabel $S_{j+1}$ at 108 200
\endlabellist
\includegraphics[width = 1.25 in]{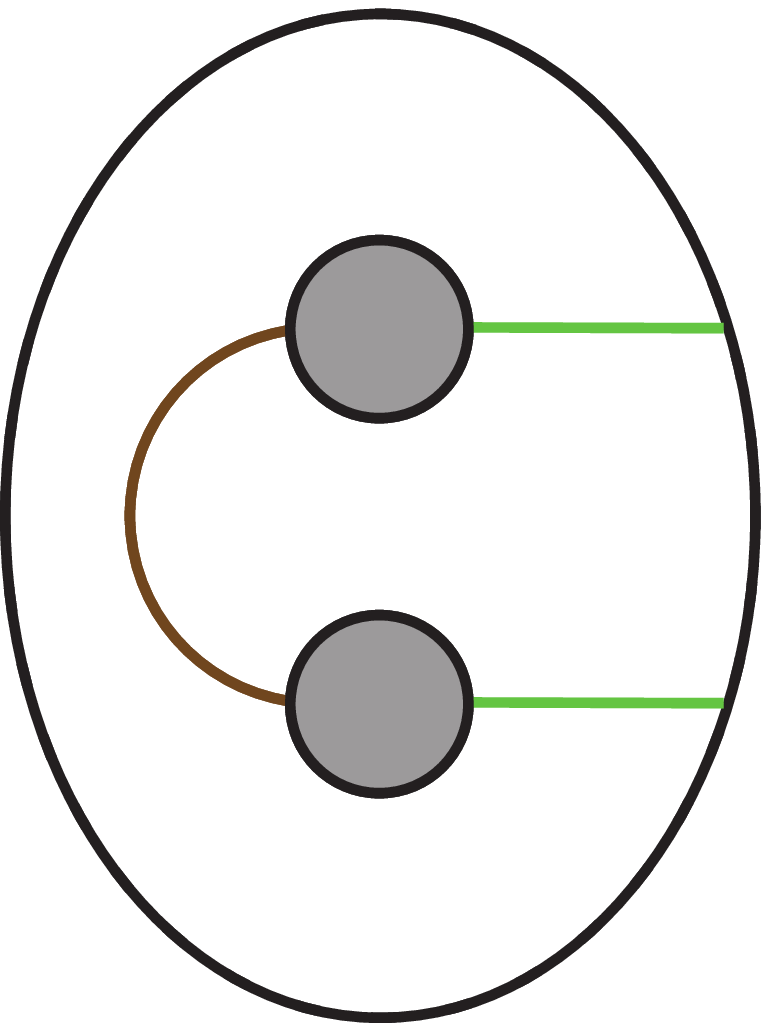}
\hspace{.25 in}
\includegraphics[width = 1.25 in]{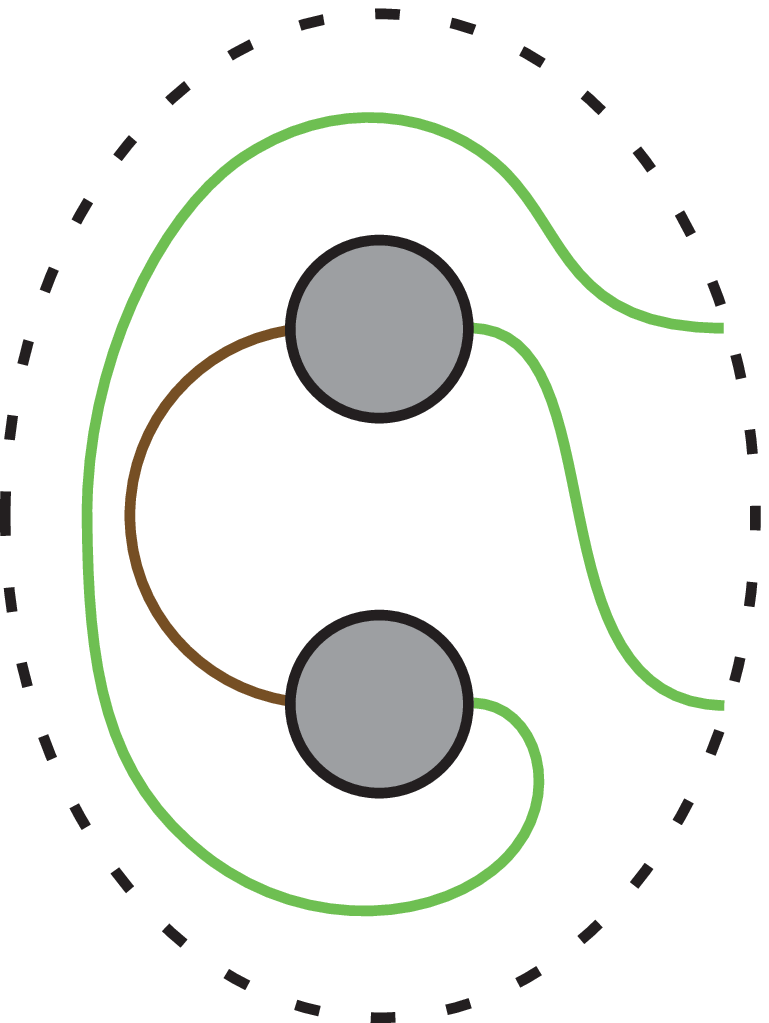}
\hspace{.25 in}
\includegraphics[width = 1.25 in]{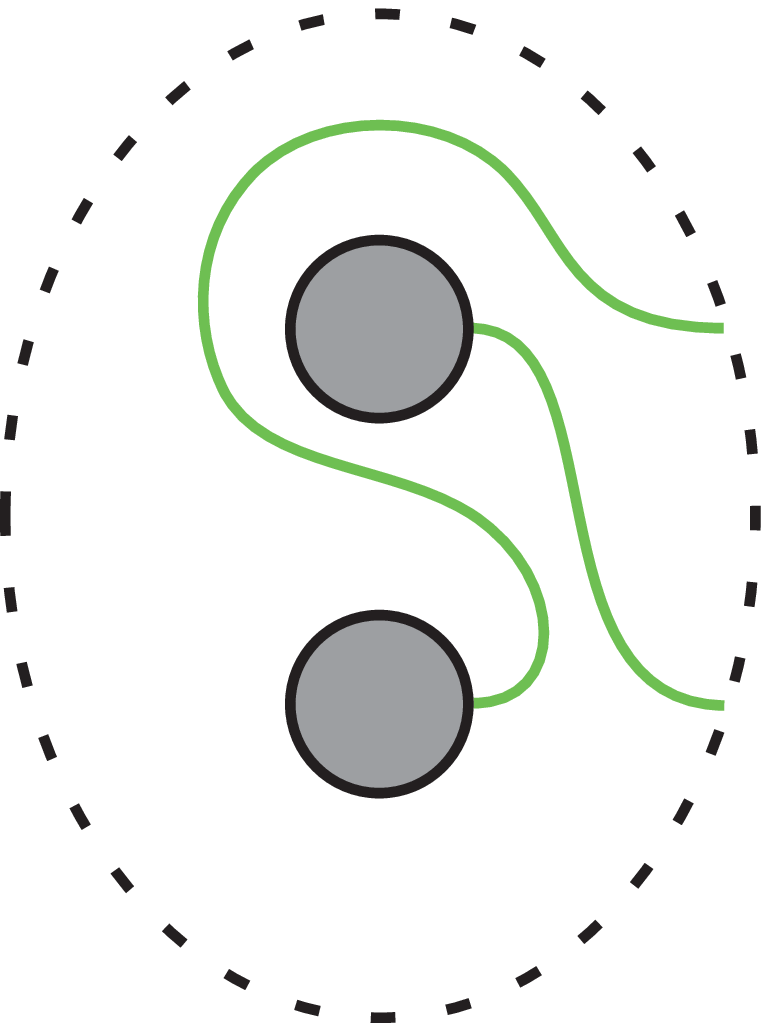}
\caption{Local picture of the spin (middle) diffeomorphism $s_i$ and dosado (right) diffeomorphism $T_i$ showing their framing difference.}
\label{fig:local models}
\end{figure}
If we set 
\[
T_i=\partial_{i}^{-1} \circ s_i
\]
then this is a diffeomorphism of $\Sigma_{i,i+1}$ that acts on the nodules in the same way $s_i$ does and acts on the base component as shown on the right of Figure~\ref{fig:local models}. Consider the composition $T_{p-1}\circ \cdots \circ T_{1}.$ This is a diffeomorphism of $\Sigma_{(p,1)}$ that acts on the nodules just as $r_p$ does, and on the base acts as shown in Figure~\ref{fig:base disk rotations}.

\begin{figure}[ht]
\relabelbox \small
{\centerline {\epsfbox{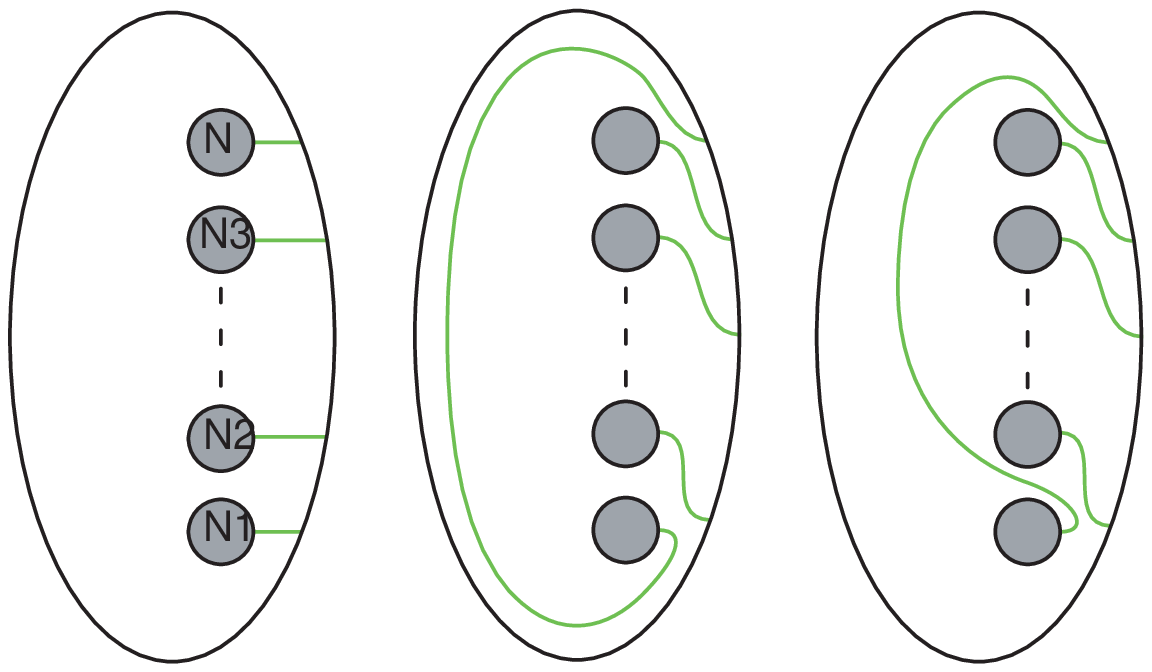}}}
\relabel{N1}{$S_1$}
\relabel{N2}{$S_2$}
\relabel{N3}{$S_{p-1}$}
\relabel{N}{$S_p$}
\endrelabelbox
\caption{On the left, the base disk with nodules labeled and framings given.  The middle shows how the rotation $r_p$ acts on the framed nodules.  The right shows the framing after applying the $p-1$ dosado maps between the $j$th and $j+1$st nodules.}
\label{fig:base disk rotations}
\end{figure}

Thus we can write 
\[
r_p=\partial_1 \circ \displaystyle\prod_{i=1}^{p-1} (T_{p-i})
\]
\end{proof}

\section{Applications}\label{applications}
In this section we give two applications of our monodromy computations from the previous section. In particular we study the monodromy of Stein fillable contact structures and prove the existence of many interesting monoids in the mapping class group of a surface. We also prove Proposition~\ref{prop:otexceptional} by exhibiting many negative exceptional cables that produce tight contact structures. 

\subsection{Stein fillable open books without a positive monodromy.}
We apply Theorem~\ref{thm:main-cable} and the factorization given in Theorem~\ref{thm:connected binding} to show that there exists open books supporting Stein fillable contact structures whose monodromy cannot be written as a product of positive Dehn twists.  The particular examples we find are $(2,1)$--cables of genus one open books compatible with the unique tight contact structure $\xi_{std}$ on the lens spaces $L(p,p-1)$ for $p\geq1$, where we include $S^3$ as the lens space $L(1,0)$.  

\subsubsection{The examples}
In the following lemma, we think about a length of a homeomorphism of a genus two surface as the algebraic length of a presentation as a product of Dehn twists about non-separating curves, which, for our purposes, counts right-handed Dehn twists positively and left-handed Dehn twists negatively.

To do this, we define an invariant of an element $m \in Map^+(\Sigma)$ which is related to the algebraic word length of such a factorization.  It is easy to see (from Wajnryb's presentation \cite{Wajnryb99}, for example) that  the abelianization $A$ of $Map^+(\Sigma)$ is $\Z / 10 \Z$.  There is a particular generator in $A$ we want to consider, that of a Dehn twist about a non-separating curve $[D]$.  Since any two such Dehn twists are conjugate they all represent the same element in $A$ and so we can determine this class without reference to a particular curve.  Moreover, since $Map^+(\Sigma)$ is generated by Dehn twists about non-separating curves, this class generates the abelianization as well.  We define the algebraic length, denoted $|m|,$ of a mapping class $m$ to be the integer $l$ such that the class of $[m]$ in the abelianization is $[m]=l [D]$.  The following lemma is obvious from the discussion here.

\begin{lem} \label{lm:modten}  Let $\phi$ be a homeomorphism of a surface $\Sigma$ of genus two and with one boundary component.  The algebraic length of $\phi$, denoted as above by $|\phi|,$ is equal, modulo ${10}$, to the algebraic word length of any factorization of $\phi$ into a product of Dehn twists about non-separating curves.  In particular, the length of any such factorization is well-defined modulo ${10}$.
\end{lem}

This length combined with the following classification of symplectic fillings of $L(p,p-1)$ by Lisca give us the obstruction to a positive monodromy for our examples.

\begin{figure}[htp]
	\labellist 
	\small\hair 2pt
		
	\pinlabel $-2$ at 4 -10
	\pinlabel $-2$ at 74 -10
	\pinlabel $-2$ at 220 -10
	\pinlabel $-2$ at 290 -10
	\pinlabel \rotatebox{90}{$ \left\{ \phantom{\begin{matrix} B \\B\\B\\B\\B\\B\\B\\B\\B\\B\\B\\B \end{matrix} }\right.$} at 157 -20
	\pinlabel $p-1$ at 157 -45
	\endlabellist
\includegraphics[width=2in]{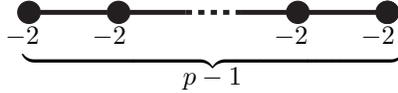}
\vspace{20pt}
\caption{Linear plumbing describing the unique minimal symplectic filling of $L(p,p-1)$. The vertices of the graph correspond to disk bundles over spheres with Euler number given by the label. An edge between two vertices denotes a plumbing between the spheres corresponding to the vertices. }
\label{fig:plumbing}
\end{figure}

\begin{thm}[Lisca 2004, \cite{Lisca04}] \label{thm:filling} Any minimal symplectic filling of the contact manifold \mbox{$(L(p,p-1), \xi_{std})$} is diffeomorphic to the plumbing described in Figure \ref{fig:plumbing}.
\end{thm}

When $p=1$ the manifold $L(1,0)$ is $S^3$, which Eliashberg proves has a unique minimal symplectic filling, namely $B^4,$ \cite{Eliash91}.  The particular case of $L(p,p-1)$ is discussed in detail as the Third Example of \cite[p.18]{Lisca04}.

\begin{proof}  [Proof of Theorem~\ref{thm:counterexample}]
As stated before, the manifold $L(p,p-1)$ has a unique tight contact structure $\xi_{std}$, which additionally admits a unique Stein filling (easily constructed from a Legendrian presentation of the plumbing diagram Figure \ref{fig:plumbing}).  The contact structure $\xi_{std}$ is supported by an annular open book whose monodromy consists of $p$ right-handed Dehn twists about the core of the annulus.  We can stabilize this open book to get the genus one open book which has monodromy $\phi = D_1^p \circ D_2$.  We will see that the $(2,1)$--cable $\phi_{(2,1)}$ of this open book does not have a positive factorization.  Theorem~\ref{thm:connected binding} gives us a factorization of the monodromy $\phi_{(2,1)}$ of the cable as 
\[
\phi_{(2,1)} =(D_4 \circ D_5)^{-6} \circ (D_1 \circ D_2)^{-6} \circ (D_5) \circ (D_4 \circ D_5)\circ \cdots \circ(D_1 \circ \cdots \circ D_5) \circ  D_1^p \circ D_2 ,
\] 
where $D_i$ is the right-handed Dehn twist about the curve $c_i$ in Figure \ref{fig:twonodules}.  We have written $\phi_{(2,1)}$ as a product of Dehn twists about non-separating curves, and our particular factorization has algebraic length $15 - 24 + p + 1 = p-8$ and so $|\phi_{(2,1)}| \equiv p-8$.  To see that $\phi_{(2,1)}$ has no positive factorization, we compare this to the necessary length of a  minimal symplectic filling of $\xi_{std}$.  

Any positive factorization of $\phi_{(2,1)}$ can be modified by chain relations to give a positive factorization whose terms are each Dehn twists about a non-separating curve.  Any factorization into positive Dehn twists about non-separating curves actually constructs a Stein filling (see \cite{Giroux02, Gompf98}).  Denote by $\mathcal{F}$ the positive factorization. The construction of the Stein filling starts with a 4-dimensional thickening of the page, and adds 4-dimensional (symplectic) 2-handles for each non-separating Dehn twist in the positive factorization of the monodromy.  For our cabled open book on $L(p,p-1)$, the page is a genus 2 surface, this means we are constructing a minimal symplectic manifold by a 0-handle, four 1-handles and $|\mathcal{F}|$, 2-handles which hence has Euler characteristic $1 - 4 + |\mathcal{F}|$, where $|\mathcal{F}|$ both is the algebraic length of the monodromy ($|\mathcal{F}| = |\phi_{(2,1)}|$) and the number of Dehn twists in the chosen positive factorization.  However, we know all the minimal symplectic fillings of $\xi_{std}$ by Theorem~\ref{thm:filling}, namely the plumbing of spheres.  This filling has Euler characteristic $p$ and so $|\mathcal{F}| = p+3 = |\phi_{(2,1)}|$.  Comparing this with the previously calculated length gives us the desired contradiction, as $p-8 \neq p+3 \mod 10$. Thus $\phi_{(2,1)}$ has no positive Dehn twist factorization.
\end{proof}

\subsubsection{Relating the monodromy of the (2,1)-- and (2,2)--cables of a genus one open book decomposition.} \label{sec:stabilizing to 2,2}

The proof of Theorem~\ref{thm:counterexample} above relies on all the work done in Section~\ref{sec:monodromy} to factor the monodromy map of a $(2,1)$-cable of an open book. As those details are quite non-trivial, we will show out to make the proof independent of that work. The proof of Theorem~\ref{thm:counterexample} starts with a factorization of the monodromy of an open book decomposition that we claim supports the standard contact structure on a lens space $L(p,p-1)$. One may easily verify that the given open book does indeed describe the said lens space. Thus to make the proof independent of the rest of the paper we merely need to verify that the supported contact structure is the Stein fillable one as claimed. To this end we will show that by positively stabilizing the open book one time we can write the monodromy as a composition of positive Dehn twists. The precise stabilziation that we do is the one that takes the $(2,1)$--cable of the original genus one open book to the $(2,2)$--cable, which has a positive factorization. 

Specifically, we present a Hopf stabilization and sequence of mapping class relations changing the monodromy of a (2,1)--cable to a (2,2)--cable of any genus one open book.  The method here generalizes directly to any higher genus open book with one boundary component.  

First, Theorem \ref{thm:connected binding} gives a factorization of the monodromy of the (2,1)--cable of a genus one open book as 
\[\phi_{(2,1)} =  \rho_{(2,1)} \circ \tilde{\phi}=  (D_4 \circ D_5)^{-6} \circ (D_1 \circ D_2)^{-6} \circ (D_5) \circ (D_4 \circ D_5) \circ \cdots \circ (D_1 \circ \cdots \circ D_5) \circ \tilde{\phi},\]
 where we think of $\tilde{\phi}$ as $\phi$ acting on the right side nodule and hence having some factorization into Dehn twists $D_1$ and $D_2$.  
 \begin{figure}
	\vspace{10 pt}
	\labellist 
	\small\hair 2pt
		
	\pinlabel $d_1$ at 77 72
	\pinlabel $d_2$ at 35 69
	\pinlabel $d_3$ at 9 65
	\pinlabel $\bdry_{\Sigma}$ at 195 109
	\endlabellist
\includegraphics{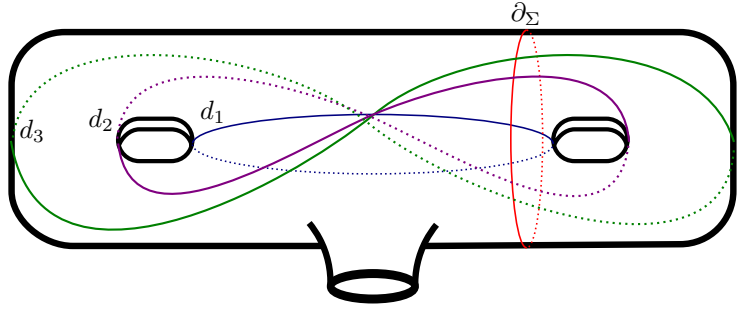}
\caption{The four Dehn twists presenting the cabling rotation for a (2,1)--cable.  We can write $\rho_{(2,1)} =   D_{d_3} \circ D_{d_2} \circ D_{d_1} \circ D^{-1}_{\bdry_{\Sigma}}$.} \label{fig:monodromy_base1}
\end{figure}
\begin{lem}
The diffeomorphism $\rho_{(2,1)}$ can be factored as 
\[
 \rho_{(2,1)} =D_{d_3} \circ  D_{d_2} \circ D_{d_1} \circ D^{-1}_{\bdry_{\Sigma}},
\]
where the curves $d_1,d_2,d_3$ and $\partial \Sigma$ are shown in Figure~\ref{fig:monodromy_base1}.
\end{lem}
\begin{proof}
We will derive the relation in the surface $\Sigma'$ of genus 2 with two boundary components and then cap off one boundary component. We can represent $\Sigma'$ as a 2--fold cover of the disk branched over 6 points. With the notation as in Lemma~\ref{lem:22factorization} we see that the original expression for $\rho_{(2,1)}$ comes from branch covering the braid $\Delta \Delta_1^{-4}\Delta_2^{-4}$ (recall that $\Delta$ is a half twists on all strands and $\Delta_1$ is a half twist on the first 3 strands and $\Delta_2$ is a half twist on the last 3 strands). Conjugating, we get the braid $\Delta\Delta_1^{-2}\Delta_2^{-6}$. From Lemma~\ref{lem:22factorization} we see that $\Delta\Delta_1^{-2}\Delta_2^{-2}= \sigma_{a_1} \sigma_{a_2}\sigma_{a_3}$. So $\Delta\Delta_1^{-2}\Delta_2^{-6}= \sigma_{a_1} \sigma_{a_2}\sigma_{a_3} \Delta_2^{-4}$. Lifting this relation to $\Sigma'$ (and capping off one boundary component) gives the desired factorization of $\rho_{(2,1)}$.
\end{proof}
 
We now stabilize as in Figure~\ref{fig:monodromy_base-stabilized} to get a monodromy factorization 
\[
D_{d_3} \circ D_{d_2} \circ D_{d_1} \circ D_\gamma \circ D^{-1}_{\bdry_{\Sigma}} \circ \tilde{\phi}
\]
 after commuting $D_\gamma$ past $\tilde{\phi}$ and $D^{-1}_{\bdry_\Sigma}$.  
 \begin{figure}
	\vspace{10 pt}
	\labellist
	\small\hair 2pt 
	\pinlabel $\gamma$ at 136 99
	\endlabellist
\includegraphics{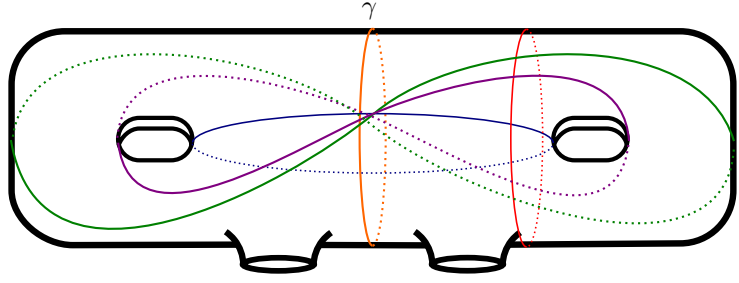}
\caption{The Dehn twists after adding another, $\gamma$, by stabilization.} \label{fig:monodromy_base-stabilized}
\end{figure}
 Now $\gamma$ and $d_1$ are two curves in the lantern relation shown in Figure~\ref{fig:monodromy_lantern1} 
  \begin{figure}
	\labellist
	\small\hair 2pt 
	\pinlabel $c_1$ at 54 99
	\pinlabel $c_2$ at 216	99
	\pinlabel $c_3$ at 94 33
	\pinlabel $c_4$ at 181 33
	\pinlabel $\beta$ at 157 50
	
	\endlabellist
\includegraphics{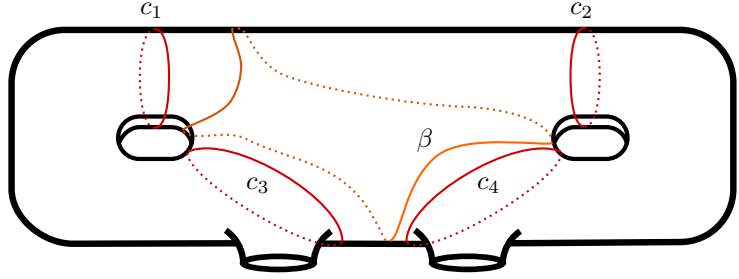}
\caption{The Dehn twists involved in the lantern with $\gamma$ and $d_1$.} \label{fig:monodromy_lantern1}
\end{figure}
 and we substitute $D_{c_4} \circ D_{c_3} \circ D_{c_2} \circ D_{c_1} \circ D^{-1}_{\beta}= D_{d_1} \circ D_\gamma$ to get 
 \[
 D_{d_3} \circ D_{d_2} \circ D_{c_4} \circ D_{c_3} \circ D_{c_2} \circ D_{c_1 \circ }D^{-1}_\beta  \circ D^{-1}_{\bdry\Sigma} \circ  \tilde{\phi}.
\]
 Now $\beta$ and $\bdry_\Sigma$ also form part of a lantern relation with the other five curves shown in Figure~\ref{fig:monodromy_lantern2}.  
 \begin{figure}
	\labellist
	\small\hair 2pt 
	\pinlabel $c'_{1}$ at 60 73
	\pinlabel $c'_{2}$ at 216	99
	\pinlabel $c'_{3}$ at 94 33
	\pinlabel $c'_{4}$ at 219 3
	\pinlabel $\delta_1$ at 188 59
	\endlabellist
\includegraphics{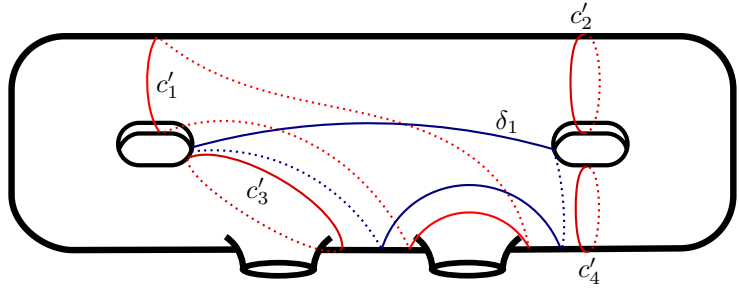}
\caption{The Dehn twists involved in the lantern with $\beta^{-1}$ and $\bdry_{\Sigma}^{-1}$.} \label{fig:monodromy_lantern2}
\end{figure}
 This lantern relation gives $D^{-1}_\beta \circ  D^{-1}_{\bdry \Sigma}=  D_{c'_{4}}^{-1} \circ D_{c'_{3}}^{-1} \circ D_{c'_{2}}^{-1} \circ D_{c'_{1}}^{-1} \circ D_{\delta_1}$ where $c_3 = c'_{3}$ and $c_2 = c'_{2}$.  Substituting gives the factorization 
 \[
D_{d_3} \circ  D_{d_2} \circ D_{c_4} \circ D_{c_3} \circ D_{c_2} \circ D_{c_1} \circ D_{c'_{4}}^{-1} \circ D_{c'_{3}}^{-1} \circ D_{c'_{2}}^{-1} \circ D_{c'_{1}}^{-1} \circ D_{\delta_1}  \circ \tilde{\phi} .
 \]
We cancel $D_{c_3}$ with  $D_{c'_{3}}^{-1}$ and $D_{c_2}$ with $D_{c'_{2}}^{-1}$ to get 
\[
   D_{d_3} \circ D_{d_2} \circ D_{c_4} \circ D_{c_1} \circ D_{c'_{4}}^{-1} \circ D_{c'_{1}}^{-1} \circ D_{\delta_1} \circ \tilde{\phi}.
\]

We will need to ``conjugate'' Dehn twists past one another. To this end, recall that for any diffeomorphism $f$ of a surface we have the relation $f \circ D_c \circ f^{-1}= D_{f(c)}$, so $f \circ D_c = D_{f(c)} \circ f$ and $D_c \circ f= f \circ D_{f^{-1}(c)}$. 
We would like the curves $d_2$ and $d_3$ to look more like $\delta_1$, that is to loop around the right most boundary component.  To this end we will conjugate $D_{d_2}$ and $D_{d_3}$ past $D_{c_4}$ and $D_{c'_{4}}^{-1}$ (after commuting $D_{c'_{4}}^{-1}$  and $D_{c_1}$ past each other) and get the factorization
\[
 D_{c_4} \circ  D_{c_{4'}}^{-1} \circ D_{\delta_3} \circ  D_{\delta_2} \circ  D_{c_1} \circ  D_{c_{1'}}^{-1} \circ  D_{\delta_1} \circ  \tilde{\phi} ,
\]
where $\delta_{2}$ and $\delta_{3}$ are, respectively, the curves $D_{c'_{4}} (D_{c_4}^{-1} (d_2))$ and $D_{c'_{4}} (D_{c_4}^{-1} (d_3))$ and are shown in Figure~\ref{fig:monodromy_twist1}.  Now for ease we want to slide the right hand boundary component up the back of the surface to the top.  This gives us the arrangements in Figure~\ref{fig:monodromy_untwist}.  Lastly, we conjugate $D_{c_1}$ and $D_{c'_{1}}^{-1}$ to the left across $D_{\delta_{2}}$ and $D_{\delta_{3}}$.  Conjugating $D_{c_1}$ across $D_{\delta_{3}} \circ D_{\delta_{2}}$ gives $D_{D_{\delta_3} (D_{\delta_2} (c_1))} = D_{c'_{4}}$ and similarly $D_{c'_{1}}$ transforms into $D_{D_{\delta_3} (D_{\delta_2} (c'_{1}))} = D_{c_{4}}$. The extra Dehn twists then cancel, leaving our final factorization 
\[D_{\delta_3} \circ   D_{\delta_2} \circ   D_{\delta_1} \circ \tilde{\phi}  ,\]
as shown at the top of Figure~\ref{fig:monodromy_untwist}.  This is the factorization handed to us from the branched cover construction in in Proposition~\ref{prop:22cable}, but in particular the factorization is in terms of positive Dehn twists. Thus verifying that the supported contact structure is Stein fillable.

\begin{figure}
	\labellist
	\small\hair 2pt 
	\pinlabel $\delta_2$ at 241 52
	\pinlabel $\delta_3$ at 265 52
	\endlabellist
\includegraphics{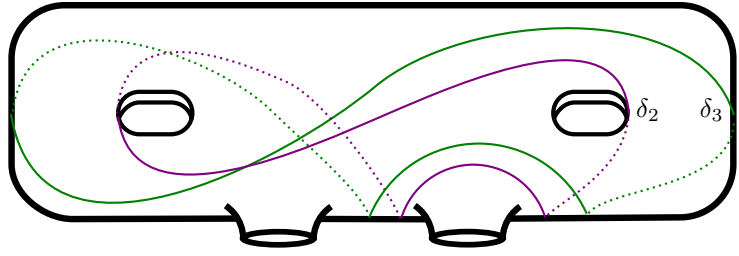}
\caption{The Dehn twists $\delta_2$ and $\delta_3$ are the images of $d_2$ and $d_3$ under $D_{c_{4}}^{-1} \circ D_{c_4'}$.} \label{fig:monodromy_twist1}
\end{figure}

\begin{figure}
	\labellist
	\small\hair 2pt 
	\pinlabel $\delta_1$ at 196 63
	\pinlabel $\delta_2$ at 241 52
	\pinlabel $\delta_3$ at 265 52
	\endlabellist
\includegraphics{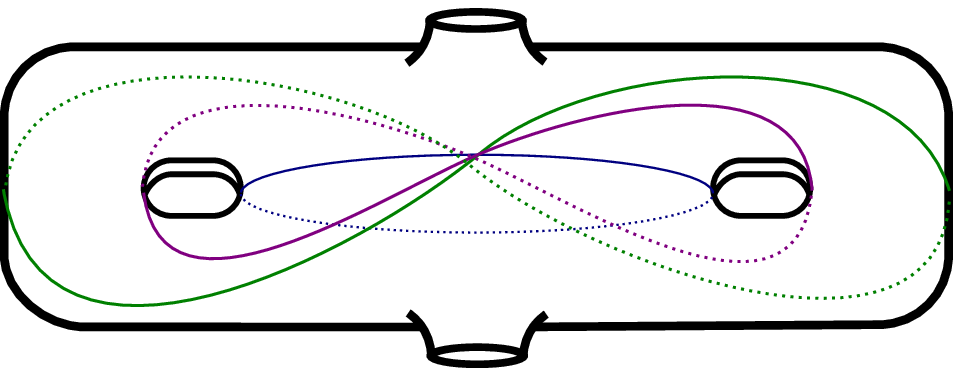}

\vspace{10pt}
	\labellist
	\small\hair 2pt 
	\pinlabel $c_{1}$ at 56 98
	\pinlabel $c_{1'}$ at 80 38
	\pinlabel $c_{4}$ at 190 39
	\pinlabel $c_{4'}$ at 219 5
	\endlabellist
\includegraphics{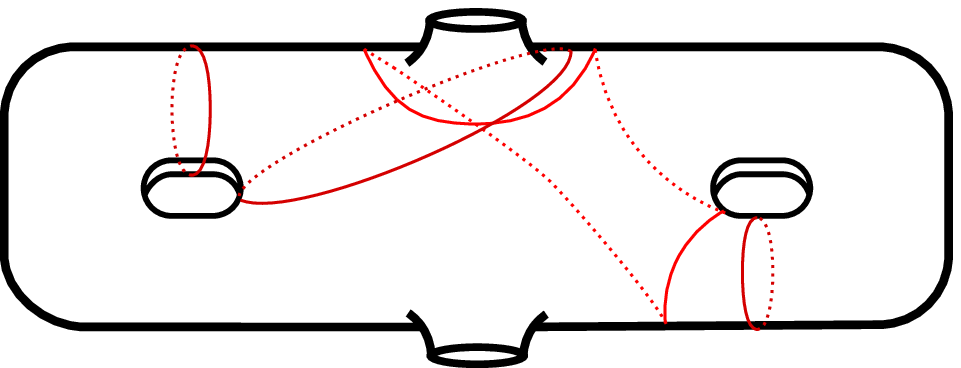}

\caption{The Dehn twists in the factorization of the (2,2)--cable after sliding the boundary component on the right up the back of the surface to the top.  The top diagram shows the curves $\delta_1$, $\delta_2$ and $\delta_3$, which give the final monodromy.} \label{fig:monodromy_untwist}
\end{figure}

\subsection{Monoids in the mapping class group}
In this subsection we establish the Stein cobordism from $(M_{(\Sigma,\phi_1)},\xi_{(\Sigma,\phi_1)}) \sqcup (M_{(\Sigma,\phi_2)},\xi_{(\Sigma,\phi_2)}) $ to $(M_{(\Sigma,\phi_2\circ \phi_1)},\xi_{(\Sigma,\phi_2\circ\phi_1)})$ and hence all the corollaries of this fact discussed in the introduction.

\begin{proof}[Proof of Theorem~\ref{constructbordism}]
The theorem follows by finding a very particular diffeomorphism  $\rho$ on $\Sigma'$, a surface related to $\Sigma$, as follows:
\begin{enumerate}
\item{$\Sigma'$ is built from two disjoint special submanifolds $\Sigma_1$ and $\Sigma_2$, each diffeomorphic to $\Sigma$, by adding 2-dimensional 1-handles,}
\item{ $\rho$ maps $\Sigma_1$ to $\Sigma_2$ and $\Sigma_2$ to $\Sigma_1$ by the identity map under their identifications with $\Sigma$,}
\item{ $\rho$ has a factorization into positive Dehn twists, and}
\item{ the pair $\left(\Sigma', \rho\right)$ is obtained from $(\Sigma_1, \Id)$ by a sequence of positive Hopf stabilizations (and no destabilizations).} 
\end{enumerate}

\noindent We begin by constructing the desired Stein bordism given the pair $(\Sigma', \rho)$.  The pair's existence is established in the last paragraph.    

We begin with an outline of the procedure in the language of open book decompositions before translating this procedure into a Stein cobordism.  We have two open books $(\Sigma, \phi_1)$ and $(\Sigma, \phi_2)$.  By adding 1-handles, we can stick them side-by-side as the subsurfaces $\Sigma_1$ and $\Sigma_2$ of $\Sigma'$, with the monodromy of the resulting open book being $\phi_i$ acting on $\Sigma_i$ and (extended as the identity everywhere else).  We will use this terminology later, so for now let's call these extensions $\widetilde{\phi_i}$.  Now by adding right-handed Dehn twists we can insert the map $\rho$ to the monodromy.  The entire result now follows from the observation that this new open book (the monodromy has been reordered for convenience $(\Sigma', \widetilde{\phi_2} \circ \rho \circ \widetilde{\phi_1})$ is a positive Hopf stabilization of $(\Sigma, \phi_2 \circ \phi_1)$.  To see this, notice that by conjugating $\widetilde{\phi_2}$ past $\rho$, we make $\phi_2$ act on the surface $\Sigma_1$ rather than $\Sigma_2$ and rewrite the monodromy as $\rho \circ \widetilde{\phi_2 \circ \phi_1}$, where both $\phi_1 $ and $\phi_2$ are now acting (as a composition) on $\Sigma_1$.  Since all the destabilizing arcs for the Dehn twists in $\rho$ sit on $\Sigma_2$, this is a positive Hopf stabilization of ($\Sigma_1, \phi_2 \circ \phi_1$).

The procedure for turning the above steps (adding 1-handles and adding positive Dehn twists) into a Stein cobordism is standard, but we include the details for those who might be unfamiliar with the setup.

Let $W_0=(M_{(\Sigma,\phi_1)}\sqcup M_{(\Sigma,\phi_2)})\times [0,1]$ with the symplectic structure on the symplectization of $\xi_{(\Sigma,\phi_1)}\sqcup \xi_{(\Sigma,\phi_2)}$ restricted to it. (It is well known that projection to $[0,1]$ is pluri-subharmonic and thus $W_0$ is Stein.) One may attach 1-handles to $(M_{(\Sigma,\phi_1)}\sqcup M_{(\Sigma,\phi_2)})\times 1$ so that the Stein structure extends over the 1-handles \cite{Eliashberg90a} and so that the new boundary component is the contact (self) connected sum of $(M_{(\Sigma,\phi_1)}\sqcup M_{(\Sigma,\phi_2)}, \xi_{(\Sigma,\phi_1)}\sqcup \xi_{(\Sigma,\phi_2)})$ (i.e., we may end up with the contact connected sum of additional copies of $S^1\times S^2$).  Thus the contact structure on this new boundary component is supported by the one obtained by attaching 1-handles to the pages of the open book decomposition for the lower boundary component. More importantly, any 2-dimensional 1-handle attachment to the pages of an open book is reflected in such a cobordism.  Thus we may construct a Stein cobordism $W_1$ with concave boundary  $(M_{(\Sigma,\phi_1)},\xi_{(\Sigma,\phi_1)}) \sqcup (M_{(\Sigma,\phi_2)},\xi_{(\Sigma,\phi_2)})$ and convex boundary supported by the open book decomposition with page $\Sigma'$ and monodromy $\widetilde{\phi_2}\circ\widetilde{\phi_1}$ where $\widetilde{\phi_i}$ is the extension of $\phi_i$ to $\Sigma'$ and $\phi_i$ acts on the subsurface $\Sigma_i$.

Again following  \cite{Eliashberg90a, Gompf98} one can attach 2-handles to the convex boundary of a Stein cobordism along a curve in the page of an open book decomposition (with framing one less than the framing the page induces on the curve) and extend the Stein structure over the 2-handle.  The monodromy of the open book decomposition on the convex boundary changes by composing with a positive Dehn twist along the curve \cite{EtnyreHonda02a}. Thus we can attach Stein 2-handles to the convex boundary of $W_1$ along curves in the pages of  $(\Sigma',\widetilde{\phi_2}\circ\widetilde{\phi_1})$ with framing one less than the page framing to get the Stein cobordism $W_2$ and the open book decomposition on the convex boundary is now $(\Sigma', \widetilde{\phi_2}\circ\rho\circ\widetilde{\phi_1}).$ 

Conjugating $\widetilde{\phi_2}$ past $\rho$ we get an open book $(\Sigma', \rho \circ (\rho^{-1}\circ\widetilde{\phi_2}\circ\rho)\circ \widetilde{\phi_1}).$ The map $\rho^{-1}\circ \widetilde{\phi_2}\circ \rho$ is the diffeomorphism where $\phi_2$ acts on $\Sigma_1$ and is the identity everywhere else. 
Thus if we denote by $\widetilde{\phi_2\circ\phi_1}$ the diffeomorphism of $\Sigma'$ that is $\phi_2\circ \phi_1$ on $\Sigma_1$ and the identity elsewhere, then the convex boundary component of the Stein cobordism $W_2$ is supported by the open book $(\Sigma',\rho \circ \widetilde{\phi_2\circ\phi_1}).$ Since $\left(\Sigma', \rho\right)$ is obtained from $(\Sigma_1, \Id)$ by a sequence of stabilizations we see that $(\Sigma',\rho \circ \widetilde{\phi_2\circ\phi_1})$ is also obtained from $(\Sigma_1,\phi_2\circ\phi_1)$ by positive stabilizations. Thus the convex boundary of $W_2$ is supported by $(\Sigma,\phi_2\circ\phi_1)$ and $W_2$ is the desired Stein cobordism. 

If $\Sigma$ has more than one boundary component then $\Sigma'$ is the cable surface $\Sigma_{(2,{\bf 1})}$ from Theorem~\ref{thm:stabilization} and similarly $\rho$ is $\rho_{(2,{\bf 1})}(\Sigma)$ from the same theorem. If $\Sigma$ has connected binding then $\Sigma'=\Sigma_{(2,2)}$ and $\rho=\rho_{(2,2)}$ from Theorem~\ref{thm:stabilization}. The properties listed above are clear from Theorem~\ref{thm:stabilization}  and Proposition~\ref{prop:monodromy split}.
\end{proof}

\subsection{Negative cables which remain tight.}\label{exceptionalsurg}

As we saw in Section \ref{sec:resolution}, $(r,-1)$--rational open books have particularly nice resolutions and abstract presentations.  In this section we take $(r-1, -1)$--cables of $(r,-1)$--open books with connected binding and see that they are again $(r,-1)$--open books.  We show how to determine the abstract presentation of the cable given that of the pattern.  As a sample application we show that there are many cases where the $(2,-1)$--cable of a $(3,-1)$--open book is a negative cable which is still tight. 

\begin{prop} \label{prop:negative cable} Let $\mathfrak{B}$ be a $(r,-1)$--open book with connected binding, written abstractly as $(\Sigma, \delta_{\frac{1}{r}} \circ \phi)$, $\phi \in Map^+(\Sigma,\bdry \Sigma)$.  The $(r-1, -1)$--cable of $\mathfrak{B}$ can be written abstractly as $(\Sigma_{(r-1, 1)}, \delta_\frac{1}{r} \circ \rho_{(r-1,1)}^{-1} \circ \bdry_1^{(2-r)}\circ \tilde{\phi})$. Here the surface $\Sigma_{(r-1,1)}$ and diffeomorphism $ \rho_{(r-1,1)}$ are as described for integral cables in Section \ref{sec:connected binding}.  The diffeomorphism  $\bdry_1$ is a Dehn twist about a the boundary of the first nodule of $\Sigma_{({r-1},-1)}.$ (Note that $\delta_\frac{1}{r}$ always refers to a $\frac 1r$ fractional Dehn twist along the boundary of the surface under consideration.  The two occurrences of this notation above refer to diffeomorphisms on different surfaces.)
\end{prop}

Before we prove the proposition, we discuss our primary interest in these cables.  Since the $(r-1,-1)$--cable is still, as a rational open book, an $(r, -1) $--open book, it has a particularly nice resolution.  We use this description to find examples where this negative cable is still tight.  These examples generalize to any $r$ with only a small modification of the proof.  

\begin{cor} The $(2,-1)$--cable of the rational open book $(\Sigma, \delta_{\frac{1}{3}}\circ \bdry^2)$ obtained by $-3$ surgery on the binding of $(\Sigma, \bdry^2)$ (framed to be a $(3,-1)$--open book) has an integral resolution with positive monodromy and hence is tight.  The monodromy $\bdry^2$ is a composition of two right-handed Dehn twists about a curve parallel to the boundary of $\Sigma$.
\end{cor}

Notice that this corollary is a more precise formulation of Proposition~\ref{prop:otexceptional}. To prove the corollary, we need a relation in the planar mapping class group proved in the lemma below.

\begin{figure}[htp]
\centering
\mbox{
	\labellist 
	\small\hair 2pt
	\pinlabel $\bdry_2$ at 53 173 
	\pinlabel $\bdry_1$ at 77 130
	\pinlabel $\delta_1$ at 117 176 
	\pinlabel $\delta_2$ at 142 121
	\pinlabel $\delta_3$ at 114 56 
	\endlabellist
\includegraphics[width=2.25in]{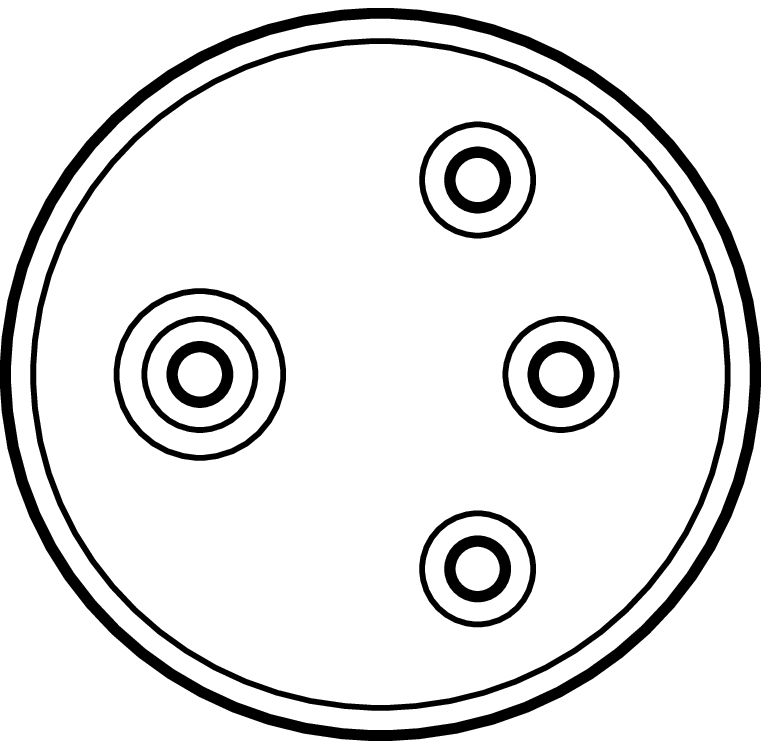}
}
\mbox{
	\labellist
	\small\hair 2pt
	\pinlabel $D_1$ at 72 158
	\pinlabel $D_2$ at 161 136
	\pinlabel $D_3$ at 71 51
	\pinlabel $D_\partial$ at 200 111
	\endlabellist
\includegraphics[width=2.25in]{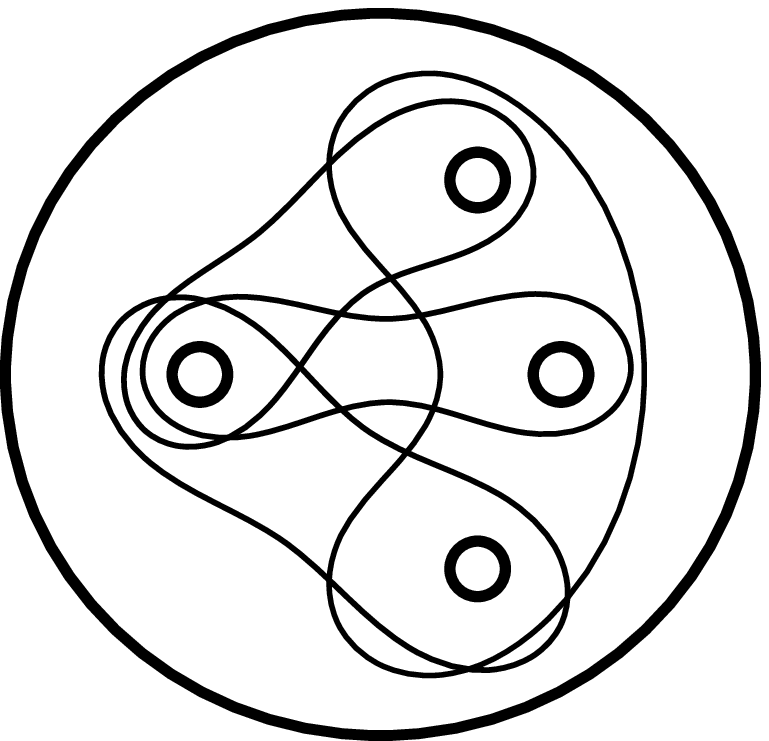}
}
\caption{In the mapping class group of a five-holed sphere the two collections of Dehn twists (left and right diagrams) compose to give the same diffeomorphism.} \label{fig:genlantern}

\end{figure}

\begin{proof}[Proof of Corollary] By Proposition \ref{prop:negative cable}, the negative cable of $(\Sigma, \delta_{\frac{1}{3}} \circ \bdry^2)$ has an abstract description $(\Sigma_{(2,1)}, \delta_\frac{1}{3} \circ \rho_{(2,1)}^{-1} \circ \bdry_1^{-1} \circ \bdry_1^2)$.  We can resolve this open book as shown in Section~\ref{sec:resolution} by adding two more boundary components to $\Sigma_{(2,1)}$ to get a surface $\widetilde{\Sigma}_{(2,1)}$ and replacing $\delta_\frac{1}{3}$ with $M_\bdry$, the Dehn multitwist about the three boundary components, to get the new monodromy $\tilde{\phi} = M_{\delta} \circ \rho_{(2,-1)}^{-1} \circ \bdry_1^{-1}\circ \bdry_1^2$.  Subsection~\ref{sec:connected binding} gives a factorization of the cable rotation $\rho_{(2,1)}$ as $\rho_{(2,1)} = \Delta \circ \bdry_2^{-1} \circ \bdry_1^{-1}$ with  $\Delta$ is the ``Garside twist.''  We need two properties of $\Delta$: 1) that $\Delta^2 = D_\bdry$, where $D_{\bdry}$ is a Dehn twist along $\bdry \Sigma_{(2,1)}$ thought of as a subset of , and 2) that $\Delta$ can be written as a product of positive Dehn twists.  Thus we can rewrite the monodromy as $\tilde{\phi} = M_\delta \circ \Delta^{-1} \circ \bdry_2\circ \bdry_1^2$.  To see that $\tilde{\phi}$ has positive Dehn twist factorization, we apply a generalization of the lantern relation to the five-holed sphere $S$, the complement of the nodules in $\widetilde{\Sigma}_{(2,1)}$. ($S$ is the base component of the cable surface in the terminology of Section~\ref{sec:monodromy}.  The boundary components of $S$ split into $\delta_1, \delta_2, \delta_3$, the boundary components of $\widetilde{\Sigma}_{(2,1)}$, and $\bdry_1, \bdry_2$, the boundaries of each of the two nodules, and these determine the Dehn twists $M_\delta$, $\bdry_1$ and $\bdry_2$.  We denote by $\delta_1,\delta_2,$ and $ \delta_3$ the Dehn twists about curves parallel to the boundary components of $\widetilde{\Sigma}_{(2,1)}.$ Thus $M_\delta=\delta_1\circ\delta_2\circ \delta_3.$  Applying the lantern relation of Lemma \ref{lem:lantern} we can rewrite part of the monodromy; $M_\delta \circ \bdry_1^2 \circ \bdry_2 = D_{\bdry} \circ D_3 \circ D_2 \circ D_1 $.   Since we can write $D_{\bdry} \circ \Delta^{-1}=\Delta$ as a product of positive Dehn twists, we can also write $\tilde{\phi}$ as a product of positive Dehn twists; $\tilde{\phi} = \Delta \circ D_3 \circ D_2 \circ D_1$, showing that the $(2,-1)$--cable of $(\Sigma, \delta_{\frac{1}{3}} \bdry^2)$ supports a Stein fillable contact structure.
\end{proof}

\begin{lemma}[Endo, Mark and Van Horn-Morris 2010, \cite{MarkVanhorn}]\label{lem:lantern} In the mapping class group of a five-holed sphere, the two factorizations shown in Figure \ref{fig:genlantern} give the same diffeomorphism.  Specifically $\bdry_{1}^2 \circ \bdry_{2} \circ \delta_1 \circ \delta_2 \circ \delta_3 = D_\partial \circ D_3 \circ D_2  \circ D_1 $.
\end{lemma}

\begin{proof}  Apply the original lantern relation twice, the first time about the sphere with four punctures containing $\bdry_1, \delta_1$ and $\delta_2$ (and its fourth boundary component a curve separating these curves from $\delta_3$ and $\bdry_2$). The second time application of the lantern relation will be about the sphere with four punctures containing $\bdry_1,\bdry_2$ and $\delta_3.$
\end{proof}

\begin{proof}[Proof of Proposition \ref{prop:negative cable}]  The proof of the proposition rests on a few simple facts about Dehn surgery.  The proof holds for disconnected binding, although the description of the monodromy is more complicated as we need to cable a single binding component and so we will not discuss this here.  First, any $(r,-1)$--open book comes from an integral open book by $-r$ surgery on the binding.  We translate $-r$ surgery on a knot $K$ into surgery on a link in a neighborhood of $K$ composed of a $0$--cable $K_0$ and the $(r-1, -1)$--cable $K_{(r-1,-1)}$ which sits on a torus nested inside the torus used for $K_0$.  The surgery coefficient is $\frac{1}{r-2}$ for $K_0$ and $-r$ for $K_{(r-1, -1)}$.  We will show that the diffeomorphism from the $-r$ surgery on $K$ to the sequence of surgeries on $K_0$ and $K_{(r-1,-1)}$ is supported on the framed solid torus and that it takes the the $(r-1,-1)$--cable of the image of $K$ to the image of $K_{(r-1,-1)}$.  First though, we show why this is enough to prove the proposition.

Once we have the surgery description, it is quite easy now to prove the proposition. Let us fix notation and denote by $K'$ the image of $K$ under $-r$--surgery along $K$ ($r>0$).  Let $K_{r-1,-1}$ be the $(r-1,-1)$--cable of $K$ and $K'_{r-1,-1}$ be that of $K'$, framed to be an $(r,-1)$ open book.  Let $(\Sigma, \phi)$ be the open book supported by $K$, which gives the rational open book with binding $K'$ the abstract presentation $(\Sigma, \delta_\frac{-1}{r} \circ \phi)$. (This is our original open book $\mathfrak{B}$.)  We construct the fibration on $K'_{r-1,-1}$ by first cabling $K$ to get $K_{r-1,-1}$, then applying the $-r$--surgery on $K_{r-1,-1}$ and the $\frac{1}{r-2}$ surgery along $K_0$.  As in Section \ref{sec:monodromy}, cabling changes the open book $(\Sigma, \phi)$ to $(\Sigma_{({r-1},1)}, \rho_{(r,-1)}\circ\phi)$. The component $K_0$ sits naturally on a page of the cabled fibration as the boundary of a nodule with page framing $0$ so $\frac{1}{r-2}$--surgery along $K_0$ adds $r-2$ left handed Dehn twists along $K_0$.  These Dehn twists we denoted as $\bdry_1^{2-r}$ in the statement of the proposition.  Last, we do the $-r$--surgery along $K_{(r-1,-1)}$ which also doesn't change the surface but adds the fractional boundary twist $\delta_\frac{1}{r}$ to the factorization of the monodromy, completing our monodromy to $\delta_{\frac{1}{r}} \circ \rho_{(r-1,-1)} \circ \bdry_1^{2-r} \circ \phi$.

Now we prove the surgery statement.  Recall $K'$ is the binding of a rational $(r,-1) $--open book, built as the image of a knot $K$ in an integral open book under $-r$--surgery.  We want to construct the $(r-1,-1)$--cable of $K'$ which we denote as $K'_{(r-1,-1)}$.  We describe the surgery along $K$ by removing a neighborhood $\nu(K)$ from our ambient manifold $M$ and regluing by a map $f:\partial \nu (K) \rightarrow -\bdry(M\backslash \nu(K))$.  Choose a meridian-longitude basis $\mu, \lambda$ for the boundary torus $T$ so that $\lambda$ is the page framing of $K$, oriented parallel to $K$, and $\mu$ is the meridian, oriented so that it links $K$ positively.  (Alternatively the oriented intersection $\mu \cdot \lambda$ on $T$ is positive.)  With respect to this basis, we choose the regluing map to be given by a matrix $A = \left(\begin{array}{cc} r & 1 \\ -1 & 0\end{array} \right)$, which picks out the desired framing curve, making the rational longitude of $K'$, $A^{-1}(0,1)^t = (-1, r)^t$.  In other words, with this framing, the resulting rational open book is exhibited as a $(r, -1) $--open book. 

We can decompose $A$ another way, however, which gives the second surgery description.  Since the framing of the cable $K_{(r-1,-1)}$ coming from the cabling torus $T$ is $-r+1$, $-r$--surgery on $K_{(r-1,-1)}$ changes the gluing along $T$ by a left-handed Dehn twist along $K_{(r-1,-1)}$.  This send $\mu$ to $\left<K_{(r-1,-1)}, \mu\right>  K_{(r-1,-1)} + \mu$ and $\lambda$ to $\left<K_{(r-1,-1)}, \lambda\right>  K_{(r-1,-1)} + \lambda$.  In coordinates, this gives a gluing matrix $\left(\begin{array}{cc} r & 1 \\ -(r-1)^2 & -r+2 \end{array} \right)$.  We need to compose this with the surgery along $K_0$, which we again think of as $r-2$ left-handed Dehn twists along $\lambda$.  This has the gluing matrix $\left(\begin{array}{cc} 1 & 0 \\ r-2 & 1 \end{array} \right)$.  Composing these two matrices gives $A$, proving the surgery equivalence.  This also shows that this describes the fibration on $K'_{(r-1,-1)}$.  Both descriptions give the same gluing matrix and in each description the cable knot is the $(r-1,-1)$--cable of the core. 
\end{proof}

\def\cprime{$'$} \def\cprime{$'$}

\end{document}